\documentclass[a4paper]{amsart}

\usepackage{esint}
\usepackage[abbrev,backrefs]{amsrefs}
\usepackage{amssymb}
\usepackage{stmaryrd}
\usepackage{mathrsfs}
\usepackage{bbm}
\usepackage[svgnames]{xcolor}
\usepackage{enumitem}
\usepackage{amsthm}

\usepackage{tikz}

\usepackage{todonotes}

\usepackage[hidelinks]{hyperref}
\usepackage[nameinlink,capitalize]{cleveref}
\usepackage{titletoc}

\usepackage{quiver}

\crefname{enumi}{part}{parts}
\Crefname{enumi}{Part}{Parts}

\newtheorem{theorem}{Theorem}[section]
\newtheorem{lemma}[theorem]{Lemma}
\newtheorem{proposition}[theorem]{Proposition}
\newtheorem{corollary}[theorem]{Corollary}

\theoremstyle{definition}
\newtheorem{definition}[theorem]{Definition}

\theoremstyle{remark}
\newtheorem{remark}[theorem]{Remark}

\def\restrict#1{\raise-.5ex\hbox{\ensuremath|}_{#1}}

\newcommand{\mcal}[1]{{\mathcal{#1}}}
\newcommand{\Torus}{{\mathbb{T}}}

\DeclareMathOperator{\Div}{div}
\DeclareMathOperator{\curl}{curl}

\newcommand{\set}[1]{\left\{#1\right\}}
\newcommand{\abs}[1]{\left\vert#1\right\vert}
\newcommand{\norm}[1]{\left\Vert#1\right\Vert}
\newcommand{\floor}[1]{\left\lfloor#1\right\rfloor}
\newcommand{\vertiii}[1]{{\left\vert\kern-0.25ex\left\vert\kern-0.25ex\left\vert #1
		\right\vert\kern-0.25ex\right\vert\kern-0.25ex\right\vert}}

\newcommand{\shortset}[1]{\{#1\}}

\newcommand{\bignorm}[1]{\big\Vert#1\big\Vert}

\newcommand{\angles}[1]{\left\langle#1\right\rangle}

\newcommand{\Angles}[1]{\llangle#1\rrangle}

\makeatletter
\newsavebox{\@brx}
\newcommand{\llangle}[1][]{\savebox{\@brx}{\(\m@th{#1\langle}\)}%
  \mathopen{\copy\@brx\kern-0.5\wd\@brx\usebox{\@brx}}}
\newcommand{\rrangle}[1][]{\savebox{\@brx}{\(\m@th{#1\rangle}\)}%
  \mathclose{\copy\@brx\kern-0.5\wd\@brx\usebox{\@brx}}}
\makeatother

\newenvironment{theorembis}[1]
  {\renewcommand{\thetheorem}{\ref*{#1}$'$}
    \renewcommand{\theHtheorem}{theorembis.#1}
   \addtocounter{theorem}{-1}%
   \begin{theorem}}
  {\end{theorem}}

\let\decreasesto\searrow

\newcommand{\indicator}[1]{\mathbbm{1}_{\set{#1}}}
\newcommand{\indicatorofset}[1]{\mathbbm{1}_{#1}}

\newcommand{\ocinterval}[1]{\left(#1\right]} 
\newcommand{\cointerval}[1]{\left[#1\right)} 

\let\union\cup
\let\intersect\cap

\let\boundary\partial

\let\phi\varphi

\newcommand{\R}{\mathbb{R}}

\newcommand{\manifold}{\mathcal{M}}
\newcommand{\length}{\ell}
\newcommand{\totalmass}{\mathbb{M}}

\newcommand{\Dim}{n}
\newcommand{\dd}{\mathrm{d}} 
\newcommand{\ed}{\mathrm{d}} 

\newcommand{\maxr}{L} 

\newcommand{\Morrey}[1]{\norm{#1}_{\mathrm{Morrey}}}

\newcommand{\currents}{\mathbf{M}}

\newcommand{\InterpolationLemma}{\hyperref[Interpolation]{Interpolation \Cref*{Interpolation}}}

\numberwithin{equation}{section}

\numberwithin{equation}{section}

\begin{document}

\title[Potential Estimates and Hodge systems on Compact Manifolds]{Potential Estimates and Hodge Systems with $L^1$ data on compact manifolds}


\author{Jesse Goodman}
\address{Department of Statistics, University of Auckland, Private Bag 92019, Auckland 1142, New Zealand}
\email[J. Goodman]{jesse.goodman@auckland.ac.nz}
\thanks{}

\author{Felipe Hern\'andez}
\address{Department of Mathematics, Penn State University, State College, PA, USA}
\email{felipeh@psu.edu}

\author{Daniel Spector}
\address{Daniel Spector\hfill\break\indent
Department of Mathematics\hfill\break\indent
National Taiwan Normal University\hfill\break\indent
No. 88, Section 4, Tingzhou Road, Wenshan District, Taipei City, Taiwan 116, R.O.C.\hfill\break\indent
and \hfill\break\indent
National Center for Theoretical Sciences\hfill\break\indent
No. 1 Sec. 4 Roosevelt Rd., National Taiwan
University, Taipei, 106, Taiwan\hfill\break\indent
and\hfill\break\indent
Department of Mathematics\hfill\break\indent
University of Pittsburgh\hfill\break\indent
Pittsburgh, PA 15261 USA
}
\email[D. Spector]{spectda@gapps.ntnu.edu.tw}



\date{}

\dedicatory{}

\commby{}

\begin{abstract}
In this paper we establish optimal Lorentz estimates for the Riesz potentials acting on closed or co-closed $k$-forms of finite mass on a smooth, compact Riemannian manifold of dimension $\Dim$:
For $\alpha \in (0,\Dim)$ and $k=1,\ldots,\Dim-1$, there exists a constant $C>0$ such that
  \begin{align*}
    \|  \mathcal{I}_{\alpha,k} F \|_{L^{\Dim/(\Dim-\alpha),1}(\Lambda^k)} \leq C \| F\|_{L^1(\Lambda^k)}
  \end{align*}
for all $k$-forms $F \in L^1(\Lambda^k)$ orthogonal to the space of harmonic $k$-forms and satisfying $\ed F=0$ or $\ed^* F=0$.
We show how this inequality implies analogous Lorentz bounds for solutions of the $k$-form Poisson equation and for the Hodge system with data having finite mass.
These results include as a special case the div--curl system on the $3$-dimensional torus, where we answer an open question originally posed by J.\ Bourgain and H.\ Brezis.
\end{abstract}

\maketitle

\setcounter{tocdepth}{3}
\tableofcontents

\section{Introduction}

\subsection{Motivation and overview}
In this paper we consider the Hodge system for $k$-forms $Z$ on a Riemannian manifold
$\mcal{M}$,
\begin{align}
  \begin{split}
    \label{eq:Hodge}
    \ed_k^* Z &= F, \\
    \ed_k Z &= G,
  \end{split}
\end{align}
where $F, G\in L^1$ satisfy suitable orthogonality and compatibility conditions.   When $F, G\in L^p$ for $p>1$, estimates for solutions of \eqref{eq:Hodge} which are analogous to the usual Sobolev inequality follow easily from the $L^p$ extension of Gaffney's inequality \cite[Theorem 4.11 on p.~59]{ISS} and the Sobolev embedding \cite[Theorem 3.5 on p.~23]{Hebey}.  The endpoint $p=1$ has so far proved intractable by such methods, as Gaffney's inequality fails in this case, while the use of known harmonic analysis techniques, e.g.\ \cite{Stein1970}, yields only weak-type estimates.  
Among the main results of this paper are novel and sharp estimates to this system for $F, G\in L^1$.  

A special case of \eqref{eq:Hodge} is to take $\manifold=\Torus^3$, $k=2$, and $G=0$, in which case the Hodge system reduces to a $\Div$--$\curl$ system for vector fields:
\begin{equation}
\begin{split}
\label{eq:Hodge_Torus}
  \curl Z &= F
  , \\
  \Div Z &= 0
  .
\end{split}
\end{equation}
These are Maxwell's equations in the static regime, where $Z$ models the magnetic field induced by a given electric current density $F$ subject to periodic boundary conditions.
In this model it is natural to assume $F\in L^1(\Torus^3;\R^3)$ and no more, because any regularization $F$ of the current carried by a thin wire of finite length stays bounded in $L^1$ and no better. 
In the Euclidean setting, Bourgain and Brezis proved that the solution of \eqref{eq:Hodge_Torus} satisfies the estimate
\begin{align}\label{BB_32}
  \|Z\|_{L^{3/2}} \leq C \|F\|_{L^1}
\end{align}
and asked~\cite{BourgainBrezis2007}*{Open Question~1 on p.~295} whether $Z$ belongs to the critical Lorentz space $L^{3/2,1}$.
This is the optimal result on this scale, and would match the known result for the gradient \cite{alvino}.
As we discuss further in \Cref{sss:Lorentzestimates}, Lorentz spaces in general, and Lorentz spaces of the form $L^{r,1}$ in particular, are natural objects of study: they arise directly from real interpolation of the couple $(L^1,L^\infty)$ \cite[Theorem 5.2.1]{BerghLofstrom1976} and have many useful applications, for instance in PDEs \cite{Danshen, KeelTao1998,Helein2002}. 

In this paper, we obtain optimal Lorentz and Besov-Lorentz bounds for the general Hodge system~\eqref{eq:Hodge} on compact Riemannian manifolds for $1\leq k\leq \Dim-1$ (see \Cref{bbq}, \Cref{BesovClosedCoclosed}, \Cref{BesovClosedCoclosedReformulated}).
As a consequence, we give an affirmative answer to the question of Bourgain and Brezis (see \Cref{bbq_t}).
As in~\cites{HS,HRS}, the key step is an optimal Lorentz space estimate for Riesz potentials, in this setting formulated in the space of differential forms (see \Cref{RieszClosedCoclosed_k}).  
From this more general estimate, we also obtain a result for solutions to the Poisson equation $\Delta_kZ=F$ on forms (see \Cref{Hodge_Laplacian_L1}).

A key technical innovation in this paper allows us to reduce fractional integration estimates for $k$-forms to those of $(n-1)$-forms.  
This is an analogue of the reduction found by Van Schaftingen~\cites{VS2,VS3} relating co-cancelling operators to divergence-like operators.
On a manifold, however, there are no global linear maps for such a reduction, and therefore this requires an adaptation of Van Schaftingen's algebraic argument.
The proof of the Lorentz space fractional integration estimate in the case of $(\Dim-1)$-forms follows a streamlined version of the argument in \cite{HRS}.  
A key difference is that we use a more general atomic decomposition for closed currents established in~\cite{ChenGoodmanHernandezSpectorPreprint}.

\subsection{Notation and main results}\label{ss:NotationAndMainResults}

To state our main result we briefly introduce some requisite notation, while we refer the reader to \Cref{s:Preliminaries} for a full accounting.
Throughout the paper we take $\manifold$ to be a compact $C^\infty$-smooth Riemannian manifold of dimension $\Dim\geq 2$.
For $k=0,\ldots,\Dim$, we write $C^\infty(\Lambda^k)$ to denote the space of smooth $k$-forms and
\begin{align*}
  \ed_k &:C^\infty(\Lambda^k) \to C^\infty(\Lambda^{k+1}) \\
   \ed_k^* &:C^\infty(\Lambda^k) \to C^\infty(\Lambda^{k-1})
\end{align*}
the exterior differential and co-differential, respectively.
For $\omega \in C^\infty(\Lambda^k)$ we define\footnote{Note that this definition is opposite to the geometer's sign convention.} the Hodge Laplacian by
\begin{equation}\label{HodgeLaplacian}
  -\Delta_k\omega = \ed^*_{k+1} \ed_k \omega + \ed_{k-1} \ed^*_k \omega.
\end{equation}
The operators $\ed,\ed^*,\Delta$ can be considered to act on forms of arbitrary degree, but for definiteness we will retain the subscript $k$.
We say that $\omega$ is harmonic and write $\omega \in \mathcal{H}(\Lambda^k)$ if $\Delta_k\omega=0$.
We write $e^{t\Delta_k}$ to denote the heat propagator on $k$-forms on $\manifold$, i.e.\ $e^{t\Delta_k}$ satisfies
\begin{align}
 \label{heat_propagator} (\partial_t - \Delta_k)e^{t\Delta_k}\omega&=0
  ,
  \\
 \label{approximate_identity} \lim_{t \to 0^+}e^{t\Delta_k}\omega &= \omega
  .
\end{align}
For $\omega,\tilde{\omega}\in C^\infty(\Lambda^k)$ we define
\begin{equation}\label{inner_product}
  \Angles{\omega,\tilde{\omega}} = \int_\manifold \angles{\omega(p),\tilde{\omega}(p)}_p \,\dd V(p)
  ,
\end{equation}
where $\angles{\omega(p),\tilde{\omega}(p)}_p$ is the inner product on $\wedge^k(T_p^*\manifold)$, the $k^\text{th}$ exterior power of the cotangent space at the point $p$, and $V$ denotes the volume measure associated to the Riemannian manifold.
We say that $\omega$ is orthogonal to the space of harmonic forms and write $\omega \in \mathcal{H}^\perp(\Lambda^k)$ if $\Angles{\omega,\tilde{\omega}}=0$ for all harmonic $\tilde{\omega}$.

We use the heat propagator to define the Riesz potentials associated to forms.
Given $F \in C^\infty\cap \mathcal{H}^\perp(\Lambda^k)$ and $\alpha>0$, we define
\begin{equation}\label{RieszPotential}
  \mathcal{I}_{\alpha,k} F := \frac{1}{\Gamma(\alpha/2)} \int_0^\infty t^{\alpha/2-1} e^{t\Delta_k} F\,\dd t
  .
\end{equation}
This formula agrees with the usual definition in Euclidean space \cite{HRS}*{formula (1.2) on p.~1924}, and is the typical definition in a more general setting \cites{Folland,GT,KrantzPelosoSpector}.
All of these operators can be extended to the natural Sobolev spaces of $k$-forms, as in \cite{ISS}, and we use the same notation for these generalized operators.
We also write $\norm{\cdot}_{L^1(\Lambda^k)}$ and $\norm{\cdot}_{L^{r,s}(\Lambda^k)}$ for the Lebesgue and Lorentz norms; see the further discussion in \Cref{ss:DiffForms} and \Cref{appendix_B}.

With these preparations, we now state our first main result, which is a Lorentz-space estimate for fractional integration on forms.
\begin{theorem}\label{RieszClosedCoclosed_k}
  Let $\manifold$ be a smooth, compact manifold of dimension $\Dim\geq 2$, let $\alpha \in (0,\Dim)$, and let $k \in \set{1,\dotsc,\Dim-1}$.
  There exists a constant $C=C(\alpha,\manifold)>0$ such that
  \begin{align}\label{potentialnodiracl1}
    \|  \mathcal{I}_{\alpha,k} F \|_{L^{\Dim/(\Dim-\alpha),1}(\Lambda^k)} \leq C \| F\|_{L^1(\Lambda^k)}
  \end{align}
  for all $F \in L^1\cap \mathcal{H}^\perp(\Lambda^k)$ such that $\ed_k F=0$ or $\ed_k^*F=0$.
\end{theorem}

\noindent
We note that \Cref{RieszClosedCoclosed_k} and its proof remain valid when $k=0$, $\ed_k F=0$, and when $k=\Dim$, $\ed_k^* F=0$, but the result is then trivial: either assumption implies that $\ed_k F$ and $\ed_k^* F$ are both zero, so that $F$ is harmonic and hence necessarily zero.
On the other hand, the bound fails when $k=\Dim$, $\ed_k F=0$ or when $k=0$, $\ed_k^* F=0$.

For $\alpha=2$, the Riesz potential $\mathcal{I}_{2,k}$ corresponds to the inverse of the Hodge Laplacian, and so \Cref{RieszClosedCoclosed_k} yields an estimate for the solution of Poisson's equation for $k$-forms.

\begin{theorem}\label{Hodge_Laplacian_L1}
  Let $\Dim \geq 3$ and suppose $\manifold$ is a smooth, compact Riemannian manifold of dimension $\Dim$.
  There exists a constant $C=C(\manifold)>0$ such that the following holds.
  Let $k \in \set{1,\dotsc,\Dim-1}$, let $F \in L^1\cap \mathcal{H}^\perp(\Lambda^k)$, and suppose that $\ed_k F=0$ or $\ed^*_k F=0$.
  Let $Z$ be the unique solution in  $\mathcal{H}^\perp(\Lambda^k)$ of the Poisson equation
  \begin{equation}\label{PoissonPDEkForm}
    \Delta_{k}Z =F
    .
  \end{equation}
  Then $Z$ admits the estimate
  \begin{align}\label{estimate_Laplace}
    \|Z\|_{L^{\Dim/(\Dim-2),1}(\Lambda^k)} \leq C\|F\|_{L^1(\Lambda^k)}
    .
  \end{align}
\end{theorem}

Likewise, for $\alpha=1$, the Riesz potential $\mathcal{I}_{1,k}$ corresponds to the inverse square root of the Laplacian.
Using standard bounds for singular integrals, \Cref{RieszClosedCoclosed_k} therefore implies an estimate for Hodge systems.
\begin{theorem}\label{bbq}
  Let $\Dim \geq 3$ and suppose $\manifold$ is a smooth, compact Riemannian manifold of dimension $\Dim$.
  There exists a constant $C=C(\manifold)>0$ such that the following holds.
  Let $k \in \set{1,\dotsc,\Dim-1}$, let $F \in L^1\cap\mathcal{H}^\perp(\Lambda^{k-1})$ and $G \in L^1\cap \mathcal{H}^\perp(\Lambda^{k+1})$, and suppose that
  \begin{align*}
    \ed_{k-1}^* F = 0, \quad \ed_{k+1} G=0
    .
  \end{align*}
  In addition, we impose the conditions $F\equiv 0$ if $k=1$ and $G\equiv 0$ if $k=\Dim-1$.
 Let $Z$ be the unique solution in $ \mathcal{H}^\perp(\Lambda^k)$ of the Hodge system
  \begin{equation}\label{HodgeSystem}
    \begin{aligned}
      \ed_k^* Z &= F,   \\
      \ed_k Z &= G
      .
    \end{aligned}
  \end{equation}
Then $Z$ admits the estimate
  \begin{align}\label{estimateHodge}
    \|Z\|_{L^{\Dim/(\Dim-1),1}(\Lambda^k)} \leq C\left(\|F\|_{L^1(\Lambda^{k-1})}+\|G\|_{L^1(\Lambda^{k+1})}\right)
    .
  \end{align}
\end{theorem}

\Cref{RieszClosedCoclosed_k,Hodge_Laplacian_L1,bbq} involve only $L^1$ norms of their inputs, so it is natural to expect analogues for measure-like inputs too.
This notion will be formalized in terms of currents in \Cref{s:Preliminaries}, and the corresponding results are stated as \Cref{bbq_currents,Poisson_currents,Riesz_currents}.

In the special case $\manifold =\mathbb{T}^3$, the Hodge system can be interpreted as a div--curl system.
(For readers who find the vector field formulation more intuitive, details of this correspondence are given in \Cref{ss:TorusProof}.)
Then \Cref{bbq} yields an immediate answer to Bourgain and Brezis's \cite{BourgainBrezis2007}*{Open Question~1 on p.~295}.

\begin{corollary}\label{bbq_t}
  Suppose $F \in L^1(\mathbb{T}^3;\mathbb{R}^3)$ satisfies
  \begin{align*}
    \Div F&=0
    ,
    \\
    \int_{\mathbb{T}^3} F \, \dd x&=0
    .
  \end{align*}
  Then the function $Z=\curl (-\Delta_{\mathbb{T}^3})^{-1} F$ satisfies
  \begin{equation}\label{divcurlT3}
    \begin{aligned}
      \curl Z &= F
      ,
      \\
      \Div Z &= 0
      ,
    \end{aligned}
  \end{equation}
  and there exists a constant $C>0$ such that
  \begin{align}\label{estimateTorus}
    \| Z\|_{L^{3/2,1}(\mathbb{T}^3;\mathbb{R}^3)} \leq  C\| F\|_{L^1(\mathbb{T}^3;\mathbb{R}^3)}
    .
  \end{align}
  In particular,
  \begin{align}\label{BB_Kato}
    \text{the function}\quad x\mapsto \frac{Z(x)}{|x-a|} \quad\text{belongs to }L^1(\mathbb{T}^3;\mathbb{R}^3)
  \end{align}
  for every $a \in \mathbb{T}^3$.
\end{corollary}

Finally we state a result that, although stronger than \Cref{RieszClosedCoclosed_k}, relies more directly on the compactness of the manifold and is not preserved under dilation.
For a similar statement that holds in the Euclidean case without the assumption of compactness, see \cite{BVS}*{(3.3) on p.~375}.
\begin{theorem}\label{divinLresult}
  Let $\manifold$ be a smooth, compact Riemannian manifold of dimension $\Dim\geq 2$, and let $\alpha \in (0,\Dim)$.
  There exists a constant $C=C(\alpha,\manifold)<\infty$ such that for all $k \in \set{0,\dotsc,\Dim-1}$,
  \begin{align}\label{boundwithd}
    \|\mathcal{I}_{\alpha,k} F \|_{L^{\Dim/(\Dim-\alpha),1}(\Lambda^k)} \leq C \left( \|F\|_{L^1(\Lambda^k)}+\|\ed_k F\|_{L^1(\Lambda^{k+1})} \right)
  \end{align}
  for all $k$-forms $F \in L^1\cap \mathcal{H}^\perp(\Lambda^k)$ such that $\ed_k F\in L^1(\Lambda^{k+1})$; and for all $k\in\set{1,\dotsc,\Dim}$,
  \begin{equation}\label{boundwithdstar}
    \norm{ \mathcal{I}_{\alpha,k} F }_{L^{\Dim/(\Dim-\alpha),1}(\Lambda^k)} \leq C \left( \norm{F}_{L^1(\Lambda^k)} + \norm{\ed^*_k F}_{L^1(\Lambda^{k-1})} \right)
  \end{equation}
  for all $k$-forms $F \in L^1\cap \mathcal{H}^\perp(\Lambda^k)$ such that $\ed^*_k F\in L^1(\Lambda^{k-1})$.
\end{theorem}

\subsection{Discussion}

\subsubsection{Vector differential inequalities in the \texorpdfstring{$L^1$}{L1} regime}
\label{sss:Lorentzestimates}
The discovery of strong-type estimates in the $L^1$ regime for certain vector differential operators begins with the pioneering papers of Bourgain and Brezis, who, among other results, in their papers \cites{BourgainBrezis2004, BourgainBrezis2007} proved that for a given divergence free $F \in L^1(\mathbb{R}^3;\mathbb{R}^3)$, the function $Z:= \curl (-\Delta)^{-1}F$ satisfies
\begin{align}
  \curl Z &= F  \text{ in } \mathbb{R}^3, \label{curl}\\
  \Div Z &= 0 \text{ in } \mathbb{R}^3,\label{div}
\end{align}
and admits the estimate
\begin{align}\label{BB}
  \|Z\|_{L^{3/2}(\mathbb{R}^3;\mathbb{R}^{3})} \leq C\|F\|_{L^1(\mathbb{R}^3;\mathbb{R}^3)}
  .
\end{align}
\noindent
The inequality \eqref{BB} was the beginning of a new area of research on Sobolev inequalities beyond the gradient, and it led to questions of simpler proofs, extensions, applications, and refinements; see e.g.~\cites{AyoStoWoj2021,BousVSch2014, BousRussWangYung2019, ChanilloVSYung1, GRV2024, HS, HouniePicon2016, HouniePicon2021, LanzaniStein, Mazya2007, Mazya2010, RSS, Spector-VanSchaftingen-2018, Stolyarov2021, Stolyarov2022, Spector1,  VS, VS1, VS2, VS2010, VS3,Yung2010} and also the expository articles~\cites{Spector2020, VanSchaftingen2014, VanSchaftingen2024}.

It was a question of Bourgain and Brezis \cite{BourgainBrezis2007}*{Open Problem~1 on p.~295} whether for the analogue of \eqref{curl}--\eqref{div} on the torus, the estimate \eqref{BB} can be improved on the Lorentz scale to obtain a bound for $\| Z\|_{L^{3/2,1}}$.
This question is reiterated in various forms in \cite{VS2010}*{Open Problems~1--2}, \cite{VS3}*{Open Problem~8.3}, \cite{VanSchaftingen2014}*{Open Problem~2}, and \cite{BousVSch2014}*{top of p.~1423}.
The state of the art emerging from these papers is that there are bounds for $\| Z\|_{L^{3/2,s}}$ for all $s>1$, see \citelist{\cite{VS2010}*{Theorem~3} \cite{VS3}*{Theorem~8.5}}.
However, the optimal Lorentz exponent $s=1$ was unavailable.

Interest in Lorentz bounds in general, and optimal Lorentz bounds in particular, is natural from several perspectives.
The Lorentz spaces $L^{r,s}$ are precisely the spaces which arise in the real interpolation of the couple $(L^1, L^\infty)$ \cite{BerghLofstrom1976}*{p.~109}.
The use of the Lorentz scale, and particularly the optimal exponents, has useful applications in PDE \cite{Danshen, KeelTao1998}, and it has long been understood that these improvements in the ``microscopic'' parameter $s$ can become magnified to imply significant strengthened conclusions \cites{BrezisLorentz, Tartar1998}.
The parameter $s=1$ on the Lorentz scale plays a distinguished role, for example in embeddings into the space of continuous functions \cites{Helein2002,Stein1981} and in obtaining optimal rearrangement-invariant Sobolev inequalities \cites{BreitCianchi,BCS1,BCS2}.
Yet the historical tools used to establish such Lorentz estimates -- symmetrization and the co-area formula \cite{Mazya1960} or projections \cite{Spector-VanSchaftingen-2018} -- do not seem to be possible to adapt to the setting of the Hodge system, while the existing arguments by Littlewood-Paley theory \cites{BourgainBrezis2007,BousRussWangYung2019} or slicing \cites{LanzaniStein, VS} have not been sufficient to give a proof for $s=1$.

An answer to the Euclidean analogue of the question of Bourgain and Brezis was obtained by the second and third named authors in \cite{HS}, who proved that the solution of \eqref{curl}--\eqref{div} admits the Lorentz estimate
\begin{align}\label{BB_Lorentz}
  \| Z\|_{L^{3/2,1}(\mathbb{R}^3;\mathbb{R}^3)} &\leq C \|F\|_{L^1(\mathbb{R}^3;\mathbb{R}^3)}.
\end{align}
Two key technical ideas underlie both \cite{HS} and this manuscript, along with a third that is manifold-specific.
The first is to study the PDE problem \eqref{HodgeSystem} by means of the Riesz potential \eqref{RieszPotential}, with a link given by the Riesz transform.
As we outline in \Cref{sss:HodgeRieszDiscussion}, the Riesz transform can be estimated in the relevant norms, so it will suffice to bound the Riesz potentials, and specifically to bound the fixed-time heat propagators applied to forms, as in the integral \eqref{RieszPotential}.

The second is to use an atomic decomposition in terms of curves, which we outline in \Cref{sss:CurrentsAtomicDiscussion}.
This further reduces the question to establishing fixed-time bounds for heat propagators applied to curves.
These estimates are intricate, but are naturally suggested from the regularity properties guaranteed by the atomic decomposition, as we see below in \Cref{s:Closedd-1FormProof}.
These fixed-time bounds are easily seen to be sharp by simple examples, yet prove sufficient to obtain the optimal Lorentz exponent.

The third key technical idea is to prove the $k$-form estimates by means of $(\Dim-1)$-form estimates.
This step relies on an identity expressing $k$-forms and their exterior derivatives in terms of $(\Dim-1)$-forms and their corresponding exterior derivatives.
We discuss in \Cref{sss:ReductionToDivergence} how this reduction compares to similar constructions in the Euclidean case.
This identity is simple in the Euclidean context, but more subtle in the manifold context: see \Cref{s:d-1FormsTokForms}.

\subsubsection{Hodge systems and Riesz potentials}\label{sss:HodgeRieszDiscussion}

In the Euclidean context, the Riesz potential can be defined in various equivalent ways, including as a convolution against a dimension-dependent spatial kernel with a singularity along the diagonal.
The representation \eqref{RieszPotential} in terms of the heat propagator is equivalent \cite{HRS}*{formula (1.2) on p.~1924} and also eliminates the dimension-dependence and generalizes straightforwardly to the context of a Riemannian manifold.
The integral identity
\begin{equation*}
  \lambda^{-\alpha/2} = \frac{1}{\Gamma(\alpha/2)} \int_0^\infty t^{\alpha/2-1} e^{-t\lambda} \,\dd t
  ,\quad\lambda>0
  ,
\end{equation*}
shows that the Riesz potential corresponds to an inverse power of the Laplacian quite generally, using only that $-\Delta_k$ has non-negative spectrum.  Similar reasoning shows that the Riesz potentials form a semigroup: for $\omega \in L^1\cap \mathcal{H}^\perp(\Lambda^j)$, $\alpha,\beta>0$,  \begin{equation}\label{semigroup}
    \mathcal{I}_{\alpha,j} \mathcal{I}_{\beta,j} \omega = \mathcal{I}_{\alpha+\beta,j} \omega.
  \end{equation}

The link between the Riesz potential $\mathcal{I}_{1,k}$ and the Hodge system \eqref{HodgeSystem} is given by the form-valued Riesz transforms $\ed_k\mathcal{I}_{1,k}, \ed^*_k\mathcal{I}_{1,k}$, as mappings from $C^\infty\intersect\mathcal{H}^\perp(\Lambda^k)$ to $C^\infty(\Lambda^{k+1}), C^\infty(\Lambda^{k-1})$, respectively.
It requires some care to extend this definition to general spaces of forms.
However, our argument here relies only a single standard result for the Riesz transform, namely that it is a bounded operator with respect to the $L^p$ norm for all $1<p<\infty$.
The crucial estimates, including the interplay between the parameter $\alpha$ in the Riesz potential and the Lorentz space $L^{\Dim/(\Dim-\alpha), 1}(\Lambda^k)$, will all emerge from an analysis of the well-behaved integral representation \eqref{RieszPotential}.
The proof will obtain pointwise estimates for the integrand in \eqref{RieszPotential}: the argument rests on a delicate analysis across different spatial scales, but is ultimately natural and direct.
We show that the integral \eqref{RieszPotential} is absolutely convergent, and we discuss in \Cref{sss:BesovLorentzDiscussion} how to interpret the resulting bounds in terms of a Besov-Lorentz norm.

We remark that in the definition of the Riesz potential $\mathcal{I}_{\alpha,k}$, and similarly in \Cref{Hodge_Laplacian_L1,bbq}, we require the forms to be orthogonal to the space of harmonic forms, i.e., orthogonal to the eigenspace of the Laplacian for $\lambda=0$.
However, this is no loss of generality: it is readily verified that for \eqref{PoissonPDEkForm} and \eqref{HodgeSystem} to have a solution, it is necessary for $F$ and $G$ to be orthogonal to the appropriate space of harmonic forms.

\subsubsection{Currents and the atomic decomposition}\label{sss:CurrentsAtomicDiscussion}

A key step in this paper's examination of forms is to express them in terms of currents, whose formal definition will be given in \Cref{ss:Currents}.
Currents are natural analytical tools that generalize both integration along lower-dimensional oriented subspaces, particularly curves (\Cref{ss:CurrentsFromCurves}); and also smooth forms, particularly $(\Dim-1)$-forms (\Cref{ss:CurrentsFromForms}).

Central to the proof is the atomic decomposition for 1-currents without boundary.
First described in \cite{HS}*{Theorem~1.5} for Euclidean space, and generalized and simplified in \cite{ChenGoodmanHernandezSpectorPreprint}*{Theorem~1.1}, the atomic decomposition builds on results of Smirnov \cite{Smirnov} and Paolini and Stepanov \cite{PS1} to express a 1-current without boundary in terms of closed curves with a crucial regularity condition expressed in terms of the Morrey 1-norm.
The version stated here applies in the simpler special case of a Riemannian manifold; for notation and definitions, and for an explanation of why it follows from the general case, see \Cref{s:Preliminaries}.

\begin{theorem}[Adapted from \cite{ChenGoodmanHernandezSpectorPreprint}*{Theorem~1.1}]
  \label{AtomicDecompositionCited}
  ~
  There exists a universal constant $M<\infty$ such that for any Riemannian manifold $\manifold$ and any 1-current $T\in \currents_1$ with $\boundary_1 T=0$, there exist piecewise-geodesic closed curves $\gamma_{i,m}$ and scalars $\lambda_{i,m} \geq 0$, $m \in \mathbb{N}$, $1\leq i\leq m$, such that
  \begin{gather}
    T(\xi)= \lim_{m \to \infty} \sum_{i=1}^m \lambda_{i,m} \frac{[[\gamma_{i,m}]](\xi)}{\length(\gamma_{i,m})}
    \quad\text{for all $\xi\in C(\Lambda^1)$}, \label{limit}
    \\
    \sum_{i=1}^m\abs{\lambda_{i,m}} \leq 2 \mathbb{M}(T)
    \quad\text{for all $m$, and} \label{mass}
    \\
    \Morrey{\gamma_{i,m}} \leq M
    \quad\text{for all $i,m$}. \label{morrey_bound}
  \end{gather}
\end{theorem}

The atomic decomposition will allow us to study the heat propagator on forms in terms of a heat propagator acting on closed curves, for which there is simpler geometric intuition.
However, because the atomic decomposition \eqref{limit} includes the lengths $\length(\gamma_{i,m})$ of the curves in the denominator, we will inherently need upper bounds that are linear in the length of the curves, at least for small curves.
The proof makes use of the regularity property \eqref{morrey_bound}, which is a crucial additional element of the atomic decomposition, via an intricate combination of bounds capturing the change in scaling behaviour across matching spatial and temporal scales.
We remark that the argument for a manifold carries additional complexities compared to the Euclidean version, where several quantities have simpler interpretations as integrals against an explicit kernel.

As we will explain later, there is a connection between $k$-forms and exterior derivatives on the one hand and $(\Dim-k)$-currents and their boundaries on the other.
Because the atomic decomposition applies only to $1$-currents, the resulting proof will be specific to $(\Dim-1)$-forms.

\subsubsection{Reducing to \texorpdfstring{$(\Dim-1)$}{(\Dim-1)}-forms}\label{sss:ReductionToDivergence}

Among other results, in \cite{VS3} Van Schaftingen observed that estimates for general first-order constant-coefficient co-cancelling operators can be reduced to estimates for the divergence.
The exterior derivatives for $\R^\Dim$ are examples of first-order co-cancelling operators, and so they can be written as divergences, modulo appropriate linear maps \cite{VS3}*{Proposition 3.3 and Lemma~2.5 and the remarks following}.
In the context of forms on a manifold, the analogue of the divergence is the exterior derivative on $(\Dim-1)$-forms.
In \Cref{s:d-1FormsTokForms}, we establish an extension of Van Schaftingen's result beyond the constant-coefficient case, showing that a $k$-form exterior derivative can be expressed in terms of $m$-form exterior derivatives for all $k\leq m\leq\Dim-1$: see \Cref{dPhiRelation,rem:dkAsdm}.
This provides a manifold analogue of Van Schaftingen's algebraic construction in \cite{VS3}*{Lemma~2.5 on p.~884} and \cite{VS2010}*{proof of Proposition~2.2 on p.~238}.

This idea allows the $k$-form estimate to be proved from the $(\Dim-1)$-form estimate.
Our proof takes a closed $k$-form and maps it to a vector of closed $(\Dim-1)$-forms; propagates the resulting closed $(\Dim-1)$-forms and bounds the result by the preceding current-based analysis; then maps back to recover a $k$-form.
A complication, specific to the manifold context, is that the heat propagator on forms is different depending on the degree, so the argument requires some analysis to handle the difference of these propagators.
However, the propagators can be shown to agree in their lowest-order terms.
This reflects the intuition that, for the purposes of analysis and integral bounds, the assumption that a $k$-form is closed is weakest when $k=\Dim-1$.

\subsubsection{Besov-Lorentz estimates}\label{sss:BesovLorentzDiscussion}

Our analysis is based on bounding heat propagators at fixed times, so that  \Cref{RieszClosedCoclosed_k} is proved as a consequence of the following slightly stronger estimate.

\begin{theorem}\label{BesovClosedCoclosed}
  Let $\manifold$ be a smooth, compact manifold of dimension $\Dim\geq 2$, let $\alpha \in (0,\Dim)$, and let $k \in \set{1,\dotsc,\Dim-1}$.
  There exists a constant $C=C(\alpha,\manifold)>0$ such that
  \begin{equation}\label{BesovIntegralBound}
    \frac{1}{\Gamma(\alpha/2)}\int_0^\infty t^{\alpha/2-1} \norm{e^{t\Delta_k}F}_{L^{\Dim/(\Dim-\alpha),1}(\Lambda^k)} \, \dd t \leq C \norm{F}_{L^1(\Lambda^k)}
  \end{equation}
  for all $F \in L^1\cap\mathcal{H}^\perp(\Lambda^k)$ such that $\ed_k F=0$ or $\ed_k^*F=0$.
\end{theorem}

To interpret \Cref{BesovClosedCoclosed}, note that the left-hand side of \eqref{BesovIntegralBound} defines a norm for the $k$-form $F$.
We will explain in \Cref{appendix_B} that this norm corresponds to a Besov-Lorentz space, under a generalized definition that applies in the setting of a Riemannian manifold.
We show that the norm from \eqref{BesovIntegralBound} is equivalent in the Euclidean case to the usual definition of Besov-Lorentz norm in terms of Littlewood-Paley projectors.
Moreover, in direct analogy with the Euclidean case, the Laplacian and Riesz potentials are homeomorphisms between different Besov-Lorentz spaces: see \Cref{Besov_homeomorphisms}.
It follows that \Cref{BesovClosedCoclosed} has the following equivalent formulation in terms of embeddings between Besov-Lorentz spaces.
\begin{theorembis}{BesovClosedCoclosed}\label{BesovClosedCoclosedReformulated}
  Let $\manifold$ be a smooth, compact manifold of dimension $\Dim\geq 2$, let $\alpha \in (0,\Dim)$, and let $k \in \set{1,\dotsc,\Dim-1}$.
  There exists a constant $C=C(\alpha,\manifold)>0$ such that
  \begin{equation}
    \norm{F}_{\dot{B}^{-\alpha,1}_{\Dim/(\Dim-\alpha),1}(\Lambda^k)} \leq C \norm{F}_{L^1(\Lambda^k)}
  \end{equation}
  for all $F \in L^1\cap\mathcal{H}^\perp(\Lambda^k)$ such that $\ed_k F=0$ or $\ed_k^*F=0$.
  In particular, there are embeddings
  \begin{equation}\label{EmbedIntoBesov}
    \begin{aligned}
      \left\{ F \in L^1\intersect\mathcal{H}^\perp(\Lambda^k) \colon \ed_k F=0 \right\} &\hookrightarrow \dot{B}^{-\alpha,1}_{\Dim/(\Dim-\alpha),1}(\Lambda^k)
      ,
      \\
      \left\{ F \in L^1\intersect\mathcal{H}^\perp(\Lambda^k) \colon \ed_k^* F=0 \right\} &\hookrightarrow \dot{B}^{-\alpha,1}_{\Dim/(\Dim-\alpha),1}(\Lambda^k)
      .
    \end{aligned}
  \end{equation}
\end{theorembis}
\noindent
In \Cref{Besov-nested} we show that our definition implies the same embeddings of Besov-Lorentz spaces that are known to hold in the Euclidean setting: for $0 \leq \alpha <\beta$,
\begin{align}
  \dot{B}^{0,1}_{1,1}(\Lambda^k) \hookrightarrow  \dot{B}^{-\alpha,1}_{\Dim/(\Dim-\alpha),1}(\Lambda^k) \hookrightarrow \dot{B}^{-\beta,1}_{\Dim/(\Dim-\beta),1}(\Lambda^k)
  .
\end{align}
Thus \Cref{BesovClosedCoclosedReformulated} shows that the closed forms of finite mass form a space that is intermediate between $\dot{B}^{0,1}_{1,1}(\Lambda^k)$ and $\dot{B}^{-\alpha,1}_{\Dim/(\Dim-\alpha),1}(\Lambda^k)$ for all $\alpha>0$:
\begin{align}
  \shortset{F\in\dot{B}^{0,1}_{1,1}(\Lambda^k)\colon \ed_k F=0} \hookrightarrow \left\{ F \in L^1\intersect\mathcal{H}^\perp(\Lambda^k) \colon \ed_k F=0 \right\}  \hookrightarrow \dot{B}^{-\alpha,1}_{\Dim/(\Dim-\alpha),1}(\Lambda^k)
  .
\end{align}
We refer the reader to \Cref{appendix_B} for the definitions, notation, proofs, and further properties of Lorentz and Besov-Lorentz spaces.

\section{Preliminaries}\label{s:Preliminaries}

Throughout the paper, we use $C=C(\alpha,\manifold)$ to denote an unspecified finite positive constant, whose value may change from occurrence to occurrence but which depends only on the choice of manifold $\manifold$ (including its dimension $\Dim$) and, where applicable, the parameter $\alpha$ from \Cref{divinLresult,RieszClosedCoclosed_k,BesovClosedCoclosed}.
Likewise, for positive $a,b$, we will write $a\lesssim b$ to mean that $a\leq C b$ for such a constant $C$, uniformly over any other quantities on which $a,b$ depend.

\subsection{A finite atlas for the manifold}\label{ss:FiniteAtlas}

Throughout the paper we take $\manifold$ to be a $C^\infty$-smooth, compact Riemannian manifold of dimension $\Dim$.
The length of a piecewise smooth curve $\gamma\colon [a,b]\to \manifold$ with finitely many corner points is given by
\begin{equation}\label{length_curve}
  \length(\gamma) = \int_a^b \norm{\gamma'(t)}_{\gamma(t)} \, \dd t
  ,
\end{equation}
where $\norm{v}_p$ means the norm of a tangent vector $v\in T_p\manifold$ at a point $p\in\manifold$, as induced by the Riemannian metric.
This in turn allows us to define the distance between points $p, q\in \manifold$:
\begin{equation}
  d_\manifold(p,q) = \inf\set{\length(\gamma)\colon \gamma\text{ is a curve in $\manifold$ from $p$ to $q$}}
  .
\end{equation}
We write $V$ for the (unsigned) measure on $\manifold$ induced by the Riemannian metric.

We recall that the manifold admits an atlas of charts $(x,U)$, where each $x \colon U \subset \manifold \mapsto x(U) \subset \mathbb{R}^\Dim$ is a homeomorphism between $U$ and an open subset of $\R^n$, and different charts are $C^\infty$-related: for any two charts $(x,U)$, $(\tilde{x},\tilde{U})$ with $U\intersect\tilde{U}\neq\emptyset$, the functions
\begin{align*}
  x \circ \tilde{x}^{-1} &\colon \tilde{x}(U \intersect \tilde{U}) \to x(U \intersect \tilde{U})
  ,
  \\
  \tilde{x} \circ x^{-1} &\colon x(U \intersect \tilde{U}) \to \tilde{x}(U \intersect \tilde{U})
  ,
\end{align*}
are $C^\infty$ maps between their corresponding domains and ranges in $\mathbb{R}^\Dim$.
Without loss of generality we may take the atlas to be maximal \cite{Spivak_DifferentialGeometry}*{Lemma~1 on p.~29}, and in particular the subsets $U\subset\manifold$ may be chosen to be small neighbourhoods as needed.

For later use, we will fix a particular finite atlas for the compact manifold $\manifold$.
Consider triples $(x,U,U')$ for which $(x,U)$ is a chart; the smooth maps $x,x^{-1}$ (considered as mappings between $U$ and a subset of $\R^\Dim$ with the usual Euclidean metric) have uniformly bounded derivatives; between any two points $p,q\in U$ there is a unique length-minimizing geodesic for $\manifold$ that lies inside $U$; and $U'$ is a non-empty open subset of a compact subset of $U$.
The collection of subsets $U'$, over all such triples $(x,U,U')$, forms an open cover of $\manifold$, and therefore by compactness we may choose finitely many triples $(x_r,U_r,U'_r)_{r=1,\dotsc,\maxr}$ such that $\cup_r U'_r=\manifold$.

Finally, fix a smooth partition of unity $(\chi_r)_{r=1,\dotsc,\maxr}$ subordinate to $(U'_r)_{r=1,\dotsc,\maxr}$, and fix smooth cutoff functions $\rho_r\colon\manifold\to[0,1]$ such that $\rho_r=1$ in a neighbourhood of the closure of $U'_r$ and $\rho_r$ vanishes outside a compact subset of $U_r$.

\subsection{Differential forms}\label{ss:DiffForms}

We next introduce the requisite notation for differential forms on $\manifold$; we refer the reader to \cite{Spivak_DifferentialGeometry} or \cite{ISS} for a more comprehensive development.
For $p \in \manifold$, we write $T_p^* \manifold$ for the cotangent space, with the inner product and norm induced by the Riemannian metric; in other words, $T_p^*\manifold$ is the dual of the tangent space, with the dual inner product and norm induced by the finite-dimensional inner product space space $T_p \manifold$.

We write $\Lambda^k$ for the space of $k$-forms on $\manifold$.
That is, a $k$-form is a section of the $k^\text{th}$ exterior power of the cotangent bundle, which we may interpret as a function $\zeta$ on $\manifold$ such that $\zeta(p) \in \wedge^k(T_p^* \manifold)$ for all $p \in \manifold$.
The value $\zeta(p)$ at a point $p$ is itself a function, and we will write $\zeta(p;v_1,\dotsc,v_k)\in\R$ for its value when we wish to apply it to tangent vectors $v_1,\dotsc,v_k\in T_p\manifold$.
The inner product and norm from $T_p \manifold$ induce an inner product $\angles{\zeta(p),\smash{\tilde{\zeta}}(p)}_p$ and a norm $\norm{\zeta(p)}_p$ on $\wedge^k(T_p^* \manifold)$.
We typically use the Greek letters $\omega, \zeta, \eta, \xi$ for forms, and also, as in the introduction, the Roman letters $F,G,Z$.

With the pairing \eqref{inner_product} (which omits the dependence on $k$), the exterior derivative and codifferential are adjoint to each other: for $\omega\in C^1(\Lambda^k)$ and $\zeta\in C^1(\Lambda^{k+1})$,
\begin{equation}
  \Angles{\ed_k\omega, \zeta} = \Angles{\omega, \ed^*_{k+1}\zeta}
  .
\end{equation}
From the definition of the Hodge Laplacian in \eqref{HodgeLaplacian}, it follows that $\Delta_k$ is self-adjoint: for all $\omega,\tilde{\omega}\in C^2(\Lambda^k)$,
\begin{equation}\label{HodgeLaplacianSelfAdjoint}
  \Angles{\Delta_k\omega,\tilde{\omega}} = \Angles{\omega,\Delta_k\tilde{\omega}}
  .
\end{equation}
Moreover the Laplacian commutes with exterior derivative and codifferential:
\begin{equation*}
  \ed_k \Delta_k\omega = \Delta_{k+1}\ed_k\omega
  , \quad
  \ed^*_k\Delta_k\omega = \Delta_{k-1}\ed^*_k\omega
  .
\end{equation*}
It follows that the same is true the heat propagator $e^{t\Delta_k}$ and hence for the Riesz potential, and in particular
\begin{gather}
  \qquad  \;\; \Angles{\mathcal{I}_{\alpha,k} \omega,\tilde{\omega}} = \Angles{\omega,\mathcal{I}_{\alpha,k} \tilde{\omega}}
  ,
  \label{Riesz_selfadjoint}
  \\
  \ed^*_{k+1} \ed_k \mathcal{I}_{\alpha,k} \omega = \mathcal{I}_{\alpha,k} \ed^*_{k+1} \ed_k \omega
  ,
  \quad \text{and} \quad
  \ed_{k-1} \ed^*_k \mathcal{I}_{\alpha,k} \omega = \mathcal{I}_{\alpha,k} \ed_{k-1} \ed^*_k \omega \label{IddstarCommute}
\end{gather}
for all $\omega, \tilde{\omega} \in C^2\cap\mathcal{H}^\perp(\Lambda^k)$.
Note that because of the compactness of the manifold and our sign conventions, the Laplacian is negative semi-definite but not strictly negative definite.
Thus the heat propagator does not decay toward zero in general, and indeed
\begin{equation}\label{LongTimeHeatPropagator}
  \lim_{t\to\infty} e^{t\Delta_k}\omega = P_{\mathcal{H}}\omega
\end{equation}
for all $\omega\in C(\Lambda^k)$, where $P_{\mathcal{H}}$ denotes orthogonal projection onto the space of harmonic forms.

These definitions can be extended to appropriate Sobolev spaces of $k$-forms, as in \cite{ISS}, and we use the same notation for these extended operators.

We write $X(\Lambda^k)$ to denote the $k$-forms with the smoothness or integrability described by $X$.
For smoothness, such notation has already appeared in the introduction with $C^\infty(\Lambda^k)$, the space of smooth $k$-forms, and we will also use the space $C(\Lambda^k)$ of continuous $k$-forms and the space $C^1(\Lambda^k)$ of continuously differentiable $k$-forms.
For integrability assumptions, we define norms on spaces of forms by reducing to scalar functions:
\begin{equation}\label{FormNormsFromScalarNorms}
  \norm{\omega}_{X(\Lambda^k)} = \norm{f}_{X} \quad\text{where}\quad f\colon U\to \cointerval{0,\infty}, \; f(p)=\norm{\omega(p)}_p
  .
\end{equation}
The corresponding spaces consist of those forms for which the norm is finite.
Thus writing $\omega\in L^1(\Lambda^k)$ or $\omega\in L^{\Dim/(\Dim-\alpha),1}(\Lambda^k)$, as in the introduction, means that $\omega$ is a $k$-form (on $\manifold$) and the scalar function $f(p)=\norm{\omega(p)}_p$ satisfies $f\in L^1(\manifold)$ or $f \in L^{\Dim/(\Dim-\alpha),1}(\manifold)$, where in all cases the associated measure is the volume measure $V$.
For simplicity of notation, we denote intersections such as $L^1(\Lambda^k)\intersect\mathcal{H}^\perp(\Lambda)$ by $L^1\intersect\mathcal{H}^\perp(\Lambda^k)$, and we will sometimes abbreviate the Lebesgue and Lorentz norms of a $k$-form as $\norm{\omega}_{L^r},\norm{\omega}_{L^{r,1}}$ instead of $\norm{\omega}_{L^r(\Lambda^k)}, \norm{\omega}_{L^{r,1}(\Lambda^k)}$.

Because norms on spaces of forms are defined in terms of norms of scalar functions, standard properties of norms carry over without change.
Notably, the Lorentz spaces of forms have a natural interpolation property.
\begin{lemma}[Interpolation]\label{Interpolation}
  \begin{equation*}
    \norm{F}_{L^{\Dim/(\Dim-\alpha),1}} \leq \frac{\Dim^2}{\alpha(\Dim-\alpha)} \norm{F}_{L^1}^{1-\alpha/\Dim} \norm{F}_{L^\infty}^{\alpha/\Dim}
    .
  \end{equation*}
\end{lemma}
\noindent
We refer the reader to \Cref{appendix_B} for a proof.

\subsection{Currents}\label{ss:Currents}

We next introduce some background concerning currents on $\manifold$.
For the purposes of this paper, a $j$-current $T \in \currents_j$ is a bounded linear functional on $C(\Lambda^j)$ equipped with the supremum norm: that is,
\begin{align*}
  T(\alpha \omega+\beta \xi)= \alpha T( \omega) +\beta T(\xi)
\end{align*}
for all $\alpha,\beta \in \mathbb{R}$ and $\omega, \xi \in C(\Lambda^j)$, and there exists $C=C(T)<\infty$ such that
\begin{align*}
  \left|T(\omega)\right| \leq C \|\omega\|_{L^\infty(\Lambda^j)}
\end{align*}
for all $\omega \in C(\Lambda^j)$.
A current under this definition is equivalent to a Federer-Fleming current of finite mass, and we remark in advance that the currents needed for our argument will all belong to the well-behaved subclass of normal currents.

The total mass of the $j$-current $T$ means the operator norm of $T$ with respect to the $L^\infty$ norm,
\begin{equation}\label{CurrentTotalMass}
  \totalmass(T) = \sup_{\xi\in C(\Lambda^k)\colon\norm{\xi}_{L^\infty}\leq 1} \abs{T(\xi)}
  .
\end{equation}
More generally, there is a unique measure $\mu_T$ on $\manifold$ such that
\begin{equation}
  \abs{T(\xi)} \leq \int_\manifold \norm{\xi(p)}_p \dd\mu_T(p) \quad\text{for all }\xi \in C(\Lambda^k)
\end{equation}
and such that $\mu_T$ is minimal among all measures on $\manifold$ with this property.
We call $\mu_T$ the mass measure of $T$ and note that
\begin{equation}
  \totalmass(T) = \mu_T(\manifold).
\end{equation}

For $T \in \currents_j$ we write $\boundary_j T$ to denote the boundary of a $j$-current $T$, defined in a distributional sense by
\begin{equation}\label{boundary}
  \boundary_j T(\xi) = T(\ed_{j-1} \xi) \quad\text{for all }\xi\in C^1(\Lambda^{j-1})
  .
\end{equation}
Similarly, the Laplacian on $j$-currents is defined by duality:
\begin{equation}\label{LaplacianOnCurrents}
  -\Delta_{j,c} T(\eta)= T(-\Delta_j\eta)
  \quad\text{for all }\eta \in C^2(\Lambda^{j})
  .
\end{equation}
Note that for some currents $T$, the functions $\xi\mapsto \boundary_j T(\xi)$ and $\eta\mapsto-\Delta_{j,c}T(\eta)$ need not have extensions as bounded linear functionals on $\xi\in C(\Lambda^{j-1})$ or $\eta\in C(\Lambda^j)$.
Thus the boundary and Laplacian of a current need not be currents.

Given a $j$-current $S$, we define the heat propagator applied to $S$ as the $j$-current
\begin{align}\label{PropagatorOnCurrentsInside}
  e^{t\Delta_{j,c}} S(\eta) = S(e^{t\Delta_{j}}\eta).
\end{align}
Note that, in contrast to the boundary and Laplacian, $e^{t\Delta_{j,c}}S$ is itself a $j$-current:
\begin{align*}
  \abs{e^{t\Delta_{j,c}}S(\eta)} = \abs{S(e^{t\Delta_{j}}\eta)} &\leq \totalmass(S) \|e^{t\Delta_{j}}\eta\|_{L^\infty(\Lambda^j)}
  \leq \totalmass(S) c(t) \|\eta\|_{L^\infty(\Lambda^j)}
  ,
\end{align*}
where the second inequality follows from the fact that the $j$-form heat propagator $e^{t\Delta_j}$ is a bounded operator from $L^\infty(\Lambda^j)$ to $L^\infty(\Lambda^j)$ for each $t>0$.

We remark that the operation $e^{t\Delta_{j,c}}$ as defined in \eqref{PropagatorOnCurrentsInside} is indeed the propagator corresponding to $\Delta_{j,c}$: if we write $S_t\equiv e^{t\Delta_{j,c}} S$, then it can be shown from the regularity properties of the $j$-form heat propagator $e^{t\Delta_j}$ that
\begin{equation*}
  \frac{\partial}{\partial t} S_t(\eta) = \Delta_{j,c}S_t(\eta)
  \quad\text{for all }t>0
\end{equation*}
and $S_t\to S_0=S$ in the weak-star topology as $t\to 0$ in the sense that
\begin{align*}
  \lim_{t\decreasesto 0} S_t(\eta) = S_0(\eta) = S(\eta)
  \quad\text{for all }\eta\in C(\Lambda^j)
  .
\end{align*}

For later reference, we also remark that if $T$ is a $j$-current on $\manifold$ and $f\colon\manifold\to \mathcal{N}$ is a smooth map with uniformly bounded derivatives, then the push-forward $f_\# T$ defined by $f_\# T(\xi)=T(f^\# \xi)$ is a $j$-current on $\mathcal{N}$.
Exterior differentiation commutes with form pull-back (see for instance \cite{Spivak_DifferentialGeometry}*{Proposition~7.10}) and it follows that the boundary operation commutes with current push-forward, $\boundary_j f_\# T = f_\# \boundary_j T$.

We remark that the definition of currents given here, in terms of differential forms arising from the differentiable structure of the manifold, differs from the definition used in \cite{ChenGoodmanHernandezSpectorPreprint}, which follows \cite{AK} in defining currents for a general metric space in terms of tuples of Lipschitz continuous functions.
For the case of normal currents we are interested in, however, the two have a canonical identification, see \cite{AK}*{Theorem~11.1 on p.~69}.
For similar reasons, our definition of $[[\gamma]]$, to be given in \Cref{ss:CurrentsFromCurves} below, is compatible.
Additionally, \cite{ChenGoodmanHernandezSpectorPreprint}*{Theorem~1.1} assumes that the space is connected (in particular, that each pair of points is joined by a length-minimizing geodesic) whereas \Cref{AtomicDecompositionCited} does not.
This does not cause problems: it suffices to apply \cite{ChenGoodmanHernandezSpectorPreprint}*{Theorem~1.1} to each connected component of $\manifold$, noting that each current $T$ is supported in at most a countable number of components of $\manifold$ since the mass measure $\mu_T$ is finite (and in any case a compact manifold can have only finitely many components).
Therefore \Cref{AtomicDecompositionCited} is a valid application of  \cite{ChenGoodmanHernandezSpectorPreprint}*{Theorem~1.1} (in the special case $\epsilon=1$).

\subsection{1-currents induced by curves}\label{ss:CurrentsFromCurves}

A key element of our analysis will be the 1-currents arising from curves.
Given a piecewise smooth curve $\gamma\colon[a,b]\to \manifold$ with tangent vectors $\gamma'(t)\in T_{\gamma(t)}\manifold$ for all $t\in[a,b]$ except possibly at finitely many corner points, we can define a $1$-current $[[\gamma]]$ as follows.
For $\xi \in C^\infty(\Lambda^1)$ with values denoted $\xi(p;v)\in\R$ for $p\in \manifold, v\in T_p \manifold$,
\begin{equation}
  [[\gamma]](\xi) = \int_\gamma \xi = \int_a^b \xi\left( \big. \gamma(t); \gamma'(t) \right) \dd t
  .
\end{equation}
This quantity is invariant under piecewise smooth orientation-preserving time-reparametrization of $\gamma$.
Note that if $b_\gamma=\gamma(a),e_\gamma=\gamma(b)$ denote the beginning and ending points of the curve, then the 0-current $\boundary_1[[\gamma]]$ is given by $\boundary_1[[\gamma]](f) = f(e_\gamma) - f(b_\gamma)$.
In particular, a curve is closed if and only if the associated 1-current has zero boundary.

Define $\mu_\gamma$, the length measure of $\gamma$ on $\manifold$, to be the Borel measure on $\manifold$ characterized by
\begin{equation}
  \int_\manifold f(p) \, \dd\mu_\gamma(p) = \int_a^b f(\gamma(t)) \norm{\gamma'(t)}_{\gamma(t)} \, \dd t
\end{equation}
for all measurable functions $f$ defined on $\manifold$ for which the right-hand side exists.
The total mass of $\mu_\gamma$ is $\length(\gamma)$.
If $\gamma\colon [0,\length(\gamma)]\to \manifold$ is parametrized by arc length then $\mu_\gamma$ is the image of Lebesgue measure on $[0,\length(\gamma)]$ under the mapping $\gamma$.

The current $[[\gamma]]$ also produces a Borel measure on $\manifold$, the mass measure $\mu_{[[\gamma]]}$.
These measures are related by $\mu_{[[\gamma]]} \leq \mu_\gamma$, see \cite{ChenGoodmanHernandezSpectorPreprint}*{Lemma~4.1}, and in particular $\totalmass([[\gamma]])\leq\length(\gamma)$.
We remark that strict inequality is possible, for instance when $\gamma$ traverses the same part of the curve in opposite directions.

Given any measure $\mu$ on $\manifold$, we define its Morrey norm to be
\begin{equation}
  \Morrey{\mu} = \sup_{r>0, \, p\in \manifold} \frac{\mu(\set{q\in\manifold\colon d_\manifold(p,q)\leq r})}{r}
  .
\end{equation}
We can then define Morrey norms on curves and currents by
\begin{align}
  \Morrey{\gamma} &= \Morrey{\mu_\gamma}
  ,
  &
  \Morrey{T} &= \Morrey{\mu_T}
  ,
\end{align}
and we note that $\Morrey{[[\gamma]]}\leq \Morrey{\gamma}$.
For further information on Morrey spaces, we refer the reader to \cite{Adams:Morrey} and the references therein.

If $\gamma$ is a sufficiently small closed curve in $\manifold$, we can think of $\gamma$ as the boundary of a two-dimensional surface, with area of order $\length(\gamma)^2$.
The following result uses the compactness of $\manifold$ to give a uniform quantitative version of this intuition.

\begin{lemma}\label{SmallCurvesAreBoundaries}
  Given a compact manifold $\manifold$, there exist constants $\length_0>0$ and $C<\infty$ such that, for all closed piecewise smooth curves $\gamma$ in $\manifold$ of length $\length(\gamma)\leq\length_0$, there exists a $2$-current $S_\gamma$ such that $\boundary_2 S_\gamma = [[\gamma]]$ and $\totalmass(S_\gamma)\leq C\length(\gamma)^2$.
  In particular,
  \begin{equation*}
    \abs{[[\gamma]](\xi)} = \abs{S_\gamma(\ed_1\xi)} \leq C\length(\gamma)^2\norm{\ed_1\xi}_{L^\infty}
    \quad\text{for all }\xi\in C^\infty(\Lambda^1)
    .
  \end{equation*}
\end{lemma}

\begin{proof}
  Recall from \Cref{ss:FiniteAtlas} the finite collection of triples $(x_r,U_r,U'_r)_{r=1,\dotsc,\maxr}$ such that $U'_r$ is contained in a compact subset of $U_r$ and the smooth mappings $x_r,x_r^{-1}$ have uniformly bounded derivatives (as mappings between subsets of $\manifold$ and $\R^\Dim$).
  We begin by choosing $\length_0$ sufficiently small so that $\length(\gamma)\leq\length_0$ implies that $\gamma$ lies within a single such subset $U'_r$ for some $r$.

  Now let $\gamma$ be an arbitrary such closed curve of length $\length(\gamma)\leq\length_0$ and suppose that $\gamma$ lies within $U'_r$.
  The composition
  \begin{align*}
    \tilde{\gamma} = x_r \circ \gamma
  \end{align*}
  defines a piecewise smooth closed curve in $\R^\Dim$, and it is readily verified that the associated 1-currents are related by $[[\tilde{\gamma}]]=(x_r)_\# [[\gamma]]$, $[[\gamma]]=(x_r^{-1})_\# [[\tilde{\gamma}]]$.
  The closedness of the curves implies that
  \begin{align*}
    \boundary_1 [[\tilde{\gamma}]]=0
    .
  \end{align*}
  Using the notation of \cite{Federer}*{4.2.10 and pp.~406--408}, we have $[[\tilde{\gamma}]]\in I_1(\R^\Dim)$, and there exist $\tilde{C}<\infty$ and $\tilde{S} \in I_2(\mathbb{R}^\Dim)\subset \currents_2(\R^\Dim)$ (a generalized spanning surface) such that
  \begin{align*}
    \boundary_2 \tilde{S}&= [[\tilde{\gamma}]]
    ,
    \\
    \totalmass([[\tilde{S}]])&\leq  \tilde{C}\length( \tilde{\gamma})^2
    .
  \end{align*}
  Moreover that result guarantees that the support of $\tilde{S}$ lies within distance $C \length(\gamma)$ of the image of $\tilde{\gamma}$.
  In particular, by shrinking the value $\length_0$ further if necessary, we may make this distance sufficiently small to guarantee $\operatorname*{supp}\tilde{S} \subset x_r(U_r)$.

  Finally, the push-forward of $\tilde{S}$ defines a $2$-current
  \begin{align*}
    S_\gamma = (x_r^{-1})_\# \tilde{S} \in\currents_2
  \end{align*}
  satisfying the desired claims.
\end{proof}

\subsection{Currents induced by forms}\label{ss:CurrentsFromForms}

Now suppose that $\manifold$ is an orientable manifold, and fix an orientation of $\manifold$.
This induces a well-defined notion of integration of a $\Dim$-form over the manifold, $\int_\manifold \zeta$, as a linear functional on $\zeta\in L^1(\Lambda^\Dim)$.
Indeed, integration of $\Dim$-forms over an oriented (sub-)manifold is a basic example of an $\Dim$-current.

We can generalize this to produce currents from forms.
Given a $k$-form $\omega \in L^1(\Lambda^k)$, define the $(\Dim-k)$-current $T_\omega$ by
\begin{equation}\label{FormAsCurrentWedge}
  T_\omega(\xi) = \int_\manifold \xi \wedge \omega
\end{equation}
for each $\xi \in C(\Lambda^{\Dim-k})$.
We canonically write $T_{\omega}$ for this map
\begin{align*}
  \omega \mapsto T_\omega
\end{align*}
from $k$-forms to $(\Dim-k)$-currents and reserve $R,S$ and the unsubscripted $T$ for various other currents that arise in our analysis.
The particular case of 1-currents induced by $(\Dim-1)$-forms will play an important role in our analysis.

Under this identification of forms as currents, the boundary operator corresponds to the exterior derivative.
\begin{proposition}\label{BoundaryTomega}
Let $\omega\in C^1(\Lambda^k)$.
Then
  \begin{equation*}
    \boundary_{\Dim-k} T_\omega = (-1)^{\Dim-k} T_{\ed_k\omega}
    .
  \end{equation*}
In particular, a $k$-form $\omega$ is closed if and only if the associated $(\Dim-k)$-current $T_\omega$ has $\boundary_{\Dim-k} T_\omega=0$.
\end{proposition}
\begin{proof}
  For all $(\Dim-k-1)$-forms $\zeta\in C^1(\Lambda^{\Dim-k-1})$, Stokes' theorem and $\ed_{\Dim-1}(\zeta\wedge\omega)=\ed_{\Dim-k-1}\zeta\wedge\omega + (-1)^{\Dim-k-1}\zeta\wedge \ed_k\omega$ give
  \begin{equation*}
    \boundary_{\Dim-k} T_\omega(\zeta) = \int_\manifold \ed_{\Dim-k-1}\zeta\wedge\omega = (-1)^{\Dim-k} \int_\manifold \zeta\wedge\ed_k\omega = (-1)^{\Dim-k} T_{\ed_k\omega}(\zeta)
    .
    \qedhere
  \end{equation*}
\end{proof}
\begin{remark}\label{AlternativeCurrentsFromForms}
  As an alternative approach to producing currents from forms, we could have started with a $j$-form $\eta\in\Lambda^j$ and constructed the $j$-current $\tilde{T}_\eta(\xi) = \Angles{\xi,\eta}$ for $\xi\in\Lambda^j$.
  This construction has the additional benefit that it does not rely on orientability of the manifold.
  If the manifold is orientable, then $\tilde{T}_\eta$ reduces to $T_\omega$ via Hodge duality, under the correspondence $k=\Dim-j$, $\omega=\star_j\eta$, where $\star_j\colon \Lambda^j\to\Lambda^{\Dim-j}$ is the Hodge star operator on $j$-forms as in \Cref{appendix_A}.

  In that approach, the analogue of \Cref{BoundaryTomega} is that $\boundary_j \tilde{T}_\eta = \tilde{T}_{\ed_j^*\eta}$.
  However, the co-differential is slightly awkward for our purposes, particularly in \Cref{s:d-1FormsTokForms}, where an identity between forms of different degree will be more conveniently expressed in terms of differentials of $(\Dim-1)$-forms, rather than co-differentials of 1-forms: see \Cref{dPhiRelation}.
  Meanwhile, the orientability assumption is mild and can be circumvented by a short argument, see \Cref{ss:NonorientableProof}.
\end{remark}

The identity
\begin{equation}\label{FormCurrentLaplaciansCompatible}
  \int_\manifold \Delta_{\Dim-k}\eta \wedge \omega = \int_\manifold \eta\wedge \Delta_k\omega
\end{equation}
for forms implies the following result for the currents $T_\omega$.
\begin{proposition}\label{TomegaPropagatorsCompatible}
  For all $\omega\in C^2(\Lambda^k)$ and $\eta\in C^2(\Lambda^{\Dim-k})$,
  \begin{equation}\label{FormCurrentLaplaciansCompatible2}
    \Delta_{\Dim-k,c}T_\omega(\eta) = T_{\Delta_k\omega}(\eta)
    ,
  \end{equation}
  and consequently
  \begin{equation}
    e^{t\Delta_{\Dim-k,c}}T_\omega = T_{e^{t\Delta_k}\omega}
    .
  \end{equation}
\end{proposition}
\noindent
We defer the proofs to \Cref{appendix_A}.

\Cref{TomegaPropagatorsCompatible} means that the Laplacians for forms and currents, and therefore the corresponding heat propagators, are compatible.
In particular, the diagrams
\[\begin{tikzcd}
	{C^\infty(\Lambda^k)} & {\currents_{n-k}} \\
	{C^\infty(\Lambda^{k+1})} & {\currents_{n-k-1}}
	\arrow[from=1-1, to=1-2]
	\arrow["{\ed_k}"', from=1-1, to=2-1]
	\arrow["{(-1)^{\Dim-k}\partial_{\Dim-k}}"', from=1-2, to=2-2]
	\arrow[from=2-1, to=2-2]
\end{tikzcd}
\qquad\qquad\begin{tikzcd}
	{C^\infty(\Lambda^k)} & {\currents_{n-k}} \\
	{C^\infty(\Lambda^{k})} & {\currents_{n-k}}
	\arrow[from=1-1, to=1-2]
	\arrow["{e^{t\Delta_k}}"', from=1-1, to=2-1]
	\arrow["{e^{t\Delta_{\Dim-k,c}}}"', from=1-2, to=2-2]
	\arrow[from=2-1, to=2-2]
\end{tikzcd}\]
commute, where the horizontal arrows correspond to the map $\omega \mapsto T_\omega$.

The next result shows that we can recognize a current $T$ as having the form $T=T_\omega$ by duality arguments.

\begin{proposition}[Duality for form and current norms]\label{NormDualityCurrentsForms}
  Let $r,r'\in[1,\infty]$ satisfy $1/r+1/r'=1$, and let $\omega\in\Lambda^k$ denote a $k$-form.
  \begin{enumerate}[label=(\alph*),ref=(\alph*)]
    \item\label{item:FormLebesgueDuality}
    \begin{equation*}
      \begin{gathered}
        \norm{\omega}_{L^r} = \sup_{\tilde{\omega}\in\Lambda^k\colon \norm{\tilde{\omega}}_{L^{r'}} \leq 1} \Angles{\omega,\tilde{\omega}} = \sup_{\xi\in C(\Lambda^{\Dim-k})\colon \norm{\xi}_{L^{r'}}\leq 1} \abs{T_\omega(\xi)}
        \quad\text{and}
        \\
        \norm{\omega}_{L^{r,1}} \asymp \sup_{\tilde{\omega}\in\Lambda^k\colon \norm{\tilde{\omega}}_{L^{r',\infty}} \leq 1} \Angles{\omega,\tilde{\omega}} = \sup_{\xi\in C(\Lambda^{\Dim-k})\colon \norm{\xi}_{L^{r',\infty}}\leq 1} \abs{T_\omega(\xi)}
        ,
      \end{gathered}
    \end{equation*}
    where $a\asymp b$ means $a\lesssim b$ and $b\lesssim a$.
    In particular, $\totalmass(T_\omega) = \norm{\omega}_{L^1}$.

    \item\label{item:CurrentDuality}
    Let $T$ be an $(\Dim-k)$-current, let $\xi\in C(\Lambda^{\Dim-k})$ denote a continuous $(\Dim-k)$-form, and suppose that $r>1$.
    If
    \begin{equation*}
      \sup_{\norm{\xi}_{L^{r'}}\leq 1}\abs{T(\xi)} < \infty
      \quad\text{or}\quad
      \sup_{\norm{\xi}_{L^{r',\infty}}\leq 1}\abs{T(\xi)} < \infty
      ,
    \end{equation*}
    then there exists a $k$-form $\omega$, unique up to equality a.e.\ on $\manifold$, such that $T=T_\omega$.
    Moreover $\omega\in L^r(\Lambda^k)$ or $\omega\in L^{r,1}(\Lambda^k)$, respectively.
  \end{enumerate}
\end{proposition}
\noindent
We defer the proof to \Cref{appendix_A}.

In view of \Cref{NormDualityCurrentsForms}, we define Lebesgue and Lorentz norms on currents by
\begin{equation}\label{TomegaNormIs_omegaNorm}
  \norm{T}_{L^r} = \norm{\omega}_{L^r}, \quad \norm{T}_{L^{r,1}}=\norm{\omega}_{L^{r,1}} \quad\text{if }T=T_\omega
  .
\end{equation}
In particular, the \InterpolationLemma\ applies without change when $F$ is a current of the form $F=T_\omega$.

We will use the definitions \eqref{TomegaNormIs_omegaNorm} only in contexts where we know that the current $T$ has the form $T=T_\omega$.
Notably, as the next result shows, we can apply these definitions when $T$ results from applying the heat propagator to an arbitrary current.

\begin{lemma}\label{CurrentEvolutionAsForm}
  Given a $j$-current $T$ and $t\in (0,1]$, there is a smooth $(\Dim-j)$-form $\omega_t$ such that $e^{t\Delta_{j,c}}T=T_{\omega_t}$ and
  \begin{equation}\label{CurrentEvolutionPointwise}
    \norm{\omega_t(p)}_p \leq C\int_\manifold K_t(p,q)\,\dd\mu_T(q)
    \quad\text{for all }p\in\manifold
    ,
  \end{equation}
  where $K_t\colon\manifold\times\manifold\to\cointerval{0,\infty}$ is the scalar function
  \begin{equation}\label{KFormula}
    K_t(p,q) = t^{-\Dim/2}\exp\left( -d_\manifold(p,q)^2/(4t) \right)+1
    .
  \end{equation}
\end{lemma}
\Cref{CurrentEvolutionAsForm} shows that the heat propagator, which we know to be smoothing on forms, is also smoothing on currents.
The proof follows easily using \Cref{NormDualityCurrentsForms} and heat kernel estimates for the heat propagator on forms; we defer the details to \Cref{appendix_A}.

For later use, we collect several simple bounds for $K_t$ and related quantities.
\begin{lemma}\label{KernelBounds}
  Uniformly over $0<t\leq 1$,
  \begin{gather}
    \label{KLinftyBound}
    \sup_{p,q\in\manifold} K_t(p,q) \lesssim t^{-\Dim/2}
    ,
    \\
    \label{KdotqL1Bound}
    \sup_{q\in\manifold} \norm{K_t(\cdot, q)}_{L^1} \lesssim 1
    ,
    \quad
    \sup_{p\in\manifold} \norm{K_t(p, \cdot)}_{L^1} \lesssim 1
    ,
    \\
    \label{dplustTimesHeatKernelL1Bound}
    \sup_{q\in\manifold} \norm{\frac{d_\manifold(\cdot,q) + t}{t^{\Dim/2}} \exp(-d_\manifold(\cdot,q)^2/(4t))}_{L^1}
    \lesssim t^{1/2}
    ,
    \\
    \label{tildeKpdotL1Bound}
    \sup_{p\in\manifold} \norm{ \frac{d_\manifold(p,\cdot) + t}{t^{(\Dim+2)/2}} \exp\left( -d_\manifold(p,\cdot)^2/(4t) \right)+1}_{L^1} \lesssim t^{-1/2}
    .
  \end{gather}
\end{lemma}
\begin{proof}
  The first bound \eqref{KLinftyBound} is immediate.
  The other bounds all follow by estimating the integrals with a dyadic expansion over annuli such that $d_\manifold(\cdot,q) \approx \sqrt{t} 2^j$, using the fact that $V(\set{p\colon d_\manifold(\cdot,q)\leq r})\lesssim r^\Dim$ and merging higher powers of $t$ into lower ones since $0<t\leq 1$.
\end{proof}

\section{Closed \texorpdfstring{$(\Dim-1)$}{(\Dim-1)}-forms via closed curves}\label{s:Closedd-1FormProof}

In this section, we prove the following small-$t$ analogue of \Cref{BesovClosedCoclosed} for $(\Dim-1)$-forms.
\begin{theorem}\label{int01Bound_dF=0}
  Let $\manifold$ be a compact manifold of dimension $\Dim\geq 2$ and let $\alpha \in (0,\Dim)$.
  There exist a constant $C=C(\alpha,\manifold)>0$ such that
  \begin{align}\label{Smallt_heat_kernel}
    \int_0^1 t^{\alpha/2-1} \norm{e^{t\Delta_{\Dim-1}} F}_{L^{\Dim/(\Dim-\alpha),1}(\Lambda^{\Dim-1})} \,\dd t \leq C \|F\|_{L^1(\Lambda^{\Dim-1})}
  \end{align}
  for all $F \in L^1(\Lambda^{\Dim-1})$ such that $\ed_{\Dim-1} F=0$.
\end{theorem}
The proof is based on the 1-current $T_F$ arising from the $(\Dim-1)$-form $F$, where the assumption $\ed_{\Dim-1} F=0$ implies that $\boundary_1 T_F=0$.
As discussed in \Cref{ss:CurrentsFromForms}, the construction of the 1-current $T_F$ requires the manifold to be orientable, so our proof of \Cref{int01Bound_dF=0} will be subject to that assumption for now.
A simple argument circumvents this orientability assumption, and since we will repeat this argument several times we defer the details until \Cref{ss:NonorientableProof}.

Otherwise, the argument proceeds solely in terms of 1-currents, and does not rely on orientability.
We will prove the following extension of \Cref{int01Bound_dF=0} to $1$-currents without boundary.
\begin{theorem}\label{int01Bound0Boundary}
  Let $\manifold$ be a compact manifold of dimension $\Dim\geq 2$ and let $\alpha \in (0,\Dim)$.
  There exists a constant $C=C(\alpha,\manifold)<\infty$ such that
  \begin{align}
    \int_0^1 t^{\alpha/2-1} \norm{e^{t\Delta_{1,c}} T}_{L^{\Dim/(\Dim-\alpha),1}} \,\dd t \leq C \, \totalmass(T)
  \end{align}
  for all 1-currents $T$ such that $\boundary_1 T=0$.
\end{theorem}

We begin with two simple estimates for the heat propagator on currents, measured in $L^1$ and in $L^\infty$, that follow from the pointwise bounds in \Cref{CurrentEvolutionAsForm}.
\begin{lemma}\label{etDeltaT_L1L1}
  For all $1$-currents $T$ and for all $0<t\leq 1$,
  \begin{equation*}
    \norm{ e^{t\Delta_{1,c}} T }_{L^1} \lesssim  \totalmass(T)
    .
  \end{equation*}
\end{lemma}
\begin{lemma}\label{etDeltaTLinftyL1}
  For all $1$-currents $T$ and for all $0<t\leq 1$,
  \begin{equation*}
    \norm{ e^{t\Delta_{1,c}} T }_{L^\infty} \lesssim t^{-\Dim/2}\totalmass(T)
    .
  \end{equation*}
\end{lemma}

\begin{proof}[Proofs]
  We recall the identification $e^{t\Delta_{1,c}} T=T_{\omega_t}$ established in \Cref{CurrentEvolutionAsForm}.
  In particular, recalling \eqref{TomegaNormIs_omegaNorm} and integrating \eqref{CurrentEvolutionPointwise} over the manifold, we have
  \begin{align*}
    \norm{ e^{t\Delta_{1,c}} T }_{L^1} &= \norm{\omega_t}_{L^1} = \int_{\manifold}
    \norm{\omega_t(p)}_p \, \dd V(p) \\
    &\leq C\int_{\manifold} \int_\manifold K_t(p,q)\,\dd\mu_T(q) \, \dd V(p)\\
    &\leq C \sup_{q \in \manifold} \norm{K_t(\cdot,q)}_{L^1}\totalmass(T)\\
    &\lesssim \totalmass(T)
    ,
  \end{align*}
  where the last step used \eqref{KdotqL1Bound}.
  Similarly, taking the supremum instead as in \eqref{KLinftyBound},
  \begin{align*}
    \norm{ e^{t\Delta_{1,c}} T }_{L^{\infty}} &= \norm{\omega_t}_{L^\infty} = \sup_{p \in \manifold}
    \norm{\omega_t(p)}_p  \\
    &\leq C \sup_{p \in \manifold}\int_\manifold K_t(p,q)\,\dd\mu_T(q)\\
    &\leq C \sup_{p,q \in \manifold} K_t(p,q) \, \mu_T(\manifold)\\
    &\lesssim t^{-\Dim/2}\totalmass(T)
    .
    \qedhere
  \end{align*}
\end{proof}
\noindent
Interpolating the bounds in \Cref{etDeltaT_L1L1,etDeltaTLinftyL1} gives an integral that just barely fails to converge: see also the discussion in \cite{HRS}*{p.~1926}.
To make progress, we will need two complementary bounds before we can interpolate usefully, based on the specific features of the 1-currents that arise in the atomic decomposition from \Cref{AtomicDecompositionCited}.
The first bound will use additional regularity in the form of the Morrey norm, and the second bound will be specific to the case of small closed curves.

The following general lemma shows how the Morrey norm can be used to bound distance-based integrals.
\begin{lemma}\label{MorreyBoundForDistanceIntegral}
  Let $\mu$ be a measure on $\manifold$, let $f\colon\manifold\to\R$ be measurable, and suppose that there exists $p\in\manifold$ and a differentiable, non-increasing function $g\colon\cointerval{0,\infty}\to\cointerval{0,\infty}$ such that $\lim_{r\to\infty}g(r)=0$ and $\abs{f(q)}\leq g(d_\manifold(p,q))$ for all $q\in\manifold$.
  Then
  \begin{equation}
    \int_\manifold \abs{f(q)}\dd\mu(q) \leq \Morrey{\mu} \int_0^\infty r (-g'(r))\,\dd r
    .
  \end{equation}
\end{lemma}
\begin{proof}
  Let $h(r)=-g'(r)$ and note that the assumptions imply
  \begin{equation*}
    g(\rho)=\int_\rho^\infty h(r)\,\dd r=\int_0^\infty h(r)\indicator{r>\rho}\,\dd r
  \end{equation*}
  for all $\rho\in\cointerval{0,\infty}$.
  Interchange the order of integration to find
  \begin{align*}
    \int_\manifold \abs{f(q)} \dd\mu(q) &\leq \int_\manifold g(d_\manifold(p,q)) \, \dd\mu(q)
    \\&
    = \int_0^\infty \int_\manifold h(r)\indicator{r> d_\manifold(p,q)} \, \dd\mu(q) \, \dd r
    \\&
    = \int_0^\infty h(r) \mu(\set{q\in\manifold\colon d_\manifold(p,q)< r}) \, \dd r
    \\&
    \leq \int_0^\infty h(r) (r\Morrey{\mu}) \, \dd r
    .
    \qedhere
  \end{align*}
\end{proof}

\begin{lemma}\label{etDeltaT_LinftyMorrey}
  For all $1$-currents $T$ and for all $0<t\leq 1$,
  \begin{equation*}
    \norm{ e^{t\Delta_{1,c}} T }_{L^{\infty}} \lesssim t^{-(\Dim-1)/2} \Morrey{T}
    .
  \end{equation*}
\end{lemma}

\begin{proof}
  By \Cref{CurrentEvolutionAsForm},
  \begin{align*}
    \norm{ e^{t\Delta_{1,c}} T }_{L^{\infty}} &= \sup_{p\in\manifold} \norm{\omega_t}_p
    \\&
    \leq C\int_\manifold (t^{-\Dim/2}\exp\left( -d_\manifold(p,q)^2/(4t) \right)+1) \,\dd\mu_T(q)
    .
  \end{align*}
  The constant term in the integrand gives a contribution $C\mu_T(\manifold)$, which we can bound in terms of the Morrey norm by compactness:
  \begin{equation*}
    \mu_T(\manifold)\leq \operatorname{diam}(\manifold)\Morrey{\mu_T}
    .
  \end{equation*}
  For the rest of the integral, we can apply the special case of \Cref{MorreyBoundForDistanceIntegral} in which $g(r)=t^{-\Dim/2}e^{-r^2/(4t)}$.
  Then $h(r)=\frac{1}{2}rt^{-(\Dim+2)/2}e^{-r^2/(4t)}$ and we find $\int_0^\infty rh(r)\,\dd r=t^{-(\Dim-1)/2}\sqrt{\pi}$.
  Thus
  \begin{align*}
    \norm{ e^{t\Delta_{1,c}} T }_{L^{\infty}} &\leq C \Morrey{\mu_T} t^{-(\Dim-1)/2}\sqrt{\pi} + C \operatorname{diam}(\manifold)\Morrey{\mu_T}
    .
  \end{align*}
  By definition $\Morrey{\mu_T}=\Morrey{T}$, and the second term can be merged into the first because $t\leq 1$.
\end{proof}

The upper bound in \Cref{etDeltaT_LinftyMorrey} is less singular than the bound in \Cref{etDeltaTLinftyL1} when $t$ is small, so interpolating will now yield a convergent integral.

\begin{corollary}\label{etDeltaT_LorentzInterpolated}
  For all $1$-currents $T$ and for all $0<t\leq 1$,
  \begin{equation*}
    \norm{ e^{t\Delta_{1,c}} T}_{L^{\Dim/(\Dim-\alpha),1}} \lesssim t^{-\alpha/2+\alpha/(2\Dim)} \totalmass(T)^{1-\alpha/\Dim} \Morrey{T}^{\alpha/\Dim}
    .
  \end{equation*}
\end{corollary}
\begin{corollary}\label{etDeltaT_LorentzIntegrated}
  For all $1$-currents $T$,
  \begin{align*}
    \int_0^1 t^{\alpha/2-1}\|  e^{t\Delta_{1,c}} T\|_{L^{\Dim/(\Dim-\alpha),1}} \, \dd t \lesssim \totalmass(T)^{1-\alpha/\Dim} \Morrey{T}^{\alpha/\Dim}
    .
  \end{align*}
\end{corollary}

\begin{proof}[Proofs]
By the \InterpolationLemma,
  \begin{align*}
    \norm{e^{t\Delta_{1,c}} T}_{L^{\Dim/(\Dim-\alpha),1}} \leq C \norm{e^{t\Delta_{1,c}} T}_{L^1}^{1-\alpha/\Dim} \norm{e^{t\Delta_{1,c}} T}_{L^\infty}^{\alpha/\Dim}
    .
  \end{align*}
The upper bound from \Cref{etDeltaT_LorentzInterpolated} then follows by using the $L^1$ bound from \Cref{etDeltaT_L1L1} and the $L^\infty$ bound from \Cref{etDeltaT_LinftyMorrey}.
Integrating as in \Cref{etDeltaT_LorentzIntegrated} then leads to a factor $t^{\alpha/(2\Dim)-1}$, which ensures that the integral is finite.
\end{proof}

If we apply \Cref{etDeltaT_LorentzIntegrated} to a current of the form $T=[[\gamma]]$, we obtain an upper bound in terms of $\length(\gamma)^{1-\alpha/\Dim}$, where the atomic decomposition allows us to assume that the factor $\Morrey{[[\gamma]]}^{\alpha/\Dim}$ is uniformly bounded.
However, even this uniform estimate is not sufficient for our purposes: it fails to offset the normalizing factor $\length(\gamma)^{-1}$ that appears in the atomic decomposition \eqref{limit}.
To proceed, we will need an upper bound of the form $O(\length(\gamma))$.

We will obtain such a bound for sufficiently small $\gamma$ by considering $[[\gamma]]$ as the boundary of a 2-current.
To that end, we will need the following $L^\infty$ bound, similar in spirit to \Cref{CurrentEvolutionAsForm}.

\begin{lemma}\label{d1etDeltaxiAsIntegral}
  For a smooth $1$-form $\xi$ and $t\in\ocinterval{0,1}$, write $\zeta_t=\ed_1(e^{t\Delta_1}\xi)$, i.e., $\zeta_t=\ed_1\xi_t$ where $\xi_t=e^{t\Delta_1}\xi$.
  Then
  \begin{equation}
    \norm{\zeta_t(p)}_p \leq C\int_\manifold \tilde{K}_t(p,q) \norm{\xi(q)}_q \,\dd V(q),
  \end{equation}
  where
  \begin{equation}\label{tildeKtpq}
    \tilde{K}_t(p,q) = t^{-(\Dim+2)/2} \left(d_\manifold(p,q)  +t\right) \exp\left( -d_\manifold(p,q)^2/(4t) \right)+1
    .
  \end{equation}
\end{lemma}

We defer the proof to \Cref{appendix_A}.

\begin{lemma}\label{etDeltaSmallgammaL1}
  Let $\length_0$ be the constant from \Cref{SmallCurvesAreBoundaries}.
  Then, uniformly over all closed piecewise smooth curves $\gamma$ in $\manifold$ of length $\length(\gamma)\leq\length_0$ and over all $0<t\leq 1$,
  \begin{equation*}
    \norm{ e^{t\Delta_{1,c}} [[\gamma]] }_{L^1} \lesssim t^{-1/2} \length(\gamma)^2
    .
  \end{equation*}
\end{lemma}

\begin{proof}
By \Cref{NormDualityCurrentsForms}, we have
\begin{align*}
   \norm{ e^{t\Delta_{1,c}} [[\gamma]] }_{L^1} &= \sup_{\xi\in C(\Lambda^1), \|\xi\|_{L^\infty}\leq 1} e^{t\Delta_{1,c}} [[\gamma]](\xi)\\
   &= \sup_{\xi\in C(\Lambda^1), \|\xi\|_{L^\infty}\leq 1} [[\gamma]](e^{t\Delta_{1}} \xi)
   \end{align*}
By the smallness assumption on the curve, we may apply \Cref{SmallCurvesAreBoundaries} to obtain
\begin{align*}
 \abs{[[\gamma]](e^{t\Delta_{1}} \xi)} = \abs{S_\gamma(\ed_1 e^{t\Delta_{1}} \xi)} \leq C\length(\gamma)^2\norm{\ed_1 e^{t\Delta_{1}} \xi}_{L^\infty},
 \end{align*}
and then \Cref{d1etDeltaxiAsIntegral}, the assumption $\norm{\xi}_{L^\infty}\leq 1$, and the bound \eqref{tildeKpdotL1Bound} give
\begin{align*}
   \norm{ e^{t\Delta_{1,c}} [[\gamma]] }_{L^1} &\lesssim \length(\gamma)^2 \sup_{p \in \manifold} \int_\manifold \tilde{K}_t(p,q)  \,\dd V(q)\\
   &\lesssim t^{-1/2}\length(\gamma)^2
   .
   \qedhere
\end{align*}
\end{proof}

\begin{corollary}\label{etDeltagammaMediumt}
  Under the assumptions of \Cref{etDeltaSmallgammaL1},
  \begin{equation*}
    \norm{ e^{t\Delta_{1,c}} [[\gamma]] }_{L^{\Dim/(\Dim-\alpha),1}} \lesssim t^{-\alpha/2}t^{-(\Dim-\alpha)/(2\Dim)} \length(\gamma)^{2-\alpha/\Dim}
    .
  \end{equation*}
\end{corollary}

\begin{proof}
As in the proof of \Cref{etDeltaT_LorentzInterpolated}, the \InterpolationLemma\ gives
  \begin{align*}
    \norm{e^{t\Delta_{1,c}} [[\gamma]] }_{L^{\Dim/(\Dim-\alpha),1}} \leq C \norm{e^{t\Delta_{1,c}} [[\gamma]] }_{L^1}^{1-\alpha/\Dim} \norm{e^{t\Delta_{1,c}} [[\gamma]] }_{L^\infty}^{\alpha/\Dim}
    .
  \end{align*}
Bound the $L^1$ norm using \Cref{etDeltaSmallgammaL1} and bound the $L^\infty$ norm using \Cref{etDeltaTLinftyL1} with $T=[[\gamma]]$, recalling that $\totalmass([[\gamma]])\leq\length(\gamma)$.
\end{proof}

\begin{corollary}\label{LorentzBoundetDeltagamma}
  Let $\length_0$ be the constant from \Cref{SmallCurvesAreBoundaries}.
  Then there exists a constant $C<\infty$ such that, for all curves $\gamma$ of length $\length(\gamma)\leq\length_0$ and all $0<t\leq 1$,
  \begin{equation*}
    \norm{e^{t\Delta_{1,c}}[[\gamma]]}_{L^{\Dim/(\Dim-\alpha),1}} \leq \frac{C\length(\gamma)}{t^{\alpha/2}} \min\left( \left(\frac{t \Morrey{\gamma}^2}{\length(\gamma)^2}\right)^{\alpha/(2\Dim)} , \left(\frac{t}{\length(\gamma)^2}\right)^{-(\Dim-\alpha)/(2\Dim)} \right)
    .
  \end{equation*}
\end{corollary}

\Cref{LorentzBoundetDeltagamma} follows immediately from \Cref{etDeltaT_LorentzInterpolated,etDeltagammaMediumt}.
Integrating now leads to the desired bound of the form $O(\length(\gamma))$.

\begin{corollary}\label{LorentzBoundetDeltagammaIntegrated}
  Let $\length_0$ be the constant from \Cref{SmallCurvesAreBoundaries}.
  Then there exists a constant $C<\infty$ such that, for all curves $\gamma$ of length $\length(\gamma)\leq\length_0$,
  \begin{equation*}
    \int_0^1 t^{\alpha/2 - 1} \norm{e^{t\Delta_{1,c}}[[\gamma]]}_{L^{\Dim/(\Dim-\alpha),1}} \dd t \leq C \length(\gamma) \Morrey{\gamma}^{\alpha(\Dim-\alpha)/\Dim^2}
    .
  \end{equation*}
\end{corollary}

\begin{proof}
  \Cref{LorentzBoundetDeltagamma} indicates a change in scaling behavior at time scales of order $\length(\gamma)^2$.
  With this in mind, make the substitutions $s=t/\length(\gamma)^2$ and later $u=s \Morrey{\gamma}^{2\alpha/\Dim}$, replacing the upper limit of integration by $\infty$, to find
  \begin{align*}
    &\int_0^1 t^{\alpha/2 - 1} \norm{e^{t\Delta_{\Dim-1}}[[\gamma]]}_{L^{\Dim/(\Dim-\alpha),1}} \dd t
    \\&\quad
    \leq C \length(\gamma) \int_0^1 \min\left( \left(\frac{t \Morrey{\gamma}^2}{\length(\gamma)^2}\right)^{\alpha/(2\Dim)} , \left(\frac{t}{\length(\gamma)^2}\right)^{-(\Dim-\alpha)/(2\Dim)} \right) \frac{\dd t}{t}
    \\&\quad
    \leq C \length(\gamma) \int_0^\infty \min\left( \left(s\Morrey{\gamma}^2\right)^{\alpha/(2\Dim)} , s^{-(\Dim-\alpha)/(2\Dim)} \right) \frac{\dd s}{s}
    \\&\quad
    = C \length(\gamma) \int_0^\infty \min\left( \left(u\Morrey{\gamma}^{2(\Dim-\alpha)/\Dim}\right)^{\alpha/(2\Dim)} , u^{-(\Dim-\alpha)/(2\Dim)}\Morrey{\gamma}^{\alpha(\Dim-\alpha)/\Dim^2} \right) \frac{\dd u}{u}
    \\&\quad
    = C \length(\gamma) \Morrey{\gamma}^{\alpha(\Dim-\alpha)/\Dim^2} \left( \int_0^1 u^{\alpha/(2\Dim)-1}\,\dd u + \int_1^\infty u^{-(1+(\Dim-\alpha)/(2\Dim))}\,\dd u \right)
    \\&\quad
    = C \length(\gamma) \Morrey{\gamma}^{\alpha(\Dim-\alpha)/\Dim^2} \left( \frac{2\Dim}{\alpha} + \frac{2\Dim}{\Dim-\alpha} \right)
    .
    \qedhere
  \end{align*}
\end{proof}

We will apply the uniform bound from \Cref{LorentzBoundetDeltagammaIntegrated} to the 1-currents $[[\gamma_{i,n}]]$ from the atomic decomposition quoted in \Cref{AtomicDecompositionCited}.
Before we proceed to the proof of \Cref{int01Bound0Boundary}, we state a technical result showing that norm upper bounds are preserved under weak-star convergence of currents.

\begin{lemma}\label{FatouForCurrentNorms}
  Suppose that $j$-currents $S_m$, $m\in\mathbb{N}$, converge weakly-star to $S$, i.e., $S_m(\xi)\to S(\xi)$ as $m\to\infty$ for all $\xi\in C(\Lambda^j)$.
  Then
  \begin{equation}\label{FatouCurrentNorms}
    \norm{S}_{L^{\Dim/(\Dim-\alpha),1}} \leq \liminf_{m\to\infty} \norm{S_m}_{L^{\Dim/(\Dim-\alpha),1}}
    .
  \end{equation}
  Moreover for all $t\geq 0$, $e^{t\Delta_{j,c}} S_m \to e^{t\Delta_{j,c}} S$ as $m\to\infty$, and in particular,
  \begin{equation}\label{WeakPropagated}
    \norm{e^{t\Delta_{j,c}}S}_{L^{\Dim/(\Dim-\alpha),1}} \leq \liminf_{m\to\infty} \norm{e^{t\Delta_{j,c}}S_m}_{L^{\Dim/(\Dim-\alpha),1}}
    .
  \end{equation}
\end{lemma}
\begin{proof}
  Let $\xi\in C(\Lambda^j)$ with $\norm{\xi}_{L^{\Dim/\alpha,1}}\leq 1$.
  Then using \Cref{NormDualityCurrentsForms},
  \begin{equation*}
    \abs{S(\xi)} = \lim_{m\to\infty}\abs{S_m(\xi)} \leq \liminf_{m\to\infty} \norm{S_m}_{L^{\Dim/(\Dim-\alpha),1}}
  \end{equation*}
  and taking the supremum over such $\xi$ shows \eqref{FatouCurrentNorms}.
  For \eqref{WeakPropagated}, let $\xi\in C(\Lambda^j)$ be arbitrary.
  Then also $e^{t\Delta_j}\xi \in C(\Lambda^j)$, so
  \begin{equation*}
    e^{t\Delta_{j,c}} S_m(\xi) = S_m(e^{t\Delta_j}\xi) \to S(e^{t\Delta_j}\xi) = e^{t\Delta_{j,c}} S(\xi)
  \end{equation*}
  verifies the weak-star convergence, and the last statement follows by \eqref{FatouCurrentNorms}.
\end{proof}

\begin{proof}[Proof of \Cref{int01Bound0Boundary}]
Let $T$ be a 1-current with $\boundary_1 T=0$, and let $\gamma_{i,m}$, $\lambda_{i,m}$, and $M$ be the curves, scalars, and universal constant from \Cref{AtomicDecompositionCited}, so that the 1-currents
\begin{equation}
  S_m = \sum_{i=1}^m \lambda_{i,m} \frac{[[\gamma_{i,m}]]}{\length(\gamma_{i,m})}
\end{equation}
satisfy $S_m\to T$ weakly-star as $m\to\infty$, by \eqref{limit}.
By the triangle inequality,
\begin{multline}
  \int_0^1 t^{\alpha/2-1} \norm{e^{t\Delta_{1,c}} S_m}_{L^{\Dim/(\Dim-\alpha),1}} \dd t
  \\
  \leq \sum_{i=1}^m \abs{\lambda_{i,m}} \length(\gamma_{i,m})^{-1} \int_0^1  t^{\alpha/2-1} \norm{e^{t\Delta_{1,c}}[[\gamma_{i,m}]]}_{L^{\Dim/(\Dim-\alpha),1}} \dd t
  .
\end{multline}
To control the effect of the factor $\length(\gamma_{i,m})^{-1}$, split the sum.
For short curves, \Cref{LorentzBoundetDeltagammaIntegrated} gives a uniform bound that is linear in $\length(\gamma_{i,m})$, whereas for long curves, the sublinear bound from \Cref{etDeltaT_LorentzIntegrated} is sufficient:
\begin{align}
  &\int_0^1 t^{\alpha/2-1} \norm{e^{t\Delta_{1,c}} S_m}_{L^{\Dim/(\Dim-\alpha),1}} \dd t
  \notag\\&\quad
  \leq \sum_{\substack{i\in\set{1,\dotsc,m}\colon \\ \length(\gamma_{i,m})\leq \length_0}} \abs{\lambda_{i,m}} \length(\gamma_{i,m})^{-1} [C \length(\gamma_{i,m}) M^{\alpha(\Dim-\alpha)/\Dim^2}]
  \notag\\&\qquad
  + \sum_{\substack{i\in\set{1,\dotsc,m}\colon \\ \length(\gamma_{i,m}) > \length_0}} \abs{\lambda_{i,m}} \length(\gamma_{i,m})^{-1} [C\length(\gamma_{i,m})^{1-\alpha/\Dim} M^{\alpha/\Dim}]
  &\text{by \eqref{morrey_bound}}
  \notag\\&\quad
  \leq C \sum_{i=1}^m \abs{\lambda_{i,m}} \left[ M^{\alpha(\Dim-\alpha)/\Dim^2} + (M/\length_0)^{\alpha/\Dim} \right]
  \notag\\&\quad
  \leq 2C \left[ M^{\alpha(\Dim-\alpha)/\Dim^2} + (M/\length_0)^{\alpha/\Dim} \right] \totalmass(T)
  &\text{by \eqref{mass}}
  .
\end{align}
Since this bound is uniform in $m$, the result follows from \Cref{FatouForCurrentNorms}.
\end{proof}

\begin{proof}[Proof of \Cref{int01Bound_dF=0} for orientable manifolds]
  Let $F\in C^1(\Lambda^{\Dim-1})$ satisfy $\ed_{\Dim-1} F=0$.
  Since the manifold is orientable, we may form the corresponding $1$-current $T_F$, which satisfies $\boundary_1 T_F=0$ by \Cref{BoundaryTomega}.
  Thus \Cref{int01Bound0Boundary} applies, and along with \Cref{TomegaPropagatorsCompatible,NormDualityCurrentsForms} gives
  \begin{align*}
  &\int_0^1 t^{\alpha/2-1} \norm{e^{t\Delta_{\Dim-1}} F}_{L^{\Dim/(\Dim-\alpha),1}(\Lambda^{\Dim-1})} \dd t
  \\&\quad
  = \int_0^1 t^{\alpha/2-1} \norm{e^{t\Delta_{1,c}} T_F}_{L^{\Dim/(\Dim-\alpha),1}(\Lambda^{\Dim-1})} \dd t
  \\&\quad
    \leq C \totalmass(T_F) = C\norm{F}_{L^1(\Lambda^{\Dim-1})}
  .
  \end{align*}
  Since this inequality holds for all $F\in C^1(\Lambda^{\Dim-1})$ satisfying $\ed_{\Dim-1}F = 0$, we may argue as in \cite{ISS} to extend the inequality to $F\in L^1(\Lambda^{\Dim-1})$ satisfying $\ed_{\Dim-1}F = 0$.
\end{proof}

We note that \Cref{int01Bound_dF=0} remains true even if the manifold is non-orientable, but we will defer the details until \Cref{ss:NonorientableProof}.

\section{From closed \texorpdfstring{$(\Dim-1)$}{(\Dim-1)}-forms to general \texorpdfstring{$(\Dim-1)$}{(\Dim-1)}-forms}\label{s:Closedd-1ToGenerald-1}

For a compact manifold, we may drop the assumption from \Cref{int01Bound_dF=0} that the $(\Dim-1)$-form is closed, instead inserting $\ed_{\Dim-1}F$ into the upper bound.
\begin{theorem}\label{int01Bound_General}
  Let $\manifold$ be a compact manifold of dimension $\Dim\geq 2$ and let $\alpha \in (0,\Dim)$.
  There exists a constant $C=C(\alpha,\manifold)<\infty$ such that
  \begin{align}\label{inhomogeneous_estimate_n-1_forms_notclosed}
    \int_0^1 t^{\alpha/2-1} \norm{e^{t\Delta_{\Dim-1}} F}_{L^{\Dim/(\Dim-\alpha),1}(\Lambda^{\Dim-1})} \,\dd t \leq C (\|F\|_{L^1(\Lambda^{\Dim-1})}+\|\ed_{\Dim-1} F\|_{L^1(\Lambda^\Dim)})
  \end{align}
  for all $F \in L^1(\Lambda^{\Dim-1})$.
\end{theorem}

We remark that the that the proof of \Cref{int01Bound_General} uses compactness in a stronger way than elsewhere in the paper.
Indeed, the upper bound in \eqref{inhomogeneous_estimate_n-1_forms_notclosed} breaks scale invariance, and we would expect such a bound to fail in the non-compact setting in general.
For this reason, proofs later in the paper will be structured so as to avoid relying on \Cref{int01Bound_General} where possible, even though \Cref{int01Bound_General} includes \Cref{int01Bound_dF=0}.

\begin{lemma}\label{CurrentFromptoq}
  Let $\manifold$ be a compact connected manifold.
  There exists a continuous function $g\colon\manifold\times\manifold\to \currents_1$, $(p,q)\mapsto g_{p,q}$, and a constant $C$ such that, for all $p,q\in\manifold$,
  \begin{equation}
    \boundary_1 g_{p,q}(f) = f(q) - f(p)
  \end{equation}
  for all $f\in C(\manifold)$, and
  \begin{equation}
    \totalmass(g_{p,q}) \leq C
    ,
    \quad
    \Morrey{g_{p,q}}\leq C
    .
  \end{equation}
\end{lemma}
\begin{proof}
  Let $(x_r,U_r)_{r=1}^\maxr$ be the finite atlas constructed in \Cref{ss:FiniteAtlas} and $(\chi_r)_{r=1}^\maxr$ be the partition of unity subordinate to $(U'_r)_{r=1}^\maxr$.
  Fix points $y_r\in U_r$, $r\in\set{1,\dotsc,\maxr}$.
  By connectedness, for all $r,r'\in\set{1,\dotsc,\maxr}$, we may fix geodesics $\tilde{\gamma}^{(r,r')}$ in $\manifold$ from $y_r$ to $y_{r'}$, possibly of length 0 if $y_r=y_{r'}$.
  For $p\in U_r, p'\in U_{r'}$, define $\gamma^{(r,r')}_{p,p'}$ to be the concatenation of the geodesics $\gamma_{p,y_r}^{(r)}$, $\tilde{\gamma}^{(r,r')}$, and $\gamma_{y_{r'},p'}^{(r')}$.
  Thus $\gamma^{(r,r')}_{p,p'}$ is a piecewise smooth curve from $p$ to $p'$.
  In particular,
  \begin{equation}
    \boundary_1[[\gamma^{(r,r')}_{p,p'}]](f) = f(p')-f(p)
  \end{equation}
  for all $f\in C(\manifold)$.
  Furthermore, since each curve $\gamma^{(r,r')}_{p,p'}$ is a concatenation of three geodesics from a compact metric space,
  \begin{equation}
    \length(\gamma^{(r,r')}_{p,p'}) \leq 3\operatorname{diam}(\manifold)
    ,
    \quad
    \Morrey{\big.\smash{\gamma^{(r,r')}_{p,p'}}} \leq 6
    ;
  \end{equation}
  see \cite{ChenGoodmanHernandezSpectorPreprint}*{Corollary~2.3} for the second assertion.
  It is readily verified that $[[\gamma_{p,y_r}^{(r)}]]$ and $[[\gamma_{y_{r'},p'}^{(r')}]]$, and hence $[[\gamma^{(r,r')}_{p,p'}]]$, depend continuously on $(p,p')\in (U_r,U_{r'})$.
  Finally we define
  \begin{equation}
    g_{p,q} = \sum_{r,r'=1}^\maxr \chi^{(r)}(p) \chi^{(r')}(q) [[\gamma^{(r,r')}_{p,q}]]
    ,
  \end{equation}
  where in the sum we can restrict to $r,r'$ for which $p\in U_r$, $q\in U_{r'}$, without loss of generality.
  The assertions then follow from linearity and the triangle inequality.
\end{proof}

We note that both the statement and proof of \Cref{CurrentFromptoq} rely crucially on the finite diameter of $\manifold$, whereas elsewhere we have used the compactness assumption in milder ways.

\begin{lemma}\label{Fd-1HasR}
  Let $\manifold$ be a compact orientable manifold, and let $F \in C^\infty(\Lambda^{\Dim-1})$.
  Then there exists a $1$-current $R_F$ such that
  \begin{align*}
    \boundary_1 (T_F-R_F)=0
  \end{align*}
  and
  \begin{align*}
    \totalmass(R_F) &\lesssim \|\ed F\|_{L^1(\Lambda^{\Dim})}
    ,
    \\
    \Morrey{R_F} &\lesssim \norm{\ed F}_{L^1(\Lambda^{\Dim})}
    .
  \end{align*}
\end{lemma}

\begin{proof}
  First suppose $\manifold$ is connected.
  Let $\nu\in\Lambda^\Dim$ denote the volume form of $\manifold$, and let $F\in C^\infty(\Lambda^{\Dim-1})$ be a given smooth $(\Dim-1)$-form.
  We may write the $\Dim$-form $\ed_{\Dim-1} F$ as $\ed_{\Dim-1} F = \phi \, \nu$, where $\phi=\star_\Dim \ed_{\Dim-1} F\in C^\infty(\manifold)$ is a smooth function.
  By Stokes' theorem,
  \begin{equation}
    \int_\manifold \ed_{\Dim-1} F = 0
    ,
    \quad\text{i.e.,}\quad
    \int_\manifold \phi(p)\,\dd V(p) = 0
    ,
  \end{equation}
  and consequently
  \begin{equation}
    \int_\manifold \phi^+(p)\,\dd V(p) = \int_\manifold \phi^-(p)\,\dd V(p) = \tfrac{1}{2} \norm{\phi}_{L^1} = \tfrac{1}{2}\norm{\ed_{\Dim-1} F}_{L^1}
    ,
  \end{equation}
  where $\phi^+$ and $\phi^-$ denote the positive and negative parts of $\phi$, respectively.

  If $\ed_{\Dim-1} F=0$ then $\boundary_1 T_F=0$ already and we may set $R_F=0$.
  Otherwise, define
  \begin{equation}
    K(p,q) = \frac{\phi^-(p)\phi^+(q)}{\frac{1}{2} \norm{\phi}_{L^1}}
    .
  \end{equation}
  Then $K$ is a transport kernel on $\manifold\times\manifold$ transporting $\phi^-$ to $\phi^+$ in the sense that $K(p,q)\geq 0$ for all $p,q\in\manifold$ and
  \begin{equation}
    \begin{aligned}
      \int_\manifold K(p,q)\,\dd V(q) &= \phi^-(p) &\text{for all }p\in\manifold
      ,
      \\
      \int_\manifold K(p,q)\,\dd V(p) &= \phi^+(q) &\text{for all }q\in\manifold
      .
    \end{aligned}
  \end{equation}
  Define
  \begin{equation}
    R_F = \iint_{\manifold\times\manifold} g_{p,q} K(p,q) \,\dd V(p)\,\dd V(q)
    .
  \end{equation}
  Then for $f\in C(\manifold)$,
  \begin{align*}
    \boundary_1 R_F(f) &= \iint_{\manifold\times\manifold} \boundary_1 g_{p,q}(f) K(p,q) \,\dd V(p)\,\dd V(q)
    \\&
    = \iint_{\manifold\times\manifold} [f(q)-f(p)] \frac{\phi^-(p)\phi^+(q)}{\frac{1}{2} \norm{\phi}_{L^1}} \,\dd V(p)\,\dd (q)
    \\&
    = \int_\manifold f(q)\phi^+(q) \frac{\norm{\phi^-}_{L^1}}{\frac{1}{2}\norm{\phi}_{L^1}} \,\dd V(q) - \int_\manifold f(p)\phi^-(p) \frac{\norm{\phi^+}_{L^1}}{\frac{1}{2}\norm{\phi}_{L^1}} \,\dd V(p)
    \\&
    = \int_\manifold f \, [\phi^- - \phi^+]\,\dd V
    \\&
    = -\int_\manifold f \phi \, \nu = -\int_\manifold f\,\ed_{\Dim-1} F = -T_{\ed_{\Dim-1} F}(f) = \boundary_1 T_F(f)
    ,
  \end{align*}
  where the last step used \Cref{BoundaryTomega}.
  The bounds on $\totalmass(R_F)$ and $\Morrey{R_F}$ follow from the triangle inequality, the uniform bounds on $\totalmass(g_{p,q})$ and $\Morrey{g_{p,q}}$, and the observation that
  \begin{align*}
    \iint_{\manifold\times\manifold} C K(p,q)\,\dd V(p)\,\dd V(q)
    &= C\frac{\norm{\phi^-}_{L^1} \norm{\phi^+}_{L^1}}{\frac{1}{2}\norm{\phi}_{L^1}}
    \\&
    = \tfrac{1}{2}C\norm{\phi}_{L_1} = \tfrac{1}{2}C\norm{\ed_1 F}_{L^1}
    .
  \end{align*}

  Finally if $\manifold$ is compact but not connected, then it consists of a finite number of connected components; apply the preceding construction to each component separately and take the sum.
\end{proof}

\begin{proof}[Proof of \Cref{int01Bound_General} for orientable manifolds]
  By \Cref{Fd-1HasR}, the 1-current $T_F-R_F$ satisfies $\boundary_1 (T_F-R_F)=0$ and
  \begin{align*}
    \totalmass(T_F-R_F) &\leq \totalmass(T_F) + \totalmass(R_F)
    \\&
    \lesssim \norm{F}_{L^1(\Lambda^{\Dim-1})} + \norm{\ed_{\Dim-1} F}_{L^1(\Lambda^\Dim)}
    .
  \end{align*}
  By \Cref{int01Bound0Boundary} we conclude that
  \begin{equation*}
    \int_0^1 t^{\alpha/2-1}\|  e^{t\Delta_{1,c}} (T_F-R_F)\|_{L^{\Dim/(\Dim-\alpha),1}(\Lambda^{\Dim-1})} \, \dd t \lesssim \|F\|_{L^1(\Lambda^{\Dim-1})}+\|\ed F\|_{L^1(\Lambda^\Dim)}
    .
  \end{equation*}
  On the other hand, \Cref{etDeltaT_LorentzIntegrated} implies that
  \begin{align*}
    \int_0^1 t^{\alpha/2-1}\|  e^{t\Delta_{1,c}} R_F\|_{L^{\Dim/(\Dim-\alpha),1}(\Lambda^{\Dim-1})} \, \dd t &\lesssim \totalmass(R_F)^{1-\alpha/\Dim} \Morrey{R_F}^{\alpha/\Dim}
    \\&
    \lesssim \norm{\ed F}_{L^1(\Lambda^{\Dim})}^{1-\alpha/\Dim+\alpha/\Dim}
  \end{align*}
  by the bounds in \Cref{Fd-1HasR}.
  Combining these two upper bounds,
  \begin{align*}
    \int_0^1 t^{\alpha/2-1}\|  e^{t\Delta_{1,c}} T_F \|_{L^{\Dim/(\Dim-\alpha),1}(\Lambda^{\Dim-1})} \, \dd t \lesssim \|F\|_{L^1(\Lambda^{\Dim-1})}+\|\ed F\|_{L^1(\Lambda^\Dim)}
  \end{align*}
  also.
  Then the identities
  \begin{align*}
    \norm{e^{t\Delta_{1,c}} T_F}_{L^{\Dim/(\Dim-\alpha),1}(\Lambda^{\Dim-1})}
    &= \norm{ T_{e^{t\Delta_{\Dim-1}}F} }_{L^{\Dim/(\Dim-\alpha),1}(\Lambda^{\Dim-1})}
    \\&
    =\norm{ e^{t\Delta_{\Dim-1}}F }_{L^{\Dim/(\Dim-\alpha),1}(\Lambda^{\Dim-1})}
  \end{align*}
  from \Cref{TomegaPropagatorsCompatible,NormDualityCurrentsForms} complete the proof.
\end{proof}

We note that \Cref{int01Bound_General} remains true even if the manifold is non-orientable, but we will defer the details until \Cref{ss:NonorientableProof}.

\section{From \texorpdfstring{$(\Dim-1)$}{(\Dim-1)}-forms to \texorpdfstring{$k$}{k}-forms}\label{s:d-1FormsTokForms}

In this section we extend \Cref{int01Bound_dF=0,int01Bound_General} from $(\Dim-1)$-forms to $k$-forms, $k=0,\dotsc,\Dim-1$.
\begin{theorem}\label{int01Bound_Closed_kform}
Let $\manifold$ be a compact manifold of dimension $\Dim\geq 2$ and let $\alpha \in (0,\Dim)$.
There exists a constant $C$ such that, for all $k=0,\dotsc,\Dim-1$ and for all $k$-forms $\omega$ such that $\ed_k\omega = 0$,
 \begin{equation*}
    \int_0^1 t^{\alpha/2-1} \norm{ e^{t\Delta_k} \omega }_{L^{\Dim/(\Dim-\alpha),1}(\Lambda^k)} \, \dd t \leq C \norm{\omega}_{L^1(\Lambda^k)}
    .
  \end{equation*}
\end{theorem}

\begin{theorem}\label{int01Bound_General_kform}
  Let $\manifold$ be a compact manifold of dimension $\Dim\geq 2$ and let $\alpha \in (0,\Dim)$.
  There exists a constant $C$ such that, for all $k=0,\dotsc,\Dim-1$,
  \begin{equation*}
    \int_0^1 t^{\alpha/2-1} \norm{ e^{t\Delta_k} \omega }_{L^{\Dim/(\Dim-\alpha),1}(\Lambda^k)} \, \dd t \leq C \left( \norm{\omega}_{L^1(\Lambda^k)}+\norm{\ed_k \omega}_{L^1(\Lambda^{k+1})} \right)
  \end{equation*}
  for all $\omega \in L^1(\Lambda^{k})$ such that $\ed_k \omega \in L^1(\Lambda^{k+1})$.
\end{theorem}

We note that \Cref{int01Bound_General_kform}, like \Cref{int01Bound_General}, relies strongly on compactness.
We will therefore give an independent proof of \Cref{int01Bound_Closed_kform}, even though it is a special case of \Cref{int01Bound_General_kform}.

The proofs of \Cref{int01Bound_Closed_kform,int01Bound_General_kform} are based on ``promoting'' a $k$-form to a vector of forms of higher degree.
This mapping will use coordinates, and to simplify the notation, we will abbreviate $[\Dim]=\set{1,\dotsc,\Dim}$, and we will use $J$ to denote a subset of $[\Dim]$ of cardinality $\abs{J}=j$.
For integers $m$ to be specified, we identify $(\Lambda^m)^{\binom{\Dim}{j}}$ with collection of vectors $\vec{\eta}=(\eta_J)_{\abs{J}=j}$ whose entries are $m$-forms $\eta_J\in\Lambda^m$ indexed by $J\subset [\Dim]$ with $\abs{J}=j$.
Later, we will similarly use $K$ to denote a subset of $[\Dim]$ of cardinality $\abs{K}=k$.

Recall from \Cref{ss:FiniteAtlas} the triples $(x_r,U_r,U'_r)_{r=1\dotsc,\maxr}$, the partition of unity $(\chi_r)_{r=1,\dotsc,\maxr}$ subordinate to $(U'_r)_{r=1\dotsc,\maxr}$, and the cutoff functions $\rho_r$ supported in $U_r$ with $\rho_r=1$ on a neighbourhood of $U'_r$.
Define the scalar functions
\begin{equation}
  y_r^i = \rho_r x_r^i
\end{equation}
for $i=1,\dotsc,\Dim$ and $r=1,\dotsc,\maxr$, where $x^1_r,\dotsc,x^\Dim_r$ denote the entries of the coordinate mapping $x_r\colon U_r\to\R^\Dim$.
By the choice of the cutoff function $\rho_r$, the functions $y_r^i$ may be zero-extended outside of $U_r$ and thus define smooth functions on $\manifold$.
For any subset $A\subset [\Dim]$, let $\ed y_r^A$ denote the wedge product of all 1-forms $\ed y^i_r$ over $i\in A$, arranged in increasing order.

For each $k,j\in\set{0,1,\dotsc,\Dim}$ and $r\in\set{1,\dotsc,\maxr}$, define the ``promotion'' mappings
\begin{equation}
  \begin{gathered}
    \phi_{k,j,r}\colon \Lambda^k \to (\Lambda^{\Dim-j+k})^{\binom{\Dim}{j}}
    ,
    \\
    \phi_{k,j,r}(\omega) = (\omega \wedge \ed y_r^{J^c})_{\abs{J}=j}
    .
  \end{gathered}
\end{equation}
These mappings can be defined for all $0\leq j\leq n$ but are identically zero if $j<k$, and we restrict our attention to the interval $k\leq j\leq n$.
The endpoint $j=\Dim$ gives the identity mapping.
The other endpoint, $j=k$, corresponds to writing $\omega$ in coordinates, $\omega=\sum_{\abs{K}=k}\omega_K \, \ed x_r^K$, except that each coordinate function $\omega_K$ is stored as an $\Dim$-form $\phi_{k,k,r,K}(\omega) = \omega_K \, \ed x_r^K \wedge \ed y_r^{K^c}$ rather than a scalar function.
Our interest will be in the case $j=k+1$ and the mappings $\phi_{k,k+1,r}$ and $\phi_{k+1,k+1,r}$.

As we will see, the mappings $\phi_{k,j,r}$ are locally injective when $k\leq j$.
To define the left inverse, define the smooth vector fields
\begin{equation}
  Y_{r,i} = \rho_r \frac{\partial}{\partial x_r^i}
\end{equation}
on $\manifold$.
When $A\subset[\Dim]$, abbreviate $(Y_r)_A$ for the vector fields $Y_{r,i}$ over $i\in A$, arranged in increasing order.
When $\zeta$ is an $m$-form and $X_1,\dotsc,X_m$ are vector fields, by $\zeta(X_1,\dotsc,X_m)$ we mean the scalar function $p\mapsto \zeta(p;X_1(p),\dotsc,X_m(p))$ on $\manifold$.
We note for later reference that by construction, the 1-forms $(\ed x_r^j)_{j\in[n]}$ and the vector fields $(\frac{\partial}{\partial x_r^i})_{i\in[n]}$ are dual bases with $\ed x_r^j(\frac{\partial}{\partial x_r^i}) = \indicator{i=j}$ on $U_r$.
Hence $(\ed y_r^j)_{j\in[n]}$ and $(Y_{r,i})_{i\in[n]}$ have the same property on $U'_r$, and more generally
\begin{equation}\label{yJcyA}
  \ed y_r^{J^c}((Y_r)_A) = \begin{cases}
    1 &\text{on $U'_r$ if $A=J^c$}
    ,
    \\
    0 &\text{on $U'_r$ otherwise}
    .
  \end{cases}
\end{equation}

Define $\psi_{k,j,r} \colon (\Lambda^{\Dim-j+k})^{\binom{\Dim}{j}} \to \Lambda^k$ by
\begin{equation}
  \psi_{k,j,r}(\vec{\eta})(X_1,\dotsc,X_k)
  = \frac{1}{\binom{\Dim-k}{j-k}} \sum_{\abs{J}=j} \eta_J \! \left(X_1,\dotsc,X_k,\left( Y_r \right)_{J^c} \right)
\end{equation}
for arbitrary vector fields $X_1,\dotsc,X_k$.
(In other words, $\tilde{\omega}=\psi_{k,j,r}(\vec{\eta})$ is the $k$-form with values
\begin{equation*}
  \tilde{\omega}(p;v_1,\dotsc,v_k) = \frac{1}{\binom{\Dim-k}{j-k}} \sum_{\substack{J^c=\set{i_1,\dotsc,i_{\Dim-j}} \\ i_1 < \dotsb < i_{\Dim-j}}} \eta_J \! \left(p;v_1,\dotsc,v_k,Y_{r,i_1}(p),\dotsc,Y_{r,i_{\Dim-j}}(p) \right)
\end{equation*}
for all $p\in\manifold$ and for all $v_1,\dotsc,v_k\in T_p\manifold$, where $i_1<\dotsb<i_{\Dim-j}$ are the elements of $J^c$ in increasing order.
We remark that this relation could instead be written in coordinates as
\begin{equation*}
  \tilde{\omega}_K = \frac{1}{\binom{\Dim-k}{j-k}} \sum_{\abs{J}=j, K\subset J} \epsilon^{K,J^c}_{J^c\union K} \, \eta_{J,J^c\union K}
  ,
\end{equation*}
but that approach would involve tracking the signs $\epsilon^{A,B}_{A\union B}\in\set{-1,1}$ defined by $\ed x_r^{A\union B} = \epsilon^{A,B}_{A\union B} \, \ed x_r^A \wedge \ed x_r^B$ for $A\intersect B=\emptyset$.)

\begin{lemma}\label{phipsiIdentityLocal}
  For all $\omega\in\Lambda^k$, all $j\geq k$, and all $r$, we have $\psi_{k,j,r}(\phi_{k,j,r}(\omega)) = \omega$ on $U'_r$.
\end{lemma}
\begin{proof}
  Note that if $\omega$ is written in coordinates on $U_r$ as $\omega=\sum_{\abs{K}} \omega_K \, \ed x_r^K$, then $\omega_K=\omega\bigl( (\frac{\partial}{\partial x_r^i} \bigr)_{i\in K})$ recovers the coordinate functions.
  Thus, abbreviating $\eta_J=\omega\wedge \ed y_r^{J^c}$ for $\abs{J}=j$ and $\tilde{\omega}=\psi_{k,j,r}(\phi_{k,j,r}(\omega))=\psi_{k,j,r}(\vec{\eta})$, it suffices to show that $\tilde{\omega}_K=\omega_K$ on $U'_r$, for all $\abs{K}=k$.
  We have
  \begin{align*}
    \tilde{\omega}_K = \tilde{\omega}((Y_r)_K) &= \frac{1}{\binom{\Dim-k}{j-k}} \sum_{\abs{J}=j} \eta_J((Y_r)_K, (Y_r)_{J^c})
    \\&
    = \frac{1}{\binom{\Dim-k}{j-k}} \sum_{\abs{J}=j, K\subset J} \eta_J((Y_r)_K, (Y_r)_{J^c})
  \end{align*}
  because, in any term for which $K$ is not a subset of $J$, we will have $K\intersect J^c\neq\emptyset$ and hence the arguments to $\eta_J$ will contain at least one repeated vector field $Y_{r,i}$.
  By \eqref{yJcyA} and the fact that $\eta_J=\omega\wedge \ed y_r^{J^c}$, it follows that
  \begin{align*}
    \eta_J((Y_r)_K, (Y_r)_{J^c}) = \omega((Y_r)_K) \, \ed y_r^{J^c}((Y_r)_{J^c}) = \omega((Y_r)_K) = \omega_K
  \end{align*}
  on $U'_r$.
  Hence
  \begin{align*}
    \tilde{\omega}_K = \frac{1}{\binom{\Dim-k}{j-k}} \sum_{\abs{J}=j, K\subset J} \omega_K = \omega_{K}
  \end{align*}
  on $U'_r$, since the choice of $J$ corresponds to choosing the $j-k$ elements of $J\setminus K$ out of the $\Dim-k$ choices in $[\Dim]\setminus K$.
\end{proof}

We obtain a globally injective function $\Phi_{k,j}\colon \Lambda^k\to (\Lambda^{\Dim-j+k})^{\binom{\Dim}{j} \maxr}$ with left inverse $\Psi_{k,j}\colon (\Lambda^{\Dim-j+k})^{\binom{\Dim}{j} \maxr}\to\Lambda^k$ by aggregating across all charts,
\begin{equation}
  \begin{gathered}
    \Phi_{k,j}(\omega) = (\phi_{k,j,r}(\omega))_{r=1,\dotsc,\maxr} = (\omega\wedge \ed y_r^{J^c})_{r=1,\dotsc,\maxr; \, \abs{J}=j}
    ,
    \\
    \Psi_{k,j}(\vec{\eta}) = \Psi_{k,j}((\eta_{r,J})_{r=1,\dotsc,\maxr; \, \abs{J}=j}) = \sum_{r=1}^\maxr \psi_{k,j,r}((\chi_r\eta_{r,J})_{\abs{J}=k+1})
    .
  \end{gathered}
\end{equation}

\begin{lemma}\label{PsiPhiIdentity}
  For all $k\leq j\leq \Dim$, $\Psi_{k,j}\circ \Phi_{k,j} = \mathrm{identity}_{\Lambda^k}$.
\end{lemma}
\begin{proof}
  Both $\phi_{k,r}$ and $\psi_{k,r}$ are linear over scalar-valued functions, so
  \begin{equation*}
    \Psi_k(\Phi_k(\omega)) = \sum_{r=1}^\maxr \psi_{k,r}( \chi_r \phi_{k,r}(\omega) ) = \sum_{r=1}^\maxr \chi_r \psi_{k,r}(\phi_{k,r}(\omega)) = \sum_{r=1}^\maxr \chi_r \omega = \omega
  \end{equation*}
  where we used \Cref{phipsiIdentityLocal} and the fact that $\chi_r$ is supported in $U'_r$.
\end{proof}

Write $\vec{\ed}_k$ for the operator that applies $\ed_k$ to each element of a vector (of unspecified size) whose entries are $k$-forms.
\begin{lemma}\label{dPhiRelation}
  For all values of $k,j$,
  \begin{equation*}
    \vec{\ed}_{\Dim-j+k} \Phi_{k,j}(\omega) = \Phi_{k+1,j}(\ed_k\omega)
    ,
  \end{equation*}
  and for $k+1\leq j$,
  \begin{equation*}
    \ed_k\omega = \Psi_{k+1,j}(\vec{\ed}_{\Dim-j+k}\Phi_{k,j}(\omega))
    .
  \end{equation*}
\end{lemma}

\begin{proof}
  The first statement follows immediately by considering vector entries separately:
  \begin{equation*}
    \ed_{\Dim-j+k} (\omega\wedge \ed y_r^{J^c}) = \ed_k\omega \wedge \ed y_r^{J^c}
  \end{equation*}
  for each $r,J$.
  When $j\geq k+1$, the second statement follows from the first because of \Cref{PsiPhiIdentity}.
\end{proof}

\begin{remark}\label{rem:dkAsdm}
  Noting that $\Dim-j+k$ runs between $k$ and $\Dim-1$ as $j$ ranges over all admissible values $j\in\set{k+1,\dotsc,\Dim}$, we see that a $k$-form exterior derivative can be expressed in terms of vectors of $m$-form exterior derivatives for any $m\in\set{k,\dotsc,\Dim-1}$.
  In particular, exterior derivatives of all form degrees can be expressed in terms of $(\Dim-1)$-form exterior derivatives.
  This is analogous to the assertion in \Cref{sss:CurrentsAtomicDiscussion} that the divergence is the most generic first-order co-cancelling annihilator.
\end{remark}

Henceforth, we will use only the mappings $\Phi_{k,k+1}$, $\Phi_{k+1,k+1}$, $\Psi_{k,k+1}$, all of which have $j=k+1$.
For simplicity, we omit the subscript $j$ and write $\Phi_k=\Phi_{k,k+1}$, $\Phi_{k+1}=\Phi_{k,k+1}$, $\Psi_k=\Psi_{k,k+1}$.
Thus $\Phi_k$ maps $k$-forms into vectors of $(\Dim-1)$-forms; $\Psi_k$ does the same in reverse; and $\Phi_{k+1}$ maps $(k+1)$-forms into vectors of $\Dim$-forms.

\Cref{dPhiRelation} shows that a closed $k$-form can be promoted to a vector of closed $(\Dim-1)$-forms, and this will allow us to prove \Cref{int01Bound_Closed_kform} using \Cref{int01Bound_dF=0}.
Even if the $k$-form $\omega$ is not closed, we will still be able to express $\vec{\ed}_{\Dim-1}\Phi_k(\omega)$ in terms of $\ed_k\omega$, which will allow us to prove \Cref{int01Bound_General_kform} using \Cref{int01Bound_General}.

We will need bounds on the effect of the functions $\Phi_k$, $\Phi_{k+1}$, $\Psi_k$.
To state these bounds, we make the following convention for vector norms.
If $\vec{\eta}=(\eta_s)_{s\in\mathcal{S}}\in(\Lambda^m)^\mathcal{S}$ is a vector of $m$-forms indexed by a finite set $\mathcal{S}$, with values at a point $p\in\manifold$ denoted $\vec{\eta}(p)=(\eta_s(p))_{s\in\mathcal{S}}\in (\wedge^m T_p^*\manifold)^\mathcal{S}$, define
\begin{equation*}
  \norm{\vec{\eta}(p)}_p = \sum_{s\in\mathcal{S}} \norm{\eta_s(p)}_p
  .
\end{equation*}
(Here we have used an $\ell^1$-type norm as far as the vector aspect is concerned, but because we consider finite vectors this is unimportant and, modulo constants, the sum over $s\in\mathcal{S}$ could be replaced by a maximum.)

\begin{lemma}\label{PhiPsiBounded}
  The linear operators $\Phi_k, \Phi_{k+1}, \Psi_k$ are uniformly pointwise bounded,
  \begin{equation*}
    \norm{\Phi_{k}(\omega)(p)}_p \lesssim \norm{\omega(p)}_p
    ,\quad
    \norm{\Phi_{k+1}(\zeta)(p)}_p \lesssim \norm{\zeta(p)}_p
    ,\quad
    \norm{\Psi_{k}(\vec{\eta})(p)}_p \lesssim \norm{\vec{\eta}(p)}_p
    ,
  \end{equation*}
  and in particular
  \begin{align*}
    \norm{\Phi_{k}(\omega)}_{L^1\bigl( (\Lambda^{\Dim-1})^{\binom{\Dim}{k+1}\maxr} \bigr)}
    &\lesssim \norm{\omega}_{L^1(\Lambda^k)}
    ,
    \\
    \norm{\Phi_{k+1}(\zeta)}_{L^1\bigl( (\Lambda^\Dim)^{\binom{\Dim}{k+1}\maxr} \bigr)}
    &\lesssim \norm{\zeta}_{L^1(\Lambda^k)}
    ,
    \\
    \norm{\Psi_{k,j}(\vec{\eta})}_{L^{\Dim/(\Dim-\alpha),1}(\Lambda^k)}
    &\lesssim \norm{\vec{\eta}}_{ L^{\Dim/(\Dim-\alpha),1}\bigl( (\Lambda^{\Dim-1})^{\binom{\Dim}{k+1}\maxr} \bigr) }
    ,
  \end{align*}
  uniformly over all $p\in\manifold$, $\omega\in\Lambda^k$, $\zeta\in\Lambda^{k+1}$, $\vec{\eta}\in(\Lambda^{\Dim-1})^{\binom{\Dim}{k+1}\maxr}$.
\end{lemma}

\begin{proof}
  The first assertion follows because the $\ed y_r^i$ and $Y_{r,i}$ form a finite collection of smooth 1-forms and vector fields on a compact manifold, and therefore have uniformly bounded norms.
  The remaining assertions follow because the $L^1$ and $L^{\Dim/(\Dim-\alpha),1}$ norms depend monotonically on the pointwise norms $\norm{\cdot}_p$.
\end{proof}

With these preparations, we can now prove \Cref{int01Bound_Closed_kform,int01Bound_General_kform}.
The proofs use \Cref{int01Bound_dF=0,int01Bound_General}, whose proofs in the non-orientable case we have deferred until \Cref{ss:NonorientableProof}.
Subject to that caveat, the argument that follows does not use orientability.

\begin{proof}[Proof of \Cref{int01Bound_General_kform}]
The argument works by comparing the $k$-form heat propagator $e^{t\Delta_k}\omega$ to the quantity $\Psi_k(e^{t\vec{\Delta}_{\Dim-1}}\Phi_k(\omega))$ involving an $(\Dim-1)$-form propagator:
\begin{align}
  \int_0^1& t^{\alpha/2-1} \bignorm{ e^{t\Delta_k}\omega  }_{L^{\Dim/(\Dim-\alpha),1}} \,\dd t
  \notag\\&
  \leq \int_0^1 t^{\alpha/2-1} \bignorm{\Psi_k(e^{t\vec{\Delta}_{\Dim-1}}\Phi_k(\omega)) }_{L^{d/(d-\alpha),1}} \,\dd t
  \notag\\&\quad
  + \int_0^1 t^{\alpha/2-1} \bignorm{ e^{t\Delta_k}\omega - \Psi_k(e^{t\vec{\Delta}_{\Dim-1}}\Phi_k(\omega)) }_{L^{d/(d-\alpha),1}} \,\dd t
  \notag\\&
  =: I+\mathit{II}
  .
  \label{IandII}
\end{align}
In particular, the estimate for $I$ follows directly from $(\Dim-1)$-form bounds:
\begin{align}
  \label{PromotedPropagatedBound}
  I&= \int_0^1 t^{\alpha/2-1} \bignorm{ \Psi_{k}(e^{t\vec{\Delta}_{\Dim-1}}\Phi_{k}(\omega)) }_{L^{\Dim/(\Dim-\alpha),1}} \,\dd t
  \\&
  \lesssim \int_0^1 t^{\alpha/2-1} \bignorm{ e^{t\vec{\Delta}_{\Dim-1}}\Phi_{k}(\omega) }_{L^{\Dim/(\Dim-\alpha),1}} \,\dd t
  &&\text{by \Cref{PhiPsiBounded}}
  \notag\\&
  \lesssim \norm{\Phi_{k}(\omega)}_{L^1} + \bignorm{\vec{\ed}_{\Dim-1}\Phi_{k}(\omega)}_{L^1}
  &&\text{by \Cref{int01Bound_General}}
  \notag\\&
  \lesssim \norm{\Phi_{k}(\omega)}_{L^1} + \bignorm{\Phi_{k+1}(\ed_k \omega)}_{L^1}
  &&\text{by \Cref{dPhiRelation}}
  \notag\\&
  \lesssim \norm{\omega}_{L^1} + \norm{\ed_k\omega}_{L^1}
  &&\text{by \Cref{PhiPsiBounded}}
  .
  \notag
\end{align}

To estimate $\mathit{II}$, we recall several results concerning the heat propagator on forms, see for instance \cite{BerGetVer1992}*{Theorem~2.30}.
It can be shown that the heat propagator is an integral operator, and we write $\mathrm{p}_{t}^k(p,q)(\cdot) \in\mathrm{Hom}(\wedge^k(T^*_q\manifold), \wedge^k(T^*_p\manifold))$ for the kernel corresponding to $e^{t\Delta_k}$, i.e.,
\begin{equation}\label{k_heat_kernel}
  e^{t\Delta_k}\omega(p) = \int_\manifold \mathrm{p}_{t}^k(p,q)(\omega(q)) \,\dd V(q)
  .
\end{equation}
Since $\Psi_k$ and $\Phi_k$ operate pointwise, we have similarly
\begin{align}
  \Psi_k(e^{t\vec{\Delta}_{\Dim-1}}\Phi_k(\omega)) &= \int_\manifold \widetilde{\mathrm{p}}_{t}^{k}(p,q) (\omega(q))\,\dd V(q)
  ,
\end{align}
where
\begin{equation}\label{Psi_Phi_kernel}
  \widetilde{\mathrm{p}}_{t}^{k}(p,q) = \Psi_k\vert_p \circ \mathbf{p}_{t}^{\Dim-1}(p,q) \circ \Phi_k\vert_q
  .
\end{equation}
In other words, the kernel for $\Psi_k \circ e^{t\vec{\Delta}_{\Dim-1}} \circ \Phi_k$ is the composition of the linear maps $\Psi_k\vert_p$, $\mathbf{p}_{t}^{\Dim-1}(p,q)$, and $\Phi_k\vert_q$, where $\mathbf{p}_{t}^{\Dim-1}(p,q)$ means the linear map $\mathrm{p}_{t}^{\Dim-1}(p,q)$ applied elementwise to a vector with entries in $\wedge^{\Dim-1}(T^*_q\manifold)$.

The kernels $\mathrm{p}_t^k,\mathrm{p}_t^{\Dim-1}$ are smooth for $t>0$, and for $t\to 0$ behave like scalar heat kernels times an asymptotic series in powers of $t$.
Namely, there are functions $V_{k,i}(p,q)$ with values in $\mathrm{Hom}(\wedge^{k}(T^*_q\manifold), \wedge^{k}(T^*_p\manifold))$ and depending smoothly on $p,q$ such that, defining
\begin{align}
  \mathrm{p}_t^{k,N}(p,q) &= \frac{\exp(-d_\manifold(p,q)^2/(4t))}{(4\pi t)^{\Dim/2}} \sum_{i=0}^N t^i V_{k,i}(p,q)
  ,
  \label{k_kernel_approx}
  \\
  \mathrm{p}_t^{\Dim-1,N}(p,q) &= \frac{\exp(-d_\manifold(p,q)^2/(4t))}{(4\pi t)^{\Dim/2}} \sum_{i=0}^N t^i V_{\Dim-1,i}(p,q)
  ,
  \label{n-1_kernel_approx}
\end{align}
and by fixing $N$ large enough (specifically, $N\geq\Dim/2$ suffices) we can ensure that
\begin{align}
  \|\mathrm{p}_t^k(p,q) - \mathrm{p}_t^{k,N}(p,q)\|_{p,q} &\lesssim 1
  ,
  \label{k_kernel_error}
  \\
  \|\mathrm{p}_t^{\Dim-1}(p,q) - \mathrm{p}_t^{\Dim-1,N}(p,q)\|_{p,q} &\lesssim 1
  ,
  \label{n-1_kernel_error}
\end{align}
uniformly over $0<t\leq 1$ and $p,q\in\manifold$, where $\norm{\cdot}_{p,q}$ means the operator norm on $\mathrm{Hom}(\wedge^{k}(T^*_q\manifold), \wedge^{k}(T^*_p\manifold))$.
Moreover
\begin{equation}\label{base_term_diagonal}
  V_{k,0}(p,p)=\mathrm{identity}_{\wedge^{k}(T^*_p\manifold)}
  ,\quad
  V_{\Dim-1,0}(p,p)=\mathrm{identity}_{\wedge^{\Dim-1}(T^*_p\manifold)}
  .
\end{equation}

By linearity, the kernel $\widetilde{\mathrm{p}}_t^k$ from \eqref{Psi_Phi_kernel} has a similar approximation
\begin{align}
  \widetilde{\mathrm{p}}_t^{k,N}(p,q)
  &= \frac{\exp(-d_\manifold(p,q)^2/(4t))}{(4\pi t)^{\Dim/2}} \sum_{i=0}^N t^i \widetilde{V}_{k,i}(p,q)
  \label{n-1_kernel_approx_promoted}
\end{align}
where $\widetilde{V}_{k,i}(p,q)=\Psi_k\vert_p \circ \mathbf{V}_{\Dim-1,i}(p,q) \circ \Phi_k\vert_q$ for all $i$, and where similarly $\mathbf{V}_{\Dim-1,i}(p,q)$ means $V_{\Dim-1,i}(p,q)$ applied entrywise to a vector with entries in $\wedge^{\Dim-1}(T_q^*\manifold)$.
The linear mappings $\Psi_k\vert_p, \Phi_k\vert_q$ vary smoothly as functions of $p,q\in\manifold$, and since $\manifold$ is compact it follows that they have uniformly bounded operator norms.
From the bound \eqref{n-1_kernel_error} it therefore follows that
\begin{equation}\label{kernel_error_promoted}
  \|\widetilde{\mathrm{p}}_t^k(p,q) - \widetilde{\mathrm{p}}_t^{k,N}(p,q)\|_{p,q} \lesssim 1
  ,
\end{equation}
again uniformly over $0<t\leq 1$ and $p,q\in\manifold$.

By \Cref{PsiPhiIdentity} and \eqref{base_term_diagonal},
\begin{align}
  \widetilde{V}_{k,0}(p,p) &= \Psi_k\vert_p \circ \mathrm{identity}_{(\wedge^{\Dim-1}(T^*_p\manifold))^{\maxr\binom{\Dim}{k+1}}} \circ \Phi_k\vert_p
  \notag\\&
  = \mathrm{identity}_{\wedge^k(T^*_p\manifold)}
  ,
\end{align}
and in particular $\widetilde{V}_{k,0}(p,p)=V_{k,0}(p,p)$ for all $p\in\manifold$.
Since both $V_{k,0}$ and $\widetilde{V}_{k,0}$ are smooth, it follows that
\begin{equation}
  \bignorm{V_{k,0}(p,q) - \widetilde{V}_{k,0}(p,q)}_{p,q} = O(d_\manifold(p,q))
  ,
\end{equation}
uniformly over $p,q\in\manifold$.
Meanwhile, the functions $V_{k,i},\widetilde{V}_{k,i}$ are continuous with compact support and therefore have uniformly bounded norms.
Combining,
\begin{align}
  &\norm{ \mathrm{p}_t^{k,N}(p,q) - \widetilde{\mathrm{p}}_t^{k,N}(p,q) }_{p,q} \notag\\&\quad
  \leq \frac{\exp(-d_\manifold(p,q)^2/(4t))}{(4\pi t)^{\Dim/2}} \left[ \bignorm{V_{k,0}(p,q) - \widetilde{V}_{k,0}(p,q)}_{p,q}
  \Bigg.\right.
  \notag\\&\qquad\qquad\qquad\qquad\qquad\qquad
  \left.
  + \sum_{i=1}^N t^i \bigl( \bignorm{V_{k,i}(p,q)}_{p,q} + \bignorm{\widetilde{V}_{k,i}(p,q)}_{p,q} \bigr) \right]
  \notag\\&\quad
  \lesssim \frac{\exp(-d_\manifold(p,q)^2/(4t))}{(4\pi t)^{\Dim/2}} [ d_\manifold(p,q) + t ]
  ,
  \label{promotion_error_N}
\end{align}
uniformly over $0<t\leq 1$ and $p,q\in\manifold$, where the last bound has used $t^i\leq t$ for all $i\geq 1$.
Further combining \eqref{k_kernel_error}, \eqref{kernel_error_promoted}, and \eqref{promotion_error_N},
\begin{equation}\label{promotion_error_total}
  \norm{\mathrm{p}_t^k(p,q) - \widetilde{\mathrm{p}}_t^k(p,q)}_{p,q} \lesssim \frac{\exp(-d_\manifold(p,q)^2/(4t))}{(4\pi t)^{\Dim/2}} [ d_\manifold(p,q) + t ] + 1
  .
\end{equation}

Considered as a function of $p$ with $q$ fixed, the upper bound in \eqref{promotion_error_total} has Lorentz norm bounded by
\begin{equation}\label{UpperBoundLorentz}
  \left\|\frac{\exp(-d_\manifold(\cdot,q)^2/(4t))}{(4\pi t)^{\Dim/2}}  \left(d_\manifold(\cdot,q) + t \right)+1\right\|_{L^{\Dim/(\Dim-\alpha),1}}
  \lesssim t^{-\alpha/2}t^{1/2} + 1
  .
\end{equation}
We can verify \eqref{UpperBoundLorentz} using the \InterpolationLemma\ as follows.
Recall from \eqref{dplustTimesHeatKernelL1Bound} that the non-constant part of \eqref{UpperBoundLorentz} has $L^1$ norm of order $t^{1/2}$, while it is easy to check that its $L^\infty$ norm is of order $t^{1/2-\Dim/2}$.
Thus
\begin{equation*}
  \norm{ \frac{\exp(-d_\manifold(\cdot,q)^2/(4t))}{(4\pi t)^{\Dim/2}}  \left(d_\manifold(\cdot,q) + t \right) }_{L^{\Dim/(\Dim-\alpha),1}} \lesssim (t^{1/2})^{1-\alpha/\Dim} (t^{1/2-\Dim/2})^{\alpha/\Dim}
  ,
\end{equation*}
giving the term $t^{-\alpha/2+1/2}$ as required.
Meanwhile the constant term gives a trivial contribution by the compactness of the manifold.

Thus applying Minkowski's inequality for integrals to the pointwise upper bound
\begin{equation*}
  \norm{ e^{t\Delta_k}\omega(p) - \Psi_k(e^{t\vec{\Delta}_{\Dim-1}}\Phi_k(\omega)) (p)}_{p} \leq \int_\manifold \norm{\mathrm{p}_t^k(p,q) - \widetilde{\mathrm{p}}_t^k(p,q)}_{p,q} \norm{\omega(q)}_q \, \dd V(q)
\end{equation*}
yields
\begin{align*}
  &\bignorm{ e^{t\Delta_k}\omega - \Psi_k(e^{t\vec{\Delta}_{\Dim-1}}\Phi_k(\omega)) }_{L^{\Dim/(\Dim-\alpha),1}} \\&\quad
  \lesssim \int_\manifold \left\|\frac{\exp(-d_\manifold(\cdot,q)^2/(4t))}{(4\pi t)^{\Dim/2}}  \left(d_\manifold(\cdot,q) + t^{1/2} \right)+1\right\|_{L^{\Dim/(\Dim-\alpha),1}}\norm{\omega(q)}_q \,\dd V(q)
  \\&\quad
  \lesssim \left(t^{-\alpha/2}t^{1/2} +1\right)\norm{\omega}_{L^1}
  .
\end{align*}
Finally, returning to \eqref{IandII}, we have
\begin{align*}
  \mathit{II} &= \int_0^1 t^{\alpha/2} \bignorm{ e^{t\Delta_k}\omega - \Psi_k(e^{t\vec{\Delta}_{\Dim-1}}\Phi_k(\omega)) }_{L^{\Dim/(\Dim-\alpha),1}} \frac{\dd t}{t}
  \\&
  \lesssim \int_0^1 \left( t^{1/2} + t^{\alpha/2} \right)\norm{\omega}_{L^1} \frac{\dd t}{t}
  \\&
  \lesssim \norm{\omega}_{L^1}
\end{align*}
which along with \eqref{PromotedPropagatedBound} completes the proof.
\end{proof}

With the additional assumption that $\ed_k \omega=0$, we can prove \Cref{int01Bound_Closed_kform} without appealing to \Cref{int01Bound_General}.

\begin{proof}[Proof of \Cref{int01Bound_Closed_kform}]
  Now suppose $\ed_k\omega=0$ is given.
  Then \Cref{dPhiRelation} implies that $\vec{\ed}_{\Dim-1}\Phi_k(\omega)=0$, i.e., every entry of $\Phi_k(\omega)$ is a closed $(\Dim-1)$-form to which \Cref{int01Bound_dF=0} applies.
  Then the bound \eqref{PromotedPropagatedBound} can be replaced by
  \begin{align}
    \label{PromotedPropagatedBoundClosed}
    I&= \int_0^1 t^{\alpha/2-1} \bignorm{ \Psi_{k}(e^{t\vec{\Delta}_{\Dim-1}}\Phi_{k}(\omega)) }_{L^{\Dim/(\Dim-\alpha),1}} \,\dd t
    \\&
    \lesssim \int_0^1 t^{\alpha/2-1} \bignorm{ e^{t\vec{\Delta}_{\Dim-1}}\Phi_{k}(\omega) }_{L^{\Dim/(\Dim-\alpha),1}} \,\dd t
    &&\text{by \Cref{PhiPsiBounded}}
    \notag\\&
    \lesssim \norm{\Phi_{k}(\omega)}_{L^1}
    &&\text{by \Cref{int01Bound_dF=0}}
    \notag\\&
    \lesssim \norm{\omega}_{L^1}
    &&\text{by \Cref{PhiPsiBounded}}
    ,
    \notag
  \end{align}
  using \Cref{int01Bound_dF=0} instead of \Cref{int01Bound_General}.
  The rest of the proof is identical, and we find that $\mathit{II}\lesssim \norm{\omega}_{L^1}$ and thus $I+\mathit{II}\lesssim \norm{\omega}_{L^1}$, as required.
\end{proof}

\section{Proofs of the main results}\label{s:MainProofs}

\subsection{Riesz potential bounds}\label{ss:RieszProofs}

Up to this point, all our bounds and integrals rely ultimately on regularity and growth properties of the heat kernel in the interval $0<t\leq 1$.
To handle $t>1$, we will need different bounds powered by the assumption of orthogonality to the space of harmonic forms.

\begin{lemma}\label{exp_decay}
  Let $t_1>0$ be given.
  Then there exist constants $C_2<\infty$, $c_3>0$ such that, for all $k\in\set{0,\dotsc,n}$, for all $k$-forms $\xi \in \mathcal{H}^\perp(\Lambda^k)$, and for all $t\geq t_1$,
  \begin{align*}
    \|e^{t\Delta_k} \xi\|_{L^\infty(\Lambda^k)}  \leq C_2 \exp(-c_3t) \|\xi\|_{L^1(\Lambda^k)}
    .
  \end{align*}
\end{lemma}

\begin{proof}
Define $\epsilon=t_1/2$, so that for any $t\geq t_1$ we may express
\begin{align*}
  t= \epsilon  +(t-t_1)+ \epsilon
\end{align*}
and use the semigroup property of the heat propagator to obtain
\begin{align*}
  e^{t\Delta_k} \xi = e^{\epsilon\Delta_k} e^{(t-t_1)\Delta_k} e^{\epsilon\Delta_k} \xi
  .
\end{align*}
By \Cref{CurrentEvolutionAsForm} and H\"older's inequality we deduce
\begin{align}
  \sup_{p\in \manifold} \norm{e^{t\Delta_k} \xi(p)}_p  &\leq C \sup_{p\in \manifold} \int_\manifold K_\epsilon(p,q) \norm{ e^{(t-t_1)\Delta_k} e^{\epsilon\Delta_k} \xi(q)}_q \,\dd V(q)
  \notag\\
  &\lesssim \|e^{(t-t_1)\Delta_k} e^{\epsilon\Delta_k}\xi\|_{L^2(\Lambda^k)}
  \label{LargetLinftyFromL2}
  .
\end{align}
A similar argument shows that
\begin{equation}\label{LargetL2FromL1}
  \|e^{\epsilon\Delta_k}\xi\|_{L^2(\Lambda^k)}  \lesssim \|\xi\|_{L^1(\Lambda^k)}
\end{equation}
and therefore it remains to bound $\|e^{(t-t_1)\Delta_k} e^{\epsilon\Delta_k}\xi\|_{L^2(\Lambda^k)}$ in terms of $\|e^{\epsilon\Delta_k}\xi\|_{L^2(\Lambda^k)}$.
The proof is similar to \cite{BerGetVer1992}*{p.~88}.
Let $\{\phi_i\}_{i=1}^\infty$ be an orthonormal basis for $L^2(\Lambda^k)$ consisting of eigenforms of $\Delta_k$ with corresponding eigenvalues $\{\lambda_i\}_{i=1}^\infty$.
The Laplacian $\Delta_k$ is positive semi-definite with a purely discrete spectrum, so finitely many eigenvalues $\lambda_i$ may be zero but the remaining eigenvalues satisfy $\min\set{\lambda_i\colon \lambda_i\neq 0}=\delta>0$.
Using the relation
\begin{align*}
  e^{(t-t_1)\Delta_k}\phi_i = \exp(-\lambda_i (t-t_1)) \phi_i
\end{align*}
along with Parseval's identity,
\begin{align*}
  \|e^{(t-t_1)\Delta_k} e^{\epsilon\Delta_k}\xi\|_{L^2(\Lambda^k)} &=\left(\sum_{i=1}^\infty \Angles{e^{(t-t_1)\Delta_k} e^{\epsilon\Delta_k}\xi,\phi_i}^2 \right)^{1/2}\\
  &= \left(\sum_{i=1}^\infty \exp(-2\lambda_i (t-t_1)) \Angles{e^{\epsilon\Delta_k}\xi,\phi_i}^2 \right)^{1/2}.
\end{align*}
By assumption, $\xi$ is orthogonal to the space of harmonic forms, so $e^{\epsilon\Delta_k}\xi$ is as well and therefore all of the terms with $\lambda_i=0$ vanish.
This implies
\begin{align*}
  \|e^{(t-t_1)\Delta_k} e^{\epsilon\Delta_k}\xi\|_{L^2(\Lambda^k)}  \leq e^{-\delta(t-t_1)}  \|e^{\epsilon\Delta_k}\xi\|_{L^2(\Lambda^k)}
  ,
\end{align*}
and together with \eqref{LargetLinftyFromL2} and \eqref{LargetL2FromL1} this completes the proof.
\end{proof}

\begin{corollary}\label{int1inftyBound}
  For all $k\in\set{0,\dotsc,n}$ and uniformly over all $k$-forms $F \in \mathcal{H}^\perp(\Lambda^k)$,
  \begin{equation*}
    \int_1^\infty t^{\alpha/2-1} \norm{e^{t\Delta_k} F}_{L^{\Dim/(\Dim-\alpha),1}(\Lambda^k)} \,\dd t
    \lesssim \norm{F}_{L^1(\Lambda^k)}
    .
  \end{equation*}
\end{corollary}
\begin{proof}
  Because of the compactness of the manifold, $L^\infty$ bounds immediately give $L^1$ bounds, whereupon the \InterpolationLemma\ gives the Lorentz bound
  \begin{equation*}
    \| e^{t\Delta_k} F\|_{L^{\Dim/(\Dim-\alpha),1}(\Lambda^k)}\lesssim V(\manifold)^{1-\alpha/\Dim} \| e^{t\Delta_k} F \|_{L^{\infty}(\Lambda^k)}
    .
  \end{equation*}
  Thus \Cref{exp_decay} with $t_1=1$ gives
  \begin{align*}
    \int_1^\infty t^{\alpha/2-1} \|e^{t\Delta_k}  F \|_{L^{\Dim/(\Dim-\alpha),1}(\Lambda^k)}\,\dd t &\lesssim \int_1^\infty t^{\alpha/2-1} \|e^{t\Delta_k}  F \|_{L^{\infty}(\Lambda^k)}\,\dd t \\
    &\lesssim \int_1^\infty t^{\alpha/2-1}  \exp(-c_3t) \|F\|_{L^1(\Lambda^k)}\,\dd t \\
    &\lesssim \|F\|_{L^1(\Lambda^k)}
    .
    \qedhere
  \end{align*}
\end{proof}

\begin{proof}[Proof of \Cref{BesovClosedCoclosed} for orientable manifolds]
  Throughout, let the $k$-form $F\in L^1\cap\mathcal{H}^\perp(\Lambda^k)$ be orthogonal to the space of harmonic $k$-forms.
  If $\ed_k F=0$, the result follows immediately from \Cref{int01Bound_Closed_kform,int1inftyBound}.

  To handle $\ed^*_k F$ when $k\in\set{1,\dotsc,\Dim}$, apply Hodge duality.
  Define the $(\Dim-k)$-form
  \begin{equation}
    \tilde{F} = \star_k F
    .
  \end{equation}
  The Hodge star is an isometry; preserves the Hodge Laplacian and hence also the heat propagator, Riesz transform, and orthogonality to harmonic forms (apart from the degree of the forms concerned); and interchanges $\ed$ with $\ed^*$ modulo signs, see \Cref{DeltaAndStar} and \eqref{dAndStar}.
  Thus
  \begin{equation}\label{starInterchanges}
    \begin{aligned}
      \bignorm{\tilde{F}(p)}_p &= \bignorm{F(p)}_p \quad\text{for all }p\in\manifold
      ,
      \\
      \bignorm{\tilde{F}}_{L^1(\Lambda^{\Dim-k})} &= \bignorm{F}_{L^1(\Lambda^k)}
      \\
      e^{t\Delta_{\Dim-k}}\tilde{F} &= \star_k e^{t\Delta_k}F
      ,
      \\
      \bignorm{e^{t\Delta_{\Dim-k}}\tilde{F}}_{L^{\Dim/(\Dim-\alpha),1}(\Lambda^{\Dim-k})} &= \bignorm{e^{t\Delta_k}F}_{L^{\Dim/(\Dim-\alpha),1}(\Lambda^k)}
      ,
      \\
      \ed_{\Dim-k}\tilde{F} &= (-1)^k \star_{k-1} \ed^*_k F
      ,
      \\
      \bignorm{\ed_{\Dim-k}\tilde{F}}_{L^1(\Lambda^{\Dim-k+1})} &= \bignorm{\ed^*_k F}_{L^1(\Lambda^{k-1})}
      .
    \end{aligned}
  \end{equation}
  In particular, the assumption $\ed^*_k F=0$ gives $\ed_{\Dim-k}\tilde{F}=0$, so we may apply our preceding result to the closed $\tilde{k}$-form $\tilde{F}$, where $\tilde{k}=\Dim-k\in\set{1,\dotsc,\Dim-1}$, to find
  \begin{align*}
    \int_0^\infty t^{\alpha/2-1} \bignorm{e^{t\Delta_k}F}_{L^{\Dim/(\Dim-\alpha),1}(\Lambda^k)} \, \dd t
    &= \int_0^\infty t^{\alpha/2-1} \bignorm{e^{t\Delta_{\Dim-k}}\tilde{F}}_{L^{\Dim/(\Dim-\alpha),1}(\Lambda^{\Dim-k})} \, \dd t
    \\&
    \lesssim \bignorm{\tilde{F}}_{L^1(\Lambda^{\Dim-k})} = \bignorm{F}_{L^1(\Lambda^k)}
    .
    \qedhere
  \end{align*}
\end{proof}

\begin{proof}[Proof of \Cref{RieszClosedCoclosed_k} for orientable manifolds]
  The bound \eqref{potentialnodiracl1} follows immediately from \Cref{BesovClosedCoclosed} by Minkowski's inequality for integrals.
\end{proof}

\begin{proof}[Proof of \Cref{divinLresult} for orientable manifolds]
  The bound \eqref{boundwithd} follows immediately from \Cref{int01Bound_General_kform} and \Cref{int1inftyBound}.
  The bound \eqref{boundwithdstar} follows from \eqref{boundwithd} by the same Hodge duality argument as in the proof of \Cref{BesovClosedCoclosed}, noting from \eqref{starInterchanges} that the term $\bignorm{\ed_{\Dim-k}\tilde{F}}_{L^1(\Lambda^{\Dim-k+1})}$ in the resulting upper bound reduces to the term $\norm{\ed^*_k F}_{L^1(\Lambda^{k-1})}$ in \eqref{boundwithdstar}.
\end{proof}

\subsection{Removing the orientability assumption}\label{ss:NonorientableProof}

\begin{proof}[Proofs for non-orientable manifolds]
  Now let $\manifold$ be a general compact manifold, not necessarily orientable, and let $\widehat{\manifold}$ be the orientable double cover of $\manifold$.
  That is, $\widehat{\manifold}$ is an orientable Riemannian manifold and there is a natural local isometry $\pi\colon\widehat{\manifold}\to\manifold$ such that $\widehat{\pi}$ is two-to-one, i.e., $\widehat{\pi}^{-1}(\set{p})$ contains exactly two elements of $\widehat{\manifold}$ for each $p\in\manifold$; see \cite{LeeSmoothManifolds}*{Chapter 15}.
  Since $\manifold$ is compact, so is $\widehat{\manifold}$.
  The conclusions of \Cref{RieszClosedCoclosed_k,divinLresult,BesovClosedCoclosed,int01Bound_dF=0,int01Bound_General,int01Bound_Closed_kform,int01Bound_General_kform} therefore apply to $\widehat{\manifold}$.

  Let $F\in C^\infty(\Lambda^k\manifold)$ be a smooth $k$-form on $\manifold$, and write $\widehat{F}=\widehat{\pi}^\# F\in C^\infty(\Lambda^k\widehat{\manifold})$ for its pull-back to $\widehat{\manifold}$.
  Here we write $\Lambda^k\manifold,\Lambda^k\widehat{\manifold}$ to distinguish forms on $\manifold$ or $\widehat{\manifold}$.
  By general properties of pull-backs and exterior derivatives, $\ed_k \widehat{F} = \widehat{\pi}^\# \ed_k F$, where on the left-hand side the exterior derivative is taken in $\widehat{\manifold}$.
  Moreover, since $\widehat{\pi}$ is a local isometry, it follows that $\widehat{\Delta}_k \widehat{F} = \widehat{\pi}^\# \Delta_k F$, where $\widehat{\Delta}_k$ means the Hodge Laplacian on $\widehat{\manifold}$, and consequently $e^{t\widehat{\Delta}_k}\widehat{F} = \widehat{\pi}^\# e^{t\Delta_k} F$.

  Note that in any compact Riemannian manifold, a $k$-form $\omega$ is orthogonal to the space of harmonic $k$-forms for that manifold if and only if its heat propagator tends to 0 as $t\to\infty$.
  Since $e^{t\widehat{\Delta}_k}\widehat{F} = \widehat{\pi}^\# e^{t\Delta_k} F$, it follows that if $F$ is orthogonal to the space of harmonic $k$-forms for $\manifold$, then $\widehat{F}$ is orthogonal to the space of harmonic $k$-forms for $\widehat{\manifold}$.

  The local isometry property also implies that $\ed^*_k \widehat{F} = \widehat{\pi}^\# \ed^*_k F$.
  For the same reason, pointwise norms coincide: for all $\widehat{p}\in\widehat{\manifold}$,
  \begin{align*}
    \bignorm{\widehat{F}(\widehat{p})}_{\widehat{p}} &= \bignorm{F(\widehat{\pi}(\widehat{p}))}_{\widehat{\pi}(\widehat{p})}
    ,
    \\
    \bignorm{\ed_k\widehat{F}(\widehat{p})}_{\widehat{p}} &= \bignorm{\ed_k F(\widehat{\pi}(\widehat{p}))}_{\widehat{\pi}(\widehat{p})}
    ,
    \\
    \bignorm{\ed^*_k\widehat{F}(\widehat{p})}_{\widehat{p}} &= \bignorm{\ed^*_k F(\widehat{\pi}(\widehat{p}))}_{\widehat{\pi}(\widehat{p})}
    ,
    \\
    \bignorm{e^{t\widehat{\Delta}_k}\widehat{F}(\widehat{p})}_{\widehat{p}} &= \norm{e^{t\Delta_k}F(\widehat{\pi}(\widehat{p}))}_{\widehat{\pi}(\widehat{p})}
    .
  \end{align*}
  Integrate over $\widehat{p}$, noting that because of the double cover, the push-forward under $\widehat{\pi}$ of volume measure on $\widehat{\manifold}$ is twice the volume measure on $\manifold$.
  Hence
  \begin{align*}
    \bignorm{\widehat{F}}_{L^1(\Lambda^k\widehat{\manifold})} &= 2\bignorm{F}_{L^1(\Lambda^k\manifold)}
    ,
    \\
    \bignorm{\ed_k\widehat{F}}_{L^1(\Lambda^{k+1}\widehat{\manifold})} &= 2\bignorm{\ed_k F}_{L^1(\Lambda^{k+1}\manifold)}
    ,
    \\
    \bignorm{\ed^*_k\widehat{F}}_{L^1(\Lambda^{k-1}\widehat{\manifold})} &= 2\bignorm{\ed^*_k F}_{L^1(\Lambda^{k-1}\manifold)}
    .
  \end{align*}
  Similar reasoning shows that Lorentz norms in $\manifold$ and $\widehat{\manifold}$ are also related by a constant factor,
  \begin{align*}
    \bignorm{e^{t\widehat{\Delta}_k}\widehat{F}}_{L^{\Dim/(\Dim-\alpha),1}(\Lambda^{\Dim-1}\widehat{\manifold})} &= C\norm{e^{t\Delta_{\Dim-1}}F}_{L^{\Dim/(\Dim-\alpha),1}(\Lambda^{\Dim-1}\manifold)}
    ,
  \end{align*}
  where careful calculation shows that $C=2^{1-\alpha/\Dim}$.

  Consequently, the bounds in \Cref{RieszClosedCoclosed_k,divinLresult,BesovClosedCoclosed,int01Bound_dF=0,int01Bound_General,int01Bound_Closed_kform,int01Bound_General_kform} for $\manifold$ all follow from the corresponding bounds for $\widehat{\manifold}$.
\end{proof}

\subsection{PDE bounds}\label{ss:PDEproofs}

The proofs of the other main results now follow by expressing the solutions of PDEs in terms of the Riesz potentials for $\alpha=2$ and $\alpha=1$.

\begin{proof}[Proof of \Cref{Hodge_Laplacian_L1}]
  If $\omega \in C^\infty(\Lambda^k)$ then, by \eqref{heat_propagator} and \eqref{LongTimeHeatPropagator},
  \begin{align}
    \mathcal{I}_{2,k} (-\Delta_k\omega) &= \int_0^\infty e^{t\Delta_k}(-\Delta_k\omega)\,\dd t
    \label{left_inverse}\\
    &= \lim_{\epsilon \to 0}\int_{\epsilon}
    ^{1/\epsilon} -\Delta_ke^{t\Delta_k}\omega\,\dd t
    \notag\\
    &= \lim_{\epsilon \to 0}\int_{\epsilon}
    ^{1/\epsilon} -\frac{\partial}{\partial t} e^{t\Delta_k}\omega\,\dd t
    \notag\\
    &= \lim_{\epsilon \to 0}\left[e^{\epsilon\Delta_k}\omega -e^{\epsilon^{-1}\Delta_k}\omega\right]
    \notag\\
    &=P_{\mathcal{H}^\perp}\omega
    ,
    \notag
  \end{align}
  where $P_{\mathcal{H}^\perp}\omega=\omega-P_{\mathcal{H}}\omega$ denotes orthogonal projection onto $\mathcal{H}^\perp(\Lambda^k)$.

  Now consider $F\in L^1\cap\mathcal{H}^\perp(\Lambda^k)$.
  Then
  \begin{align*}
    Z=-\mathcal{I}_{2,k} F
  \end{align*}
  is well-defined, and we claim that it is a weak solution of \eqref{PoissonPDEkForm} in the sense that
  \begin{align*}
    \Angles{Z, \Delta_k\omega}  =   \Angles{F ,\omega}
  \end{align*}
  for all $\omega \in C^\infty(\Lambda^k)$.
  This follows from \eqref{left_inverse} and the self-adjoint property  \eqref{Riesz_selfadjoint} of the Riesz potentials:
  \begin{align*}
    \Angles{Z, \Delta_k \omega}  &=   \Angles{-\mathcal{I}_{2,k} F ,\Delta_k\omega} \\
    &=\Angles{ F ,-\mathcal{I}_{2,k}\Delta_k\omega} \\
    &=\Angles{ F ,P_{\mathcal{H}^\perp}\omega}\\
    &=\Angles{ P_{\mathcal{H}^\perp}F ,\omega}\\
    &=\Angles{F ,\omega}.
  \end{align*}
  Thus $Z$ is the required solution, readily verified to be unique since any two solutions of \eqref{PoissonPDEkForm} differ by a harmonic $k$-form.
  Either of the assumptions $\ed_k F=0$ or $\ed^*_k F=0$ allows us to apply \Cref{RieszClosedCoclosed_k} and deduce the desired bound.
\end{proof}

The proof of \Cref{bbq} follows a similar scheme, with the connection to the Hodge system given by means of Riesz transforms.

\begin{proof}[Proof of \Cref{bbq}]
  For $F,G$ smooth and orthogonal to the spaces of harmonic forms, we may define
  \begin{equation}\label{hodge_solution}
    Z=\ed_{k-1} \mathcal{I}_{2,k-1} F + \ed^*_{k+1} \mathcal{I}_{2,k+1} G
    .
  \end{equation}
  Then, recalling \eqref{IddstarCommute} and \eqref{left_inverse} and using $\ed^*_{k-1} F=0$, we have
  \begin{align*}
    \ed^*_k Z &= \ed^*_k \ed_{k-1} \mathcal{I}_{2,k-1} F + \ed^*_k \ed^*_{k+1} \mathcal{I}_{2,k+1} G \\
    &= \ed^*_k \ed_{k-1} \mathcal{I}_{2,k-1} F  \\
    &= \mathcal{I}_{2,k-1} \ed^*_k \ed_{k-1}  F  \\
    &= \mathcal{I}_{2,k-1} \left(\ed^*_k \ed_{k-1} + \ed_{k-2} \ed^*_{k-1}\right) F  \\
    &= \mathcal{I}_{2,k-1}(- \Delta_{k-1} F) = F
    ,
  \end{align*}
  and similarly $\ed_k Z=G$.
  When $F,G$ are not smooth, it can be shown that \eqref{hodge_solution} still defines a weak solution to the Hodge system, via a duality argument similar to the proof of \Cref{Hodge_Laplacian_L1}.

For the purpose of estimation, it is useful to introduce the Riesz transforms, which can be defined by
\begin{equation}\label{Rdef}
  \begin{aligned}
    R_j \omega &:= \frac{1}{\Gamma(1/2)}\int_0^\infty t^{1/2-1}\ed_j e^{t\Delta_j} \omega\, \dd t
    ,
    \\
    R^*_j \omega &:=  \frac{1}{\Gamma(1/2)}\int_0^\infty t^{1/2-1} \ed^*_j e^{t\Delta_j} \omega \,\dd t
  \end{aligned}
  \end{equation}
whenever $\omega \in L^r(\Lambda^k)$ or $\omega\in L^{r,1}(\Lambda^k)$ for some $r>1$.
If in addition $\omega \in \mathcal{H}^\perp(\Lambda^j)$, this definition is equivalent to
    \begin{equation}\label{R}
    \begin{aligned}
   R_j \omega &\equiv \ed_j \mathcal{I}_{1,j} \omega,\\
       R^*_j \omega &\equiv  \ed^*_j \mathcal{I}_{1,j} \omega,
         \end{aligned}
         \end{equation}
so recalling from \eqref{semigroup} that $ \mathcal{I}_{1,j}  \mathcal{I}_{1,j} = \mathcal{I}_{2,j}$, we can express the solution $Z$ as
  \begin{equation}\label{Z=RIF+RIG}
    Z = R_{k-1} \mathcal{I}_{1,k-1} F + R^*_{k+1} \mathcal{I}_{1,k+1} G
    .
  \end{equation}
 As noted by Strichartz \cite{Strichartz}, the Riesz transforms are bounded as linear operators $R_{k-1} \colon L^r(\Lambda^{k-1}) \to L^r(\Lambda^{k})$ and $R^*_{k+1} \colon L^r(\Lambda^{k+1}) \to L^r(\Lambda^{k})$ for all $1<r<\infty$.
  By interpolation, it follows that they are also bounded as linear operators $R_{k-1} \colon L^{r,s}(\Lambda^{k-1}) \to L^{r,s}(\Lambda^{k})$
  and $R^*_{k+1} \colon L^{r,s}(\Lambda^{k+1}) \to L^{r,s}(\Lambda^{k})$ on Lorentz spaces, for all $1<r<\infty$ and $1\leq s \leq \infty$.

  Combining the observations above, \Cref{RieszClosedCoclosed_k} gives
  \begin{align*}
    \|Z\|_{L^{\Dim/(\Dim-1),1}(\Lambda^k)}
    &\leq  \| R_{k-1} \mathcal{I}_{1,k-1} F \|_{L^{\Dim/(\Dim-1),1}(\Lambda^k)} + \| R^*_{k+1} \mathcal{I}_{1,k+1} G \|_{L^{\Dim/(\Dim-1),1}(\Lambda^k)}
    \\&
    \lesssim \|\mathcal{I}_{1,k-1} F \|_{L^{\Dim/(\Dim-1),1}(\Lambda^{k-1})} + \| \mathcal{I}_{1,k+1} G \|_{L^{\Dim/(\Dim-1),1}(\Lambda^{k+1})}
    \\&
    \lesssim \|F\|_{L^1(\Lambda^{k-1})}+\|G\|_{L^1(\Lambda^{k+1})}
    .
    \qedhere
  \end{align*}
\end{proof}

\subsection{Analogues of the main results for currents}\label{ss:AnaloguesCurrents}

In both the preceding proofs, the PDEs are interpreted in weak form, and this allows us to consider non-smooth forms.
More generally, the Hodge system in \Cref{bbq} (or for the Euclidean case, in \cite{HRS}*{Theorem 1.3}) can also be interpreted when the given data are currents rather than forms.
To state this explicitly, first define
\begin{equation}
  \partial^*_j T(\zeta) = T(\ed^*_{j+1}\zeta) \quad\text{for all }\zeta\in C^1(\Lambda^{j+1})
\end{equation}
for a $j$-current $T$, where as with \eqref{boundary} this definition is interpreted in a distributional sense and there is no guarantee in general that $\partial^*_j T$ is itself a current.

\begin{theorem}\label{bbq_currents}
  Let $\Dim\geq 3$ and suppose $\manifold$ is a smooth, compact Riemannian manifold of dimension $\Dim$, or $\manifold=\R^\Dim$.
  There exists a constant $C=C(\manifold)>0$ such that the following holds.
  Let $k \in \set{1,\dotsc,\Dim-1}$, let $S \in \currents_{k-1}$ and $T\in\currents_{k+1}$, and suppose that
  \begin{equation}\label{Orthogonality_currents}
    \begin{aligned}
      S(\xi)&=0 &&\text{for all }\xi\in\mathcal{H}(\Lambda^{k-1})
      ,
      \\
      T(\zeta)&=0 &&\text{for all }\zeta\in\mathcal{H}(\Lambda^{k+1})
      ,
    \end{aligned}
  \end{equation}
  and
  \begin{equation}
    \partial_{k-1} S = 0, \quad \partial^*_{k+1} T=0
    .
  \end{equation}
  In addition, we impose the conditions $S\equiv 0$ if $k=1$ and $T\equiv 0$ if $k=\Dim-1$.
  Then there is a unique solution $Z$ in $\mathcal{H}^\perp(\Lambda^k)$ of the Hodge system
  \begin{equation}\label{HodgeSystem_currents}
    \begin{aligned}
      \Angles{Z, \ed_{k-1}\xi} &= S(\xi) &&\text{for all }\xi\in C^1(\Lambda^{k-1})
      ,
      \\
      \Angles{Z, \ed^*_{k+1}\zeta} &= T(\zeta) &&\text{for all }\zeta\in C^1(\Lambda^{k+1})
      ,
    \end{aligned}
  \end{equation}
  and
  \begin{align}\label{estimateHodge_currents}
    \norm{Z}_{L^{\Dim/(\Dim-1),1}(\Lambda^k)} \leq C\left(\big. \totalmass(S) + \totalmass(T) \right)
    .
  \end{align}
\end{theorem}
\noindent
The conditions \eqref{HodgeSystem_currents} are analogous to the conditions for $Z$ to be a weak solution of the Hodge system \eqref{HodgeSystem}, and similarly \eqref{Orthogonality_currents} is the analogue of the requirement for the forms $F,G$ to be orthogonal the spaces of harmonic forms in \Cref{bbq}.

The existence and uniqueness solution $Z$ and the uniform bound \eqref{estimateHodge_currents} can be proved from \Cref{bbq} using standard regularization arguments.
For instance, \Cref{NormDualityCurrentsForms} can be used to show that for any positive $t$ the currents $e^{t\Delta_{k-1,c}}S,e^{t\Delta_{k+1,c}}T$ correspond to smooth bounded forms.
For this purpose it is more convenient to use the alternative mapping from forms to currents denoted by $\tilde{T}_\eta$ in \Cref{AlternativeCurrentsFromForms}; indeed, with that notation the Hodge system \eqref{HodgeSystem_currents} can be written compactly as $\partial_k\tilde{T}_Z = S$, $\partial^*_k\tilde{T}_Z=T$.
Whichever mapping is used, it is readily verified that the forms corresponding to the currents $e^{t\Delta_{k-1,c}}S,e^{t\Delta_{k+1,c}}T$ are orthogonal to the appropriate spaces of harmonic forms if and only if \eqref{Orthogonality_currents} holds.

Similar arguments yield analogues of Theorems~\ref{Hodge_Laplacian_L1}, \ref{RieszClosedCoclosed_k}, and their Euclidean counterparts \cite{HS}*{Theorem 1.4} and \cite{HRS}*{Theorem 1.1}.
\begin{theorem}\label{Poisson_currents}
  Under the assumptions of \Cref{Hodge_Laplacian_L1}, there exists a constant $C=C(\manifold)>0$ such that, for all $k$-currents $T\in\currents_k$ vanishing on $\mathcal{H}(\Lambda^k)$ and satisfying $\partial_k T=0$ or $\partial^*_k T=0$, there is a unique solution $Z$ in $\mathcal{H}^\perp(\Lambda^k)$ of the Poisson equation
  \begin{equation}\label{PoissonPDE_currents}
    \Angles{Z, \Delta_k\omega} = T(\omega) \quad\text{for all }\omega\in C^2(\Lambda^k)
    ,
  \end{equation}
  and
  \begin{align}\label{estimate_Laplace_currents}
    \norm{Z}_{L^{\Dim/(\Dim-2),1}(\Lambda^k)} \leq C\totalmass(T)
    .
  \end{align}
  The same holds when $\manifold=\R^\Dim$.
\end{theorem}
\begin{theorem}\label{Riesz_currents}
  Under the assumptions of \Cref{RieszClosedCoclosed_k}, there exists a constant $C=C(\alpha,\manifold)>0$ such that
  \begin{align}\label{potentialnodiracl1_currents}
    \norm{\mathcal{I}_{\alpha,k,c} T}_{L^{\Dim/(\Dim-\alpha),1}(\Lambda^k)} \leq C \totalmass(T)
  \end{align}
  for all $k$-currents $T \in \currents_k$ vanishing on $\mathcal{H}(\Lambda^k)$ and satisfying $\partial_k T=0$ or $\partial^*_k T=0$.
  The same holds when $\manifold=\R^\Dim$.
\end{theorem}
\noindent
Here the Riesz potential acts on a current by the same integral as in \eqref{RieszPotential} with $\Delta_k$ replaced by $\Delta_{k,c}$.
We remark that \Cref{int01Bound0Boundary}, which was already stated in terms of currents, can be used along with \Cref{int1inftyBound} to give a direct proof for the case $k=1$, $\partial_1 T=0$.

\subsection{Special case: the torus and identification with vector fields}\label{ss:TorusProof}

\begin{proof}[Proof of \Cref{bbq_t}]
Let $\manifold = \mathbb{T}^3$.
In this setting, we can identify both 2-forms and 1-forms with vector-valued functions, and both 3-forms and 0-forms with scalar functions,
\begin{align*}
  C^\infty(\Lambda^2\mathbb{T}^3) \cong C^\infty(\Lambda^1\mathbb{T}^3) &\cong C^\infty(\mathbb{T}^3;\mathbb{R}^3)
  ,
  \\
  C^\infty(\Lambda^3\mathbb{T}^3) \cong C^\infty(\Lambda^0\mathbb{T}^3) &\cong C^\infty(\mathbb{T}^3;\R)
  .
\end{align*}
Under these identifications, the co-exterior derivative and exterior derivative on 2-forms correspond to the curl and divergence, respectively:
\begin{align*}
  \curl &\cong \ed_2^* \colon C^\infty(\Lambda^2\mathbb{T}^3) \to C^\infty(\Lambda^1\mathbb{T}^3)
  ,
  \\
  \Div &\cong \ed_2 \colon C^\infty(\Lambda^2\mathbb{T}^3) \to C^\infty(\Lambda^3\mathbb{T}^3)
  .
\end{align*}
Thus, identifying $Z$ as a 2-form and $F$ as a 1-form, the div-curl system \eqref{divcurlT3} reads
\begin{equation}
  \begin{aligned}
  \ed_2^* Z&= F
  ,
  \\
  \ed_2 Z &= 0
  .
\end{aligned}
\end{equation}
For 1-forms, the identification of curl and divergence with exterior derivative and co-exterior derivative is reversed,
\begin{align*}
  \curl &\cong \ed_1 \colon C^\infty(\Lambda^1\mathbb{T}^3) \to C^\infty(\Lambda^2\mathbb{T}^3)
  ,
  \\
  \Div &\cong \ed^*_1 \colon C^\infty(\Lambda^1\mathbb{T}^3) \to C^\infty(\Lambda^0\mathbb{T}^3)
  ,
\end{align*}
so the compatibility condition $\Div F=0$ can be expressed as $\ed^*_1 F=0$.
The harmonic 1-forms on $\mathbb{T}^3$ correspond to the constant vector fields, so the orthogonality condition from \Cref{bbq} reduces to the vector integral equality $\int_{\mathbb{T}^3} F\,\dd x = 0$.
We conclude that \Cref{bbq} applies with $k=2$, and yields
\begin{align*}
  \|Z\|_{L^{3/2,1}(\Lambda^2)} \lesssim \|F\|_{L^1(\Lambda^1)}
  .
\end{align*}
The identifications between forms and vector-valued functions are isometries, and therefore this inequality can be expressed in terms of $Z,F$ as vectors:
\begin{align*}
  \|Z\|_{L^{3/2,1}(\mathbb{T}^3;\mathbb{R}^3)} \lesssim \|F\|_{L^1(\mathbb{T}^3;\mathbb{R}^3)}
  .
\end{align*}
Finally, H\"older's inequality on the Lorentz scale (see \Cref{holder}) gives, for $a \in \mathbb{T}^3$,
\begin{align*}
  \int_{\mathbb{T}^3} \frac{|Z(x)|}{|x-a|}\,dx &\lesssim \|Z\|_{L^{3/2,1}(\mathbb{T}^3;\mathbb{R}^3)} \left\|\frac{1}{|\cdot-a|} \right\|_{L^{3,\infty}(\mathbb{T}^3;\mathbb{R}^3)}
  \\
  &\lesssim \|F\|_{L^1(\mathbb{T}^3;\mathbb{R}^3)}
\end{align*}
where the last inequality follows by checking that the weak-$L^3$ norm (or quasi-norm) of the scalar function $f(x)=1/\abs{x-a}$ is finite and independent of $a$.
\end{proof}

\section*{Acknowledgements}
J.~Goodman is supported by the Marsden Fund grants 20-UOO-079, 22-UOA-052, and 23-UOA-148, administered by the Royal Society Te Ap\=arangi, New Zealand.
D.~Spector is supported by the National Science and Technology Council of Taiwan under research grant number 113-2115-M-003-017-MY3 and the Taiwan Ministry of Education under the Yushan Fellow Program.

\appendix

\section{Supplementary facts about forms and currents}\label{appendix_A}

The choice of orientation on $\manifold$ can be expressed as a fixed choice of volume form $\nu\in\Lambda^\Dim$ and the corresponding Hodge star operator on $k$-forms, $\star_k\colon\Lambda^k\to\Lambda^{\Dim-k}$, defined by the requirement
\begin{equation}
  \tilde{\omega}(p) \wedge \star_k\omega(p) = \angles{\tilde{\omega}(p),\omega(p)}_p \nu(p)
\end{equation}
for all $\omega,\tilde{\omega}\in\Lambda^k$ and all $p\in\manifold$.
We note the identities
\begin{equation}\label{starInverse}
  \star_{\Dim-k}\star_k \omega = (-1)^{k(\Dim-k)}\omega
  \text{ for all }\omega\in\Lambda^k
  ,
  \quad\text{i.e.,}\quad
  \star_k^{-1} = (-1)^{k(\Dim-k)}\star_{\Dim-k}
  .
\end{equation}
Note that the inverse function $\star_k^{-1}\colon\Lambda^{\Dim-k}\to\Lambda^k$ accepts an $(\Dim-k)$-form, not a $k$-form as in our labelling elsewhere.

With this notation, the integral of a $\Dim$-form $\zeta\in\Lambda^\Dim$ over $\manifold$ can be expressed as
\begin{equation}
  \int_\manifold \zeta = \int_\manifold \star_\Dim\zeta \,\dd V
  ,
\end{equation}
with the $0$-form $\star_\Dim\zeta$ being interpreted unambiguously as a scalar function on $\manifold$, and with this convention
\begin{equation}
  \Angles{\omega,\tilde{\omega}} = \int_\manifold \omega\wedge \star_k\tilde{\omega}
\end{equation}
for $\omega,\tilde{\omega}\in\Lambda^k$.
Thus the $(\Dim-k)$-current $T_\omega$ from \eqref{FormAsCurrentWedge} can be expressed equivalently as
\begin{equation}\label{FormAsCurrentAngleStar}
  T_\omega(\xi) = \int_\manifold \xi \wedge \star_{\Dim-k}\star_{\Dim-k}^{-1}\omega
  = \int_\manifold \angles{\xi(p), \star_{\Dim-k}^{-1}\omega(p)}_p \dd V(p)
  = \Angles{\xi, \star_{\Dim-k}^{-1}\omega}
  .
\end{equation}

\begin{lemma}\label{DeltaAndStar}
  \begin{equation}
    \star_k \Delta_k \star^{-1}_k = \Delta_{\Dim-k}
    .
  \end{equation}
\end{lemma}
\begin{proof}
  The codifferential can be expressed in terms of the star operator by
  \begin{equation}\label{dAndStar}
    \ed^*_j = (-1)^j \star_{j-1}^{-1} \ed_{\Dim-j} \star_j
  \end{equation}
  so that \eqref{HodgeLaplacian} and \eqref{starInverse} give
  \begin{align*}
    -\star_k \Delta_k \star^{-1}_k
    &=
    \star_k \ed^*_{k+1} \ed_k \star^{-1}_k + \star_k \ed_{k-1} \ed^*_k \star^{-1}_k
    \\&
    = (-1)^{k+1} \ed_{\Dim-k-1}\star_{k+1} \ed_k \star^{-1}_k + (-1)^k \star_k \ed_{k-1} \star^{-1}_{k-1} \ed_{\Dim-k}
    \\&
    = (-1)^{k+1} (-1)^{(k+1)(\Dim-k-1)} (-1)^{k(\Dim-k)} \ed_{\Dim-k-1} \star^{-1}_{\Dim-k-1} \ed_k \star_{\Dim-k}
    \\&\quad
    + (-1)^k (-1)^{k(\Dim-k)} (-1)^{(k-1)(\Dim-k+1)} \star^{-1}_{\Dim-k} \ed_{k-1} \star_{\Dim-k+1} \ed_{\Dim-k}
    \\&
    = (-1)^{(2k+1)(\Dim-k)} (-1)^{\Dim-k} \ed_{\Dim-k-1} \ed^*_{\Dim-k}
    \\&\quad
    + (-1)^{(2k-1)(\Dim-k+1)} (-1)^{\Dim-k+1} \ed^*_{\Dim-k+1} \ed_{\Dim-k}
    \\&
    = -\Delta_{\Dim-k}
    .
    \qedhere
  \end{align*}
\end{proof}

We next prove \eqref{FormCurrentLaplaciansCompatible} and \Cref{TomegaPropagatorsCompatible}, which justifies the commutative diagrams from \Cref{ss:CurrentsFromForms}.
\begin{proof}[Proof of \eqref{FormCurrentLaplaciansCompatible} and \Cref{TomegaPropagatorsCompatible}]
  For \eqref{FormCurrentLaplaciansCompatible},
  \begin{align*}
    \Delta_{\Dim-k,c}T_\omega(\eta) &= T_\omega(\Delta_{\Dim-k}\eta)
    &&\text{by \eqref{LaplacianOnCurrents}}
    \\&
    = \Angles{\Delta_{\Dim-k}\eta, \star^{-1}_{\Dim-k}\omega}
    &&\text{by \eqref{FormAsCurrentAngleStar}}
    \\&
    = \Angles{\eta, \Delta_{\Dim-k}\star^{-1}_{\Dim-k}\omega}
    &&\text{by \eqref{HodgeLaplacianSelfAdjoint}}
    \\&
    = \Angles{\eta, \star^{-1}_{\Dim-k}\Delta_k\omega}
    &&\text{by \Cref{DeltaAndStar}}
    \\&
    = T_{\Delta_k\omega}(\eta)
    &&\text{by \eqref{FormAsCurrentAngleStar}}
    .
  \end{align*}
  The assertion of \eqref{FormCurrentLaplaciansCompatible} is equivalent to \eqref{FormCurrentLaplaciansCompatible2} by the definitions of $\Delta_{\Dim-k,c}$ and $T_\omega$, see \eqref{LaplacianOnCurrents} and \eqref{FormAsCurrentWedge}.

  Fix $\omega,\eta$ and let $h(s,t)=\int_\manifold e^{s\Delta_{\Dim-k}}\eta\wedge e^{t\Delta_k}\omega$.
  Then
  \begin{align*}
    \frac{\partial h}{\partial s}(s,t) &= \int_\manifold \Delta_{\Dim-k}e^{s\Delta_{\Dim-k}}\eta\wedge e^{t\Delta_k}\omega
    \\&
    = \int_\manifold e^{s\Delta_{\Dim-k}}\eta\wedge \Delta_k e^{t\Delta_k}\omega
    = \frac{\partial h}{\partial t}(s,t)
  \end{align*}
  and it follows that $h(s,t)$ depends only on $s+t$.
  In particular, using \eqref{PropagatorOnCurrentsInside},
  \begin{equation*}
    e^{t\Delta_{\Dim-k,c}}T_\omega(\eta) = T_\omega(e^{t\Delta_{\Dim-k}}\eta) = h(t,0) = h(0,t) = T_{e^{t\Delta_k}\omega}(\eta)
    .
    \qedhere
  \end{equation*}
\end{proof}

\begin{proof}[Proof of \Cref{NormDualityCurrentsForms}]
  The first equality of \ref{item:FormLebesgueDuality}, the duality pairing between $L^r(\Lambda^k)$ and $L^{r'}(\Lambda^k)$, follows from the usual one for scalar functions because, given any scalar function $\tilde{f}\colon\manifold\to\cointerval{0,\infty}$, we have
  \begin{equation*}
    \norm{\omega(p)}_p \tilde{f}(p) = \sup_{\tilde{\omega}\in\Lambda^k\colon \norm{\tilde{\omega}(q)}_q = \tilde{f}(q) \, \forall q\in\manifold} \angles{\omega(p),\tilde{\omega}(p)}_p
  \end{equation*}
  and the maximum is attained by setting $\tilde{\omega}(p) = \tilde{f}(p) \omega(p)/\norm{\omega(p)}_p$ whenever $\omega(p)\neq 0$.
  A standard approximation argument shows that the supremum is unaffected if we add the requirement that $\tilde{\omega}$ is continuous.
  The second equality of \ref{item:FormLebesgueDuality} follows by \eqref{FormAsCurrentAngleStar} and the substitution $\xi=\star_k\tilde{\omega}$, which is a pointwise isometry.

  The statement for the Lorentz norm follows for the same reason, except that the duality pairing between Lorentz norms is only up to constant factors, see \eqref{LorentzDual}.

  For part \ref{item:CurrentDuality}, recall the finite atlas $(x_j,U_j)_{j=1,\dotsc,\maxr}$ and the partition of unity $(\chi_j)_{j=1,\dotsc,\maxr}$.
  Define the currents $\tilde{T}_j$ on $\R^\Dim$ given by $\tilde{T}_j(\tilde{\xi}) = T(\chi_j \cdot (x_j)^\# \tilde{\xi})$.
  Then each $\tilde{T}_j$ satisfies bounds analogous to those for $T$, and we deduce that $\tilde{T}_j$ is given by integration against a Radon measure that is absolutely continuous with respect to Lebesgue measure.
  That is, we deduce that $\tilde{T}_j(\tilde{\xi}) = \int_{x_j(U'_j)} \angles{\smash{\tilde{\xi}}(y), w_j(y)} \,\dd y$ for some $k$-form $w_j$ on $\R^\Dim$ supported in $U'_j$, and furthermore $w_j$ is of class $L^{r'}$ or $L^{r',\infty}$.
  Finally, map back to obtain $\omega=\sum_{j=1}^\maxr x_j^\# w_j$.
\end{proof}

\begin{proof}[Proof of \Cref{CurrentEvolutionAsForm}]
  To establish that $e^{t\Delta_{j,c}}T$ has the form $T_{\omega_t}$, \Cref{NormDualityCurrentsForms} shows that it is sufficient to bound $\abs{e^{t\Delta_{j,c}}T(\xi)}$ in terms of $\norm{\xi}_{L^1}$ for $\xi\in C(\Lambda^j)$.
  To that end, recall the heat kernel $\mathrm{p}_t^j$ and its approximation $\mathrm{p}_t^{j,N}$ from \eqref{k_heat_kernel} and \eqref{k_kernel_approx}.
  The functions $V_{j,i}$ are smooth on the compact manifold $\manifold\times\manifold$ and therefore uniformly bounded, so together with the uniform error bound \eqref{k_kernel_error} we have
  \begin{equation}\label{HeatKernelBoundedByK}
    \bignorm{\mathrm{p}_t^j(p,q)}_{p,q} \lesssim K_t(p,q)
    ,
  \end{equation}
  uniformly over $t\in\ocinterval{0,1}$ and $p,q\in\manifold$.
  That is, we can take the constant $C$ large enough so that
  \begin{equation*}
    \bignorm{\mathrm{p}_t^j(p,q)}_{p,q} \leq C K_t(p,q)
  \end{equation*}
  for all such $t,p,q$.
  Then, applying Minkowski's inequality to the integral \eqref{k_heat_kernel},
  \begin{align}\label{heatpropagator_bound}
    \norm{e^{t\Delta_j}\xi(p)}_p
    &\leq \int_\manifold \bignorm{\mathrm{p}_t^j(p,q)}_{p,q} \norm{\xi(q)}_q \dd V(q)
    \notag\\&
    \leq C\int_\manifold K_t(p,q) \norm{\xi(q)}_q \dd V(q)
    ,
  \end{align}
  so that
  \begin{align}
    \abs{e^{t\Delta_{j,c}}T(\xi)} &= \abs{T(e^{t\Delta_j}\xi)}
    \notag\\
    &\leq \int_\manifold \norm{e^{t\Delta_j}\xi(p)}_p \dd\mu_T(p)
    \notag\\
    &\leq C \int_{\manifold} \int_\manifold K_t(p,q) \norm{\xi(q)}_q \dd V(q) \, \dd\mu_T(p)
    \notag\\
    &= C \int_{\manifold} \norm{\xi(q)}_q \left(\int_\manifold K_t(p,q) \, \dd\mu_T(p)\right) \dd V(q)
    .
    \label{AbsetDeltaTxiBound}
  \end{align}
  In particular, $\abs{e^{t\Delta_{j,c}}T(\xi)} \leq \tilde{C} \norm{\xi}_{L^1}$ for all $\xi\in C(\Lambda^j)$, where $\tilde{C}=\tilde{C}(\manifold,T)<\infty$.
  By \Cref{NormDualityCurrentsForms} we can identify $e^{t\Delta_{j,c}}T=T_{\omega_t}$ for some $\omega_t \in L^\infty(\Lambda^{\Dim-j})$.
  By \Cref{TomegaPropagatorsCompatible} and the semigroup property,
  \begin{align*}
    e^{t\Delta_{j,c}}T &= e^{(t/2)\Delta_{j,c}}e^{(t/2)\Delta_{j,c}}T\\
    &=e^{(t/2)\Delta_{j,c}}T_{\omega_{t/2}}\\
    &= T_{e^{(t/2)\Delta_{\Dim-j}} \omega_{t/2}}
    ,
  \end{align*}
  and since the heat propagator on $(\Dim-j)$-forms is smoothing, this shows that $e^{t\Delta_{j,c}}T$ admits representation as a smooth $(\Dim-j)$-form, necessarily unique.
  We may therefore assume without loss of generality that $\omega_t$ is this smooth $(\Dim-j)$-form, for all $t>0$.

  Finally, recalling \eqref{FormAsCurrentAngleStar} and \eqref{AbsetDeltaTxiBound}, we have
  \begin{align*}
    \abs{\int_\manifold \angles{\xi(p),\star_{\Dim-j}^{-1}\omega_t(p)} \dd V(p)}
    &= \abs{T_{\omega_t}(\xi)} = \abs{e^{t\Delta_{j,c}}T(\xi)}
    \\&
    \leq C \int_{\manifold} \norm{\xi(q)}_q \left(\int_\manifold K_t(p,q) \, \dd\mu_T(p)\right) \dd V(q)
  \end{align*}
  for all $\xi\in C(\Lambda^j)$.
  By Lebesgue's differentiation theorem we deduce
  \begin{align*}
    \norm{\omega_t(p)}_p = \norm{\star^{-1}_{\Dim-j}\omega(p)}_p \leq C\int_\manifold K_t(p,q)\,\dd\mu_T(q)
  \end{align*}
  for almost every $p \in \manifold$.
  Since $\omega_t$ is smooth, this inequality holds for all $p \in \manifold$.
\end{proof}

In \Cref{d1etDeltaxiAsIntegral}, we apply one spatial derivative to an integral against the kernel $\mathrm{p}_t^1$, which we can interpret as yielding a spatial derivative applied to the kernel.
Before proceeding to the proof, we recall the Gaussian-type approximation to $\mathrm{p}_t^1$ in \eqref{k_kernel_approx}, and we note that a spatial derivative of such an expression has scaling comparable to what we would obtain in Euclidean space.
\begin{lemma}\label{DerivOfGaussianTimesSmooth}
  Fix a smooth scalar function $g\colon\manifold\times\manifold\to\R$ with support sufficiently close to the diagonal so that $(p,q)\mapsto d_\manifold(p,q)^2$ is a smooth function on the support of $g$.
  Define smooth scalar functions $f_q\colon\manifold\to\R$ by
  \begin{equation*}
    f_q(p) = \exp\left( -d_\manifold(p,q)^2/(4t) \right) g(p,q)
    .
  \end{equation*}
  Then the derivatives of $f_q$ (expressed as a 1-form) have a uniform bound
  \begin{equation*}
    \norm{\ed f_q(p)}_p \leq c(g,\manifold) \frac{d_\manifold(p,q) + t}{t} \exp\left( -d_\manifold(p,q)^2/(4t) \right)
  \end{equation*}
  over all $p,q\in\manifold$ and $t>0$, where $c(g,\manifold)<\infty$.
\end{lemma}
\begin{proof}
  This follows by using geodesic normal coordinates around $q$, so that $d_\manifold(p,q)^2$ reduces to the sum of squares of the coordinates of $p$, for $p$ sufficiently close to $q$.
\end{proof}

In the proof of \Cref{d1etDeltaxiAsIntegral}, we are effectively using an analogue of \Cref{DerivOfGaussianTimesSmooth} in which the functions are no longer scalar-valued.
For definiteness, however, we will use coordinates to reduce to the scalar case.

\begin{proof}[Proof of \Cref{d1etDeltaxiAsIntegral}]
  We again use the asymptotic series approximation from \eqref{k_kernel_approx}, this time fixing $N\geq (\Dim+1)/2$.
  In the notation of \cite{BerGetVer1992}*{Theorem~2.30 with $\ell=1,k=0$}, we have the bound
  \begin{equation}\label{pt1Approx1norm}
    \bignorm{\mathrm{p}_t^1 - \mathrm{p}_t^{1,N}}_1 \lesssim 1
    ,
  \end{equation}
  where the norm $\norm{\cdot}_1$ means the supremum of values and partial derivatives up to order 1 for a function of class $C^1$, see \cite{BerGetVer1992}*{pp.~70--71}.

  In this case, the functions $\mathrm{p}_t^1,\mathrm{p}_t^{1,N}$ are smooth functions on $\manifold\times\manifold$ with values in $\union_{p,q\in\manifold}\mathrm{Hom}(T_q^*\manifold, T_p^*\manifold)$.
  For an interpretation of \eqref{pt1Approx1norm} concretely in coordinates, recall the coordinate charts $(x_r,U_r)$ and partition of unity $\chi_r$, $r\in\set{1,\dotsc,\maxr}$, from \Cref{ss:FiniteAtlas}, and write $\xi_t=\sum_{i=1}^\Dim \xi_{t,r,i} \ed x_r^i$ with $\xi=\xi_0$.
  Then each mapping $\xi\mapsto \chi_r e^{t\Delta_1}(\chi_{r'}\xi)$ can be represented in coordinates by a smooth kernel, and
  \begin{equation}\label{xiIntegralCoordinates}
    \xi_{t,r,i}(p) = \sum_{r' = 1}^\maxr \sum_{j=1}^\Dim \int_{U_{r'}} w_{t,r,r',i,j}^1(p,q) \xi_{0,r',j}(q) \, \dd V(q)
    \quad\text{for }p\in U_r
    ,
  \end{equation}
  where $w_{t,r,r',i,j}^1$ are smooth scalar functions.
  If $w_{t,r,r',i,j}^{1,N}$ are the corresponding functions for $\mathrm{p}_t^{1,N}$, then \eqref{pt1Approx1norm} implies that
  \begin{equation*}
    \abs{w_{t,r,r',i,j}^1(p,q) - w_{t,r,r',i,j}^{1,N}(p,q)}\lesssim 1
    ,\quad
    \abs{\frac{\partial w_{t,r,r',i,j}^1}{\partial x_r(p)^i}(p,q) - \frac{\partial w_{t,r,r',i,j}^{1,N}}{\partial x_r(p)^i}(p,q)}\lesssim 1
    ,
  \end{equation*}
  where $\frac{\partial}{\partial x_r(p)^i}$ means partial differentiation in $\manifold\times\manifold$ under the coordinate mapping $(p,q)\mapsto (x_r(p),x_{r'}(q))\in\R^{2\Dim}$ on $U_r\times U_{r'}$.
  Moreover the scalar functions $w_{t,r,r',i,j}^{1,N}$ also have representations of the form \eqref{k_kernel_approx} as finite sums of powers of $t$ times functions to which \Cref{DerivOfGaussianTimesSmooth} applies.
  (Note that we can freely shrink the support of the functions $V_{k,i}$ from \eqref{k_kernel_approx} to lie close enough to the diagonal so that \Cref{DerivOfGaussianTimesSmooth} applies, without affecting the error bounds from \cite{BerGetVer1992}*{Theorem~2.30}, and indeed their derivation already performs such shrinking.)
  For $0<t\leq 1$, the most singular such power is $t^{-\Dim/2}$, and we conclude that
  \begin{equation*}
    \abs{\frac{\partial w_{t,r,r',i,j}^1}{\partial x_r(p)^i}(p,q)} \lesssim 1 + t^{-\Dim/2} \frac{d_\manifold(p,q) + t}{t} \exp\left( -d_\manifold(p,q)^2/(4t) \right) = \tilde{K}_t(p,q)
    .
  \end{equation*}

  Now take the partial derivatives inside the integral in \eqref{xiIntegralCoordinates} to find
  \begin{equation*}
    \frac{\partial \xi_{t,r,i}}{\partial x_r^k}(p) = \sum_{r' = 1}^\maxr \sum_{j=1}^\Dim \int_{U_{r'}} \frac{\partial w_{t,r,r',i,j}^1}{\partial x_r^k}(p,q) \xi_{0,r',j}(q) \, \dd V(q)
    .
  \end{equation*}
  The $2$-form $\ed_1\xi_t$ can be expressed in terms of a linear combination of such partial derivatives, and therefore
  \begin{align*}
    \norm{\zeta(p)}_p = \norm{\ed_1 \xi_t(p)}_p
    &\lesssim \sum_{r=1}^R \sum_{i=1}^\Dim \sum_{k=1}^\Dim \int_{U_{r'}} \abs{\frac{\partial w_{t,r,r',i,j}^1}{\partial x_r(p)^i}(p,q)} \abs{\xi_{0,r',j}(q)} \dd V(q)
    \\&
    \lesssim \int_\manifold \tilde{K}_t(p,q) \norm{\xi(q)}_q \dd V(q)
    .
    \qedhere
  \end{align*}
\end{proof}

We remark that the proof above could be used to justify the following more succinct coordinate-free approach.
The mapping $\xi\mapsto \ed_1(e^{t\Delta_1}\xi)$ is given by a kernel $\tilde{\mathrm{p}}^{1\to 2}_t(p,q)$ with values in $\mathrm{Hom}(T_q^*\manifold, \wedge^2(T_p^*\manifold))$, obtained by taking the exterior derivative of $\mathrm{p}^1_t(p,q)$ with respect to the variable $p$.
The error estimate \eqref{pt1Approx1norm} leads to $\bignorm{\tilde{\mathrm{p}}^{1\to 2}_t(p,q) - \tilde{\mathrm{p}}^{1\to 2,N}_t(p,q)}_{p,q}\lesssim 1$, where $\tilde{\mathrm{p}}^{1\to 2,N}_t$ is the analogous quantity obtained from $\mathrm{p}_t^{1,N}$.
The analogue of \Cref{DerivOfGaussianTimesSmooth} applies when taking exterior derivatives instead of partial derivatives.
Then $\bignorm{\tilde{\mathrm{p}}^{1\to 2}_t(p,q)}_{p,q}\lesssim \tilde{K}_t(p,q)$ follows by combining all of these bounds, thus proving \Cref{d1etDeltaxiAsIntegral}.

\section{Besov-Lorentz spaces on a Riemannian manifold}\label{appendix_B}

\subsection{Lorentz Spaces}
Let $(X, \mu)$ be a measure space.
For $f$ a measurable function on $X$, we define
\begin{align*}
  m(f,y):= \mu(\{ |f|>y\})
  .
\end{align*}
The non-increasing function $y\mapsto m(f,y), [0,\infty]\to[0,\infty],$ admits a left-continuous inverse $f^*\colon(0,\infty)\to [0,\infty]$, called the non-increasing rearrangement of $f$.
For $t\in(0,\infty)$ we further define
\begin{align*}
  f^{**}(t):= \frac{1}{t}\int_0^t f^*(u)\, \dd u
  .
\end{align*}

\begin{definition}\label{LorentzDefinition}
  For $1<r<\infty$ and $1\leq s<\infty$, define
  \begin{align}\label{LorentzNorm_sFinite}
    \|f\|_{L^{r,s}(X)} := \left( \int_0^\infty \left[t^{1/r} f^{**}(t)\right]^s\frac{\dd t}{t}\right)^{1/s}
    ,
  \end{align}
  and for $1<r\leq\infty$ and $s=\infty$
  \begin{align}\label{LorentzNorm_sInfinite}
    \|f\|_{L^{r,\infty}(X)} := \sup_{t>0} t^{1/r} f^{**}(t)
    .
  \end{align}
  We also define 
  \begin{equation}\label{LorentzNorm_r1}
    \norm{f}_{L^{1,1}(X)}=\norm{f}_{L^1(X)}=\int_0^\infty f^*(u)\,\dd u
    .
  \end{equation}
  For those values of $(r,s)$, the Lorentz space $L^{r,s}(X)$ is given by
  \begin{align*}
    L^{r,s}(X) := \left\{ f \text{ measurable} \colon \|f\|_{L^{r,s}(X)} <\infty\right\}
    .
  \end{align*}
\end{definition}
For this range of parameters $r,s$, the functionals $\norm{\cdot}_{r,s}$ are norms and the associated spaces $L^{r,s}(X)$ are Banach spaces, see, e.g., \cite{SteinWeissbook}*{Chapter V}.
For these spaces, there is a duality between $L^{r,s}(X)$ and $L^{r',s'}(X)$ for $1<r<\infty$ and $1\leq s < \infty$, see for example \cite{grafakos}*{Theorem 1.4.16}. 
One consequence is that
\begin{equation}\label{LorentzDual}
  \| f\|_{L^{r,s}(X)} \asymp \sup \left\{ \int_{X} fg \,\dd\mu \colon g \in L^{r',s'}(X) \text{ and } \|g \|_{L^{r',s'}(X)}\leq 1\right\}
  ,
\end{equation}
where $a\asymp b$ means $a\lesssim b$ and $b\lesssim a$.

Instead of the norm $\norm{\cdot}_{L^{r,s}(X)}$, it is often easier to deal with
\begin{align}\label{eq:quasinorm}
  \vertiii{f}_{L^{r,s}(X)} \equiv \left(r \int_0^\infty \left(y \mu(\set{p\in X\colon |f(p)|>y })^{1/r}\right)^{s} \frac{\dd y}{y}\right)^{1/s}
\end{align}
for $1\leq r<\infty$ and $1\leq s<\infty$, with
\begin{equation*}
  \vertiii{f}_{L^{r,\infty}(X)}=\sup_{y>0} y \mu(\set{p\in X\colon |f(p)|>y })^{1/r}
  .
\end{equation*}
We note first that while $\vertiii{\cdot}_{L^{r,s}(X)}$ is not a norm in general, it is a quasi-norm equivalent to the norm $\norm{\cdot}_{L^{r,s}(X)}$.

\begin{proposition}[\cite{SteinWeissbook}*{Theorem~3.21 on p.~204}]\label{norm_qnorm_equivalence}
  Let $1<r<\infty$ and $1\leq s \leq \infty$.
  Then
  \begin{align*}
    \vertiii{f}_{L^{r,s}(X)} \leq \|f\|_{L^{r,s}(X)}\leq r' \vertiii{f}_{L^{r,s}(X)}
    .
  \end{align*}
\end{proposition}
\noindent
We remark that the norm $\norm{\cdot}_{L^{r,s}(X)}$, as defined in \eqref{LorentzNorm_sFinite}, becomes unstable when $r\decreasesto 1$ and $s<\infty$, but the quasi-norm $\vertiii{\cdot}_{L^{r,s}(X)}$ does not have this problem and indeed $\vertiii{\cdot}_{L^{r,r}(X)}=\norm{\cdot}_{L^r(X)}$ for all $1\leq r<\infty$.

A version of H\"older's inequality relates norms from different Lorentz spaces: see \cite{oneil}*{Theorem~3.4} and the proof in \cite{MS}*{Theorem~2.6 on p.~883}.
\begin{proposition}\label{holder}
  Let $f \in L^{r_1,s_1}(X)$ and $g \in L^{r_2,s_2}(X)$, where
  \begin{align*}
    \frac{1}{r_1}+\frac{1}{r_2}&=\frac{1}{r}<1
    ,
    \\
    \frac{1}{s_1}+\frac{1}{s_2}&\geq  \frac{1}{s}
    ,
  \end{align*}
  for some $r > 1$.
  Then
  \begin{align*}
    \|fg\|_{L^{r,s}(X)} \leq r'\|f \|_{L^{r_1,s_1}(X)}\|g \|_{L^{r_2,s_2}(X)}
    .
  \end{align*}
\end{proposition}

A version of H\"older's inequality also holds for the $L^1$ norm.
The following statement, whose validity follows from a rearrangement inequality of Hardy and Littlewood, will be sufficient for our purposes.
\begin{proposition}\label{holder_prime}
Suppose 
  \begin{align*}
    \frac{1}{r_1}+\frac{1}{r_2}&=1
  \end{align*}
with $(r_1,r_2)\neq(\infty,1)$.
Then
  \begin{align*}
    \|fg\|_{L^1(X)} \leq  \|f \|_{L^{r_1,1}(X)}\|g \|_{L^{r_2,\infty}(X)}
  \end{align*}
for all $f \in L^{r_1,1}(X)$ and $g \in L^{r_2,\infty}(X)$.
\end{proposition}

Finally, we give a proof of the basic interpolation inequality asserted in \Cref{Interpolation} and state and prove a similar result for operators.
Both of these can be derived from a general result \cite{grafakos}*{Theorem 1.4.19}, but we give direct proofs for the reader's interest.

\begin{proof}[Proof of \Cref{Interpolation}]
  The inequality is trivial if $\norm{f}_{L^\infty}=0$, so assume $\norm{f}_{L^\infty}>0$.
  Noting that $\int_0^\infty f^*(u)\,\dd u = \norm{f}_{L^1}$, we have $f^{**}(t)\leq \norm{f}_{L^1}/t$ for all $t>0$.
  In addition, $f^*(t)\leq \norm{f}_{L^\infty}$ and consequently $f^{**}(t)\leq\norm{f}_{L^\infty}$ for all $t>0$.
  Write $t_c=\norm{f}_{L^1}/\norm{f}_{L^\infty}$ for the point where these upper bounds coincide.
  Then, for $1<r<\infty$ and $1\leq s<\infty$,
  \begin{align*}
    \|f \|_{L^{r,s}(X)}^s &= \int_0^\infty [t^{1/r}f^{**}(t)]^s \frac{\dd t}{t}
    \\&
    \leq \int_0^{t_c} t^{s/r-1} \norm{f}_{L^\infty}^s \,\dd t + \int_{t_c}^\infty t^{s/r-s-1} \norm{f}_{L^1}^s \dd t
    \\&
    = \frac{r}{s} t_c^{s/r} \norm{f}_{L^\infty}^s + \frac{1}{s-s/r} t_c^{-(s-s/r)} \norm{f}_{L^1}^s
    \\&
    = \left( \frac{r}{s} + \frac{1}{s-s/r} \right) \norm{f}_{L^1}^{s/r} \norm{f}_{L^\infty}^{s-s/r}
    .
  \end{align*}
  We have therefore shown the general inequality
  \begin{equation}\label{GeneralLorentzInterpolation}
    \|f \|_{L^{r,s}(X)} \leq \frac{1}{s^{1/s}} \left( \frac{1}{1/r} + \frac{1}{1-1/r} \right)^{1/s} \norm{f}_{L^1}^{1/r} \norm{f}_{L^\infty}^{1-1/r}
  \end{equation}
  and \Cref{Interpolation} is the special case $s=1$, $1/r = 1-\alpha/\Dim$.
\end{proof}
\noindent
We remark that the corresponding inequality for the quasi-norm can be obtained via H\"older's inequality, similar to the classical proof for the Lebesgue interpolation statement, whereupon \Cref{Interpolation} can be derived using \Cref{norm_qnorm_equivalence}.

For later use, we record a similar interpolation result for operators.

\begin{lemma}\label{GrafakosLorentzOperatorNorms}
  Fix $1\leq r<\infty$.
  Then, uniformly over linear operators $T\colon L^1(X,\mu)+L^\infty(X,\mu) \to L^1(\tilde{X},\tilde{\mu}) + L^\infty(\tilde{X},\tilde{\mu})$,
  \begin{equation}\label{GrafakosLr1Lr1}
    \norm{T}_{L^{r,1}\to L^{r,1}} \lesssim \norm{T}_{L^1\to L^1}^{1/r} \norm{T}_{L^\infty\to L^\infty}^{1-1/r}
    .
  \end{equation}
  More generally, for $1\leq r\leq p<\infty$ fixed,
  \begin{equation}\label{GrafakosLr1Lp1}
    \norm{T}_{L^{r,1}\to L^{p,1}} \lesssim \norm{T}_{L^1\to L^1}^{1/p}  \norm{T}_{L^1\to L^\infty}^{1/r - 1/p} \norm{T}_{L^\infty\to L^\infty}^{1-1/r}
    .
  \end{equation}
\end{lemma}
\noindent 
The first inequality \eqref{GrafakosLr1Lr1} is a special case of \cite{grafakos}*{Theorem 1.4.19}.
The second inequality \eqref{GrafakosLr1Lp1} can be obtained by applying \cite{grafakos}*{Theorem 1.4.19} twice, but we record the following direct proof using \Cref{Interpolation}.

\begin{proof}
  Let $f\colon X\to\R$ be a measurable function and define $A_y=\set{x\in X\colon \abs{f(x)}>y}$ for $y\geq 0$.
  Then by definition
  \begin{equation}\label{Lr1QuasinormmuAy}
    \vertiii{f}_{L^{r,1}} = r \int_0^\infty \mu(A_y)^{1/r} \,\dd y
  \end{equation}
  and 
  \begin{equation}\label{absfxIntegral}
    \abs{f(x)} = \int_0^\infty \indicator{x\in A_y} \,\dd y
    .
  \end{equation}
  To make use of the operator norms $\norm{T}_{L^1\to L^1}$, $\norm{T}_{L^1\to L^\infty}$, and $\norm{T}_{L^\infty\to L^\infty}$, rewrite \eqref{absfxIntegral} to express $f$ as the integral with respect to the measure appearing in \eqref{Lr1QuasinormmuAy} of functions in $L^1\intersect L^\infty$:
  \begin{equation}
    f = \int_0^\infty h_y \, \mu(A_y)^{1/r} \,\dd y
    \quad\text{$\mu$-a.e., where}\quad
    h_y = \frac{\indicatorofset{A_y}}{\mu(A_y)^{1/r}} \operatorname{sgn}(f)
    ,
  \end{equation}
  with $h_y$ interpreted as the zero function (modulo equality $\mu$-a.e.) if $\mu(A_y)=0$.
  By Minkowski's inequality for integrals,
  \begin{equation}\label{TfLp1Integral}
    \norm{Tf}_{L^{p,1}} \leq \int_0^\infty \norm{T h_y}_{L^{p,1}} \mu(A_y)^{1/r} \,\dd y
    .
  \end{equation}
  From the operator norms of $T$ we obtain three bounds
  \begin{subequations}
    \begin{align}
      \norm{T h_y}_{L^1} &\leq \norm{T}_{L^1\to L^1} \norm{h}_{L^1} = \norm{T}_{L^1\to L^1} \mu(A_y)^{1-1/r}
      ,
      \label{ThyL1L1}
      \\
      \norm{T h_y}_{L^\infty} &\leq \norm{T}_{L^1\to L^\infty} \norm{h}_{L^1} = \norm{T}_{L^1\to L^\infty} \mu(A_y)^{1-1/r}
      ,
      \label{ThyL1Linfty}
      \\
      \norm{T h_y}_{L^\infty} &\leq \norm{T}_{L^\infty\to L^\infty} \norm{h}_{L^\infty} = \norm{T}_{L^1\to L^\infty} \mu(A_y)^{-1/r}
      .
      \label{ThyLinftyLinfty}
    \end{align}
  \end{subequations}
  The \InterpolationLemma, or \eqref{GeneralLorentzInterpolation}, applied to \eqref{ThyL1L1} paired with each of \eqref{ThyL1Linfty} and \eqref{ThyLinftyLinfty} yields two complementary bounds for $\norm{T h_y}_{L^{p,1}}$:
  \begin{equation}\label{ThyLp1TwoBounds}
    \norm{T h_y}_{L^{p,1}} \lesssim 
    \left\lbrace 
    \begin{aligned}
      &\norm{T}_{L^1\to L^1}^{1/p} \norm{T}_{L^1\to L^\infty}^{1-1/p} \mu(A_y)^{1-1/r}
      ,
      \\
      &\norm{T}_{L^1\to L^1}^{1/p} \norm{T}_{L^\infty\to L^\infty}^{1-1/p} \mu(A_y)^{-(1/r-1/p)}
      .
    \end{aligned}
    \right.
  \end{equation}
  (If $p=r=1$, \eqref{GeneralLorentzInterpolation} and \Cref{Interpolation} do not apply, but then $\norm{\cdot}_{L^{1,1}}=\norm{\cdot}_{L^1}$ and \eqref{ThyLp1TwoBounds} holds trivially.)
  The first upper bound is non-decreasing in $\mu(A_y)$ for all $r\geq 1$, and the second is non-increasing since $p\geq r$.
  The cross-over occurs at $\mu(A_y)=\norm{T}_{L^\infty\to L^\infty}/\norm{T}_{L^1\to L^\infty}$, and in particular there is a uniform bound
  \begin{equation}\label{supyThyLp1}
    \sup_{y\geq 0} \norm{T h_y}_{L^{p,1}} \leq \norm{T}_{L^1\to L^1}^{1/p}  \norm{T}_{L^1\to L^\infty}^{1/r - 1/p} \norm{T}_{L^\infty\to L^\infty}^{1-1/r}
    .
  \end{equation}
  By \eqref{TfLp1Integral} and \eqref{Lr1QuasinormmuAy},
  \begin{equation*}
    \norm{T f}_{L^{p,1}}
    \leq \sup_y \norm{T h_y}_{L^{p,1}} \int_0^\infty \mu(A_y)^{1/r}\,\dd y
    \lesssim \sup_y \norm{T h_y}_{L^{p,1}} \vertiii{f}_{L^{r,1}}
    ,
  \end{equation*}
  and by \eqref{supyThyLp1} and \Cref{norm_qnorm_equivalence} this proves \eqref{GrafakosLr1Lp1} for $r>1$.
  The case $r=1$ also follows because $\vertiii{f}_{L^{1,1}}=\norm{f}_{L^{1,1}}=\norm{f}_{L^1}$, and \eqref{GrafakosLr1Lr1} is the special case $p=r$.
\end{proof}

We remark that the \InterpolationLemma\ can be derived from \Cref{GrafakosLorentzOperatorNorms}, and thus as a special case of \cite{grafakos}*{Theorem 1.4.19}, by taking $X=\set{x_0}$ with $\mu(X)=1$ and $Tf=f(x_0)F$ for $F\colon\tilde{X}\to\R$, cf.\ \cite{BerghLofstrom1976}*{proof of Theorem 3.9.1 on p.~58}.

\subsection{Besov-Lorentz Spaces}

In \Cref{sss:BesovLorentzDiscussion}, we asserted that our work establishes an embedding of closed or co-closed $k$-forms of finite mass into certain Besov-Lorentz spaces: see \Cref{BesovClosedCoclosedReformulated}.
In the Euclidean setting, Besov spaces have many equivalent definitions, for example in terms of the harmonic extension \cite[p.~152]{Stein1970}, differences, interpolation, frames, local means, harmonic and thermic extension \cite[p.~8, 40, 54, 131, 138, 152]{Triebel1992}.
Besov-Lorentz spaces are less common, with a definition by Peetre \cite{Peetre} and further study by  Seeger and Trebels \cite{ST}.
On a manifold, the various definitions of Besov spaces are somewhat involved, see e.g.\ \cite{Triebel}, \cite[7.3]{Triebel1992}, and \cite[p.~1055]{CKP}, and as far as the authors are aware, a definition of Besov-Lorentz spaces has not appeared in the literature.  
We begin this section by showing the definition of Besov-Lorentz spaces in \cite{ST} agrees with an alternate definition utilizing the heat propagator.
This result is a homogeneous Lorentz extension of the corresponding equivalence catalogued in \cite[p.~152]{Triebel1992} and motivates our adoption of an analogous definition of Besov-Lorentz spaces in the manifold setting.
We then give a self-contained treatment of some properties of these spaces.
It is beyond the scope of this paper to establish whether  this definition would be equivalent to one in the spirit of \cite[7.3]{Triebel1992} or  \cite[p.~1055]{CKP}.

We next outline the contents of the remainder of the appendix.
We show that the traditional Euclidean Besov-Lorentz norm is equivalent to another norm defined in terms of the Euclidean heat propagator: see \Cref{other_direction,other_direction_extended}.
The definition of Besov-Lorentz spaces can therefore be generalized to the setting of a compact Riemannian manifold, and we show that they satisfy the same mapping properties under the Hodge Laplacian as in Euclidean space: see \Cref{Besov_manifold_definition,Besov_homeomorphisms,Besov-nested}.

In Euclidean space, for $\beta \in \mathbb{R}$ and $1\leq r<\infty$, the Besov-Lorentz space $\dot{B}^{\beta,1}_{r,1}(\mathbb{R}^\Dim)$ is traditionally defined as the completion of the smooth compactly supported functions with respect to the norm
\begin{align}\label{BLnorm}
  \|f\|_{\dot{B}^{\beta,1}_{r,1}(\mathbb{R}^\Dim)}:= \sum_{j \in \mathbb{Z}} 2^{\beta j} \|f\ast (\psi_{2^{j+1}}-\psi_{2^j})\|_{L^{r,1}(\mathbb{R}^\Dim)}
  ,
\end{align}
where
\begin{align*}
 \psi_{2^j}(x)= 2^{j\Dim} \psi\left(2^j x\right)
\end{align*}
are dilates of some function $\psi \in \mathcal{S}(\mathbb{R}^\Dim)$ whose Fourier transform $\widehat{\psi}$ satisfies
\begin{align*}
  \operatorname*{supp} \widehat{\psi} &\subset B(0,1)
  ,
  \\
  \widehat{\psi} \equiv 1 &\text{ on }  B(0,1/2)
  .
\end{align*}
It can be shown that this space is equivalent to the quotient space of tempered distributions modulo polynomials $f \in \mathcal{S}'(\mathbb{R}^\Dim) / \mathcal{P}(\mathbb{R}^\Dim)$ such that \eqref{BLnorm} is finite.
We refer the reader to \cite{AH}*{Definition 4.1.2} for a similar definition of Besov-Lebesgue spaces, and Example 6 on p.~57 as well as pp.~106 and 232 in  Peetre's monograph \cite{Peetre} for an explicit treatment in classical literature.
For a modern treatment, we refer the reader to Seeger and Trebels \cite{ST}*{equation (1) on p.~1018}: they deal primarily with an inhomogeneous version denoted by $B^0_1[L^{r,1}]$, in which the sum in \eqref{BLnorm} runs over $j\geq 0$.
However, they comment at the end of the introduction \cite{ST}*{p.~1020} that their results apply also to the space corresponding to \eqref{BLnorm}, which they call homogeneous and denote by $\dot{B}^0_1[L^{r,1}]$.
Some of these definitions use different choices of $\psi$, i.e., different Littlewood-Paley decompositions, and may use the quasi-norm $\vertiii{\cdot}_{L^{r,1}(\R^\Dim)}$ instead of $\norm{\cdot}_{L^{r,1}(\R^\Dim)}$, but are otherwise equivalent.

Before continuing, we note a key property of the Besov spaces: the Laplacian and heat propagator map between different values of the parameter $\beta$.
\begin{proposition}[\cite{BerghLofstrom1976}*{Theorem~6.2.7}]\label{Besov_homeomorphism_Euclidean}
  Let $\beta \in \mathbb{R}$ and $1\leq r<\infty$.
  The mappings
  \begin{align*}
    -\Delta_k &\colon \dot{B}^{\beta,1}_{r,1}(\R^\Dim) \to \dot{B}^{\beta-2,1}_{r,1}(\R^\Dim)
    \\
    \mathcal{I}_{\alpha,k}&\colon \dot{B}^{\beta,1}_{r,1}(\R^\Dim) \to \dot{B}^{\beta+\alpha,1}_{r,1}(\R^\Dim)
  \end{align*}
  are homeomorphisms.
\end{proposition}
\noindent
We note that in the text \cite{BerghLofstrom1976} the conclusion of \Cref{Besov_homeomorphism_Euclidean} is asserted only for the Besov-Lebesgue spaces.
However, both that result and \cite{BerghLofstrom1976}*{Lemma~6.2.1}, on which its proof relies, can be adaptated straightforwardly to Lorentz spaces.

We show next that the Euclidean Besov-Lorentz norm for $\beta<0$ has an equivalent norm defined in terms of the Euclidean heat kernel, similar to the integral in \Cref{BesovClosedCoclosed}.

\begin{theorem}\label{other_direction}
  Let $\alpha>0$ and $1\leq r<\infty$.
  For $f \in \dot{B}^{-\alpha,1}_{r,1}(\mathbb{R}^\Dim)$,
  \begin{equation*}
    \|f\|_{\dot{B}^{-\alpha,1}_{r,1}(\mathbb{R}^\Dim)} \asymp \int_0^\infty t^{\alpha/2-1} \|e^{t\Delta} f\|_{L^{r,1}(\mathbb{R}^\Dim)}\,\dd t
    .
  \end{equation*}
\end{theorem}
\begin{proof}
In \cite{HRS}*{Equation (2.7) on p.~1935}, it is shown that
\begin{align}\label{Stolyarov_bound}
  \|\mathcal{I}_\alpha f\|_{\dot{B}^{0,1}_{r,1}(\mathbb{R}^\Dim)} &\lesssim \int_0^\infty t^{\alpha/2-1} \|e^{t\Delta} f\|_{L^{r,1}(\mathbb{R}^\Dim)}\,\dd t
  .
\end{align}
(Their statement is for the case $r=\Dim/(\Dim-\alpha)$, but the same proof applies for all $r\geq1$.)
By \Cref{Besov_homeomorphism_Euclidean}, the Riesz potential is a homeomorphism $I_\alpha\colon \dot{B}^{-\alpha ,1}_{r,1}(\mathbb{R}^\Dim) \to \dot{B}^{0,1}_{r,1}(\mathbb{R}^\Dim)$, so \eqref{Stolyarov_bound} is equivalent to the assertion that
\begin{equation}\label{Stolyarov_bound_equivalent}
  \|f\|_{\dot{B}^{-\alpha,1}_{r,1}(\mathbb{R}^\Dim)} \lesssim \int_0^\infty t^{\alpha/2-1} \|e^{t\Delta} f\|_{L^{r,1}(\mathbb{R}^\Dim)}\,\dd t
  .
\end{equation}
It therefore suffices to establish the reverse inequality, and by density it is sufficient to prove such a bound for smooth, compactly supported functions.
Let $f \in C^\infty_c(\mathbb{R}^n)$ and set $f_k = f\ast (\psi_{2^{k+1}} -\psi_{2^k})$.
Using $f = \sum_k f_k$ we have
\begin{align*}
  \int_0^\infty t^{\alpha/2-1} \| e^{t\Delta}f\|_{L^{r,1}(\mathbb{R}^\Dim)} \,\dd t &= \sum_{l\in \mathbb{Z}} \int_{2^{-2l}}^{2^{-2l+2}} t^{\alpha/2-1}  \| e^{t\Delta}f\|_{L^{r,1}(\mathbb{R}^\Dim)} \,\dd t\\
  &\leq  \sum_{l\in \mathbb{Z}} (2^{-2l+2})^{\alpha/2} \int_{2^{-2l}}^{2^{-2l+2}}   \| e^{t\Delta} f\|_{L^{r,1}(\mathbb{R}^\Dim)} \,\frac{\dd t}{t}\\
  &\leq 2^\alpha \sum_{k \in \mathbb{Z}} \sum_{l\in \mathbb{Z}} 2^{-\alpha l} \int_{2^{-2l}}^{2^{-2l+2}}   \| e^{t\Delta}f_k\|_{L^{r,1}(\mathbb{R}^\Dim)} \,\frac{\dd t}{t}\\
  &= 2^\alpha \left(\sum_{k \in \mathbb{Z}}   \sum_{l\leq k} \cdots + \sum_{k \in \mathbb{Z}}  \sum_{l> k} \cdots\right)\\
  &=: I + \mathit{II}
  .
\end{align*}
For $I$, the support condition $\operatorname*{supp} \widehat{f_k} \subset B(0,2^{k+2})\setminus  B(0,2^{k})$ allows us to use Bernstein's inequality
\begin{align*}
  \| e^{t\Delta}f_k\|_{L^{r,1}(\mathbb{R}^\Dim)} \leq 2^{-2k}  \| \Delta e^{t\Delta}f_k\|_{L^{r,1}(\mathbb{R}^\Dim)}
\end{align*}
and $\| \Delta e^{t\Delta}\|_{L^{r,1}(\mathbb{R}^\Dim) \to L^{r,1}(\mathbb{R}^\Dim)}\lesssim t^{-1}$ gives
\begin{align*}
  \int_{2^{-2l}}^{2^{-2l+2}}   \| e^{t\Delta}f_k\|_{L^{r,1}(\mathbb{R}^\Dim)} \,\frac{\dd t}{t} &\leq \int_{2^{-2l}}^{2^{-2l+2}}     2^{-2k} \| \Delta e^{t\Delta} f_k\|_{L^{r,1}(\mathbb{R}^\Dim)}  \,\frac{\dd t}{t} \\
  &\leq \int_{2^{-2l}}^{2^{-2l+2}}    2^{-2k}  \|   f_k\|_{L^{r,1}(\mathbb{R}^\Dim)}   \,\frac{\dd t}{t^2}\\
  &\lesssim  2^{2l-2k} \|  f_k\|_{L^{r,1}(\mathbb{R}^\Dim)}
\end{align*}
and therefore
\begin{align*}
  I &\leq 2^{\alpha+1} \sum_{k \in \mathbb{Z}} 2^{-2k} \|  f_k\|_{L^{r,1}(\mathbb{R}^\Dim)} \sum_{l\leq k}  2^{(2-\alpha) l} \\
  &\lesssim \sum_{k \in \mathbb{Z}} 2^{-\alpha k} \|  f_k\|_{L^{r,1}(\mathbb{R}^\Dim)}
  .
\end{align*}

For $\mathit{II}$,  we use that the heat kernel is a contraction on $L^{r,1}$, $\| e^{t\Delta}f_k\|_{L^{r,1}(\mathbb{R}^\Dim)}\leq \| f_k\|_{L^{r,1}(\mathbb{R}^\Dim)}$, to obtain the bound
\begin{align*}
  \mathit{II} &\lesssim  \sum_{k \in \mathbb{Z}}    \|   f_k\|_{L^{r,1}(\mathbb{R}^\Dim)
  }\sum_{l> k}  2^{-\alpha l} \\
  &\approx \sum_{k \in \mathbb{Z}}  2^{-\alpha k}  \|   f_k\|_{L^{r,1}(\mathbb{R}^\Dim)
  }
  .
  \qedhere
\end{align*}
\end{proof}

\Cref{other_direction} shows that there is an equivalent norm defined without reference to the Fourier transform, and therefore suggests a sensible definition of Besov or Besov-Lorentz spaces for $\beta<0$ in a setting where there is a heat propagator.
For $\beta\geq 0$, we can deduce a similar equivalent norm by appealing to \Cref{Besov_homeomorphism_Euclidean}.

\begin{theorem}\label{other_direction_extended}
  Let $\beta\geq 0$ and set $i=\floor{\beta/2}+1$.
  For $f \in \dot{B}^{\beta,1}_{r,1}(\mathbb{R}^\Dim)$, $1\leq r<\infty$,
  \begin{align*}
    \| f\|_{\dot{B}^{\beta,1}_{r,1}(\mathbb{R}^\Dim)} \asymp \int_0^\infty t^{i-\beta/2-1} \|\Delta^i e^{t\Delta} f\|_{L^{r,1}(\mathbb{R}^\Dim)}\,\dd t
    .
  \end{align*}
\end{theorem}
\begin{proof}
  Replacing $f$ by $\Delta f$ and applying \Cref{other_direction} and \Cref{Besov_homeomorphism_Euclidean} gives
  \begin{align*}
    \int_0^\infty t^{\alpha/2-1} \|\Delta e^{t\Delta} f\|_{L^{r,1}(\mathbb{R}^\Dim)}\,\dd t
    \asymp \|\Delta f\|_{\dot{B}^{-\alpha,1}_{r,1}(\mathbb{R}^\Dim)}
    \asymp \| f\|_{\dot{B}^{2-\alpha,1}_{r,1}(\mathbb{R}^\Dim)}
  \end{align*}
  for $\alpha>0$.
  More generally, applying \Cref{Besov_homeomorphism_Euclidean} $i$ times and \Cref{other_direction} with $\alpha=2i-\beta>0$ and $f$ replaced by $\Delta^i f$ yields the result.
\end{proof}

In view of \Cref{other_direction,other_direction_extended}, we can generalize the definition of Besov-Lorentz to an arbitrary Riemannian manifold.

\begin{definition}\label{Besov_manifold_definition}
  Let $\manifold$ be an arbitrary $\Dim$-dimensional Riemannian manifold and let $k\in\set{0,\dotsc,\Dim}$.
  For $\beta\in\R$, let
  \begin{equation*}
    i(\beta) = \max\set{\floor{\beta/2}+1, 0}
    = \begin{cases}
      \floor{\beta/2}+1 &\text{if }\beta\geq 0
      ,\\
      0 &\text{if }\beta<0
      .
    \end{cases}
  \end{equation*}
  Then, for all $r,s$ for which the $L^{r,s}$ norm is defined, 
  \begin{equation*}
    \norm{F}_{\dot{B}^{\beta,1}_{r,s}(\Lambda^k)} = \int_0^\infty t^{i(\beta)-\beta/2-1} \|\Delta_k^{i(\beta)} e^{t\Delta_k} F\|_{L^{r,s}(\Lambda^k)}\,\dd t
  \end{equation*}
  defines a norm on $C_c^\infty\cap\mathcal{H}^\perp(\Lambda^k)$.
  Define the space $\dot{B}^{\beta,1}_{r,s}(\Lambda^k)$ to be the completion of $C_c^\infty\cap\mathcal{H}^\perp(\Lambda^k)$ with respect to this norm.
\end{definition}
Indeed this construction can be carried out for many choices of underlying norm, cf.\ \cite{ST}.
Henceforth we restrict our attention to the Lorentz spaces $L^{r,1}$.

To close the chain of reasoning that led to \Cref{Besov_manifold_definition}, it would be natural to wonder whether the Laplacian and Riesz potentials remain homeomorphisms under this definition.
We show that this holds in the setting of a compact Riemannian manifold.
\begin{theorem}\label{Besov_homeomorphisms}
  Let $\manifold$ be a compact $\Dim$-dimensional Riemannian manifold and let $k\in\set{0,\dotsc,\Dim}$.
  For each $\beta \in \mathbb{R}$, $\alpha>0$, and $1\leq r<\infty$, the mappings
  \begin{align*}
    -\Delta_k &\colon \dot{B}^{\beta,1}_{r,1}(\Lambda^k) \to \dot{B}^{\beta-2,1}_{r,1}(\Lambda^k)
    \\
    \mathcal{I}_{\alpha,k}&\colon \dot{B}^{\beta,1}_{r,1}(\Lambda^k) \to \dot{B}^{\beta+\alpha,1}_{r,1}(\Lambda^k)
  \end{align*}
  are homeomorphisms.
\end{theorem}

To prove \Cref{Besov_homeomorphisms}, we will use bounds on the operator norms of $e^{t\Delta_k}$ as an operator between various Lorentz spaces, which we obtain from $L^1$ and $L^\infty$ operator norms via the first part of \Cref{GrafakosLorentzOperatorNorms}.

\begin{lemma}\label{OperatorBoundsLorentz}
  For $1<r<\infty$ fixed and uniformly over $t>0$,
  \begin{align}
    \|e^{t\Delta_k} \|_{L^{r,1}(\Lambda^k)\to L^{r,1}(\Lambda^k)} &\lesssim 1 \label{propagator}
    \\
    \|\Delta_ke^{t\Delta_k}\|_{L^{r,1}(\Lambda^k)\to L^{r,1}(\Lambda^k)} &\lesssim \frac{1}{t} \label{propagator_prime}
    .
  \end{align}
\end{lemma}

\begin{proof}
  For $0<t\leq 1$, the operator $T=e^{t\Delta_k}$ is given by an integral kernel bounded in terms of the function $K_t$, see \eqref{HeatKernelBoundedByK} and \Cref{CurrentEvolutionAsForm,KernelBounds}, and it follows immediately that
  \begin{equation}
    \norm{e^{t\Delta_k}}_{L^1\to L^1} \lesssim \sup_q \norm{K_t(\cdot, q)}_{L^1}
    ,
    \quad
    \norm{e^{t\Delta_k}}_{L^\infty\to L^\infty} \lesssim \sup_p \norm{K_t(p,\cdot)}_{L^1}
  \end{equation}
  for $0<t\leq 1$.
  Recalling \eqref{KdotqL1Bound}, we therefore have
  \begin{equation}\label{PropagatorOperatorBounds}
    \begin{aligned}
      \norm{e^{t\Delta_k}}_{L^1(\Lambda^k)\to L^1(\Lambda^k)} \lesssim 1
      ,
      \quad
      \norm{e^{t\Delta_k}}_{L^\infty(\Lambda^k)\to L^\infty(\Lambda^k)} \lesssim 1
    \end{aligned}
  \end{equation}
  for all $0<t\leq 1$.
  
  For $t>1$, the composition $e^{t\Delta_k}\vert_{\mathcal{H}^\perp(\Lambda^k)}$ has exponential decay by \Cref{exp_decay}, while $e^{t\Delta_k}\vert_{\mathcal{H}(\Lambda^k)}$ acts as the identity.
  Thus both of these operators satisfy the bounds in \eqref{PropagatorOperatorBounds}.
  Moreover, since $\mathcal{H}(\Lambda^k)$ is a finite-dimensional vector subspace of $L^\infty(\Lambda^k)\subset L^1(\Lambda^k)$, the projections $P_{\mathcal{H}(\Lambda^k)}$ and $P_{\mathcal{H}^\perp(\Lambda^k)}=\mathrm{identity}-P_{\mathcal{H}(\Lambda^k)}$ are bounded as operators on both $L^1$ and $L^\infty$.
  Thus \eqref{PropagatorOperatorBounds} holds for all $t>1$ also.
  By \Cref{GrafakosLorentzOperatorNorms}, \eqref{PropagatorOperatorBounds} implies \eqref{propagator}.
  
  A similar argument applies for $\Delta_k e^{t\Delta_k}$.
  By analogy with \Cref{DerivOfGaussianTimesSmooth} and its proof, this operator is given by an integral kernel that is bounded up to constants by
  \begin{equation*}
    K^{(2)}_t(p,q) = \frac{(d_\manifold(p,q)+t)^2}{t^{\Dim/2 + 2}} \exp(-d_\manifold(p,q)^2/(4t)) + 1
  \end{equation*}
  for $0<t\leq 1$, and it can be readily verified that
  \begin{align*}
    \sup_p \bignorm{K^{(2)}_t(p,\cdot)}_{L^1} \lesssim 1/t
    ,
    \quad
    \sup_q \bignorm{K^{(2)}_t(\cdot,q)}_{L^1} \lesssim 1/t
    .
  \end{align*}
  For $t>1$, we can again use the exponential decay from \Cref{exp_decay} on $\mathcal{H}^\perp(\Lambda^k)$, noting that $\Delta_k e^{t\Delta_k}$ vanishes on $\mathcal{H}(\Lambda^k)$.
  We conclude that $\norm{\Delta_k e^{t\Delta_k}_{L^1\to L^1}}$ and $\norm{\Delta_k e^{t\Delta_k}}_{L^\infty\to L^1\infty}$ are both of order $1/t$, and \eqref{propagator_prime} follows by \Cref{GrafakosLorentzOperatorNorms}.
\end{proof}

\begin{lemma}\label{Besov_Laplacian_bound}
  Let $\beta \in \mathbb{R}$ and $1\leq r<\infty$.
  There exists a constant $C>0$ such that
  \begin{align*}
    \|\Delta_k F\|_{\dot{B}^{\beta-2,1}_{r,1}(\Lambda^k)} \leq C \| F\|_{\dot{B}^{\beta,1}_{r,1}(\Lambda^k)}
  \end{align*}
  for all $F \in \dot{B}^{\beta,1}_{r,1}(\Lambda^k)$.
\end{lemma}

\begin{proof}
  For $\beta\geq 0$, the inequality is trivial and indeed the two norms are equal.
  To see this, note that $i(\beta)=i(\beta-2)+1$ for all $\beta\geq 0$, so that $i(\beta-2)-(\beta-2)/2 = i(\beta)-\beta/2$ and
  \begin{align*}
    \norm{\Delta_k F}_{\dot{B}^{\beta-2,1}_{r,1}(\Lambda^k)} &= \int_0^\infty t^{i(\beta-2)-(\beta-2)/2-1} \norm{ \Delta_k^{i(\beta-2)} e^{t\Delta_k} \Delta_k F}_{L^{r,1}(\Lambda^k)} \dd t \\
    &= \int_0^\infty t^{i(\beta)-\beta/2-1} \norm{ \Delta_k^{i(\beta-2)+1} e^{t\Delta_k} F}_{L^{r,1}(\Lambda^k)} \dd t \\
    &=   \|F\|_{\dot{B}^{\beta,1}_{r,1}(\Lambda^k)}
    .
  \end{align*}
  For $\beta<0$, use \Cref{OperatorBoundsLorentz} and make the substitution $t=2\tilde{t}$ to obtain
  \begin{align}
    \norm{\Delta_k F}_{\dot{B}^{\beta-2,1}_{r,1}(\Lambda^k)} &= \int_0^\infty t^{-(\beta-2)/2-1} \norm{ \Delta_k e^{(t/2)\Delta_k} e^{(t/2)\Delta_k} F }_{L^{r,1}(\Lambda^k)} \dd t 
    \notag\\
    \label{AbsorbDelta}
    &\lesssim  \int_0^\infty (2\tilde{t})^{-\beta/2} \tilde{t}^{-1} \norm{ e^{\tilde{t}\Delta_k} F }_{L^{r,1}(\Lambda^k)} (2 \, \dd \tilde{t})\\
    &\lesssim\| F\|_{\dot{B}^{\beta-2,1}_{r,1}(\Lambda^k)}
    .
    \qedhere
  \end{align}
\end{proof}

\begin{lemma}\label{Besov_potential_bound}
  Fix $\beta \in \mathbb{R}$, $\alpha>0$, and $1\leq r<\infty$.
  Then there exists a constant $C=C(\alpha,\beta,r,\manifold)<\infty$ such that
  \begin{align*}
    \|\mathcal{I}_{\alpha,k}F\|_{\dot{B}^{\beta+\alpha,1}_{r,1}(\Lambda^k)} \leq C \| F\|_{\dot{B}^{\beta,1}_{r,1}(\Lambda^k)}
  \end{align*}
  for all $F \in \dot{B}^{\beta,1}_{r,1}(\Lambda^k)$.
\end{lemma}

\begin{proof}
  Apply Minkowski's inequality to the integral representation \eqref{RieszPotential} and make the substitution $\tau=s+t,y=t/(s+t)$:
  \begin{align*}
    &\norm{ \mathcal{I}_{\alpha,k}F }_{\dot{B}^{\beta+\alpha,1}_{r,1}(\Lambda^k)} \\&\quad
    \leq \iint_{(0,\infty)^2} s^{\alpha/2 - 1} t^{i(\beta+\alpha)-(\beta+\alpha)/2-1} \norm{\Delta_k^{i(\beta+\alpha)} e^{t\Delta_k} e^{s\Delta_k} F }_{L^{r,1}(\Lambda^k)} \frac{\,\dd s\,\dd t}{\Gamma(\alpha/2)}
    \\&\quad
    = \int_0^\infty \int_0^1 (\tau(1-y))^{\alpha/2-1} (\tau y)^{i(\beta+\alpha)-(\beta+\alpha)/2-1} \norm{\Delta_k^{i(\beta+\alpha)} e^{\tau\Delta_k} F }_{L^{r,1}(\Lambda^k)} \frac{\tau\,\dd y\,\dd \tau}{\Gamma(\alpha/2)}
    \\&\quad
    = \int_0^\infty \tau^{i(\beta+\alpha)-\beta/2-1} \norm{\Delta_k^{i(\beta+\alpha)} e^{\tau\Delta_k} F }_{L^{r,1}(\Lambda^k)} \dd \tau
    \\&\qquad\cdot
    \int_0^1 (1-y)^{\alpha/2-1} y^{i(\beta+\alpha)-(\beta+\alpha)/2-1}  \frac{\dd y}{\Gamma(\alpha/2)}
    .
  \end{align*}
  Note that $i(\gamma)\geq \floor{\gamma/2}+1 > \gamma/2$ for all $\gamma\in\R$, so the last $y$-integral converges.

  Since $\alpha>0$, we may write $i(\beta+\alpha)=i(\beta)+m$ for some $m\in\set{0,1,\dotsc}$.
  Similar to \eqref{AbsorbDelta}, substitute $\tau=(m+1)\tilde{\tau}$ and apply \eqref{propagator_prime} $m$ times to conclude
  \begin{align*}
    \norm{ \mathcal{I}_{\alpha,k}F }_{\dot{B}^{\beta+\alpha,1}_{r,1}(\Lambda^k)}
    &\lesssim \int_0^\infty \tau^{i(\beta)+m-\beta/2} \norm{\Delta_k^{i(\beta)+m} e^{\tau\Delta_k} F }_{L^{r,1}(\Lambda^k)} \frac{\dd \tau}{\tau}
    \\&
    \lesssim \int_0^\infty [(m+1)\tilde{\tau}]^{i(\beta)+m-\beta/2} \tilde{\tau}^{-m}\norm{\Delta_k^{i(\beta)} e^{\tilde{\tau}\Delta_k} F }_{L^{r,1}(\Lambda^k)} \frac{\dd \tilde{\tau}}{\tilde{\tau}}
    \\&
    \lesssim \norm{F}_{\dot{B}^{\beta,1}_{r,1}(\Lambda^k)}
    .
    \qedhere
  \end{align*}
\end{proof}

\begin{proof}[Proof of \Cref{Besov_homeomorphisms}]
  The mappings are linear so it suffices to show two-sided inequalities up to constants between the relevant norms.
  By density, it suffices to consider $F \in C^\infty\intersect \mathcal{H}^\perp(\Lambda^k)$.
  
  For the Laplacian, the inequality $\norm{\Delta_k F}_{\dot{B}^{\beta-2,1}_{r,1}(\Lambda^k)} \lesssim \norm{F}_{\dot{B}^{\beta,1}_{r,1}(\Lambda^k)}$ is given by \Cref{Besov_Laplacian_bound}.
  The reverse inequality follows from the identity $F=\mathcal{I}_{2,k} (-\Delta_k F)$, see \eqref{left_inverse}, and \Cref{Besov_potential_bound}:
  \begin{align*}
    \| F\|_{\dot{B}^{\beta,1}_{r,1}(\Lambda^k)} &= \| \mathcal{I}_{2,k} (-\Delta_k F)\|_{\dot{B}^{\beta,1}_{r,1}(\Lambda^k)}\\
    &\leq C \|\Delta_k F\|_{\dot{B}^{\beta-2,1}_{r,1}(\Lambda^k)}.
  \end{align*}
  
  Similarly for the Riesz potential, one inequality appears in \Cref{Besov_potential_bound}.
  To show the reverse direction, we begin by showing the commutativity relation
  \begin{equation}\label{DeltaICommuteOnHperp}
    -\Delta_k \mathcal{I}_{\alpha,k} \omega = \mathcal{I}_{\alpha,k}(-\Delta_k)\omega
    \quad\forall\omega\in C^\infty\intersect\mathcal{H}^\perp(\Lambda^k), \,\forall\alpha>0
    .
  \end{equation}
  To see this, note that \Cref{exp_decay} and \eqref{propagator_prime} imply that the Laplacian can be passed through the convergent integral \eqref{RieszPotential}, so that both sides of \eqref{DeltaICommuteOnHperp} reduce to $\int_0^\infty t^{\alpha/2-1} (-\Delta_k)e^{t\Delta_k}\omega \,\dd t$ or equivalently $\int_0^\infty t^{\alpha/2-1} e^{t\Delta_k}(-\Delta_k)\omega \,\dd t$.
  
  Now set $i=\lfloor \alpha/2 \rfloor+1$, so that $2i>\alpha$.
  By \eqref{DeltaICommuteOnHperp} and \eqref{semigroup},
  \begin{equation}
    \mathcal{I}_{2i-\alpha,k} (-\Delta_k)^i \mathcal{I}_{\alpha,k} F = \mathcal{I}_{2i-\alpha,k} \mathcal{I}_{\alpha,k} (-\Delta_k)^i F = \mathcal{I}_{2i,k}(-\Delta_k)^i F = F
    ,
  \end{equation}
  so that the combination of \Cref{Besov_Laplacian_bound} and \Cref{Besov_potential_bound} imply
  \begin{align*}
    \| F\|_{\dot{B}^{\beta,1}_{r,1}(\Lambda^k)} &=  \| \mathcal{I}_{2i-\alpha,k} (-\Delta)^i \mathcal{I}_{\alpha,k}F\|_{\dot{B}^{\beta,1}_{r,1}(\Lambda^k)}\\
    &\lesssim  \|  (-\Delta)^i \mathcal{I}_{\alpha,k}F\|_{\dot{B}^{\beta - 2i +\alpha ,1}_{r,1}(\Lambda^k)}\\
    &\lesssim  \|\mathcal{I}_{\alpha,k}F\|_{\dot{B}^{\beta+\alpha,1}_{r,1}(\Lambda^k)}
    .
    \qedhere
  \end{align*}
\end{proof}

Our definition of Besov-Lorentz spaces on a manifold also gives the usual inclusions between various exponents of differentiability:
\begin{theorem}\label{Besov-nested}
For $0 \leq \alpha \leq \gamma < \Dim$,
\begin{equation}\label{Besov-nested-inclusions}
  \dot{B}^{0,1}_{1,1}(\Lambda^k) \hookrightarrow  \dot{B}^{-\alpha,1}_{\Dim/(\Dim-\alpha),1}(\Lambda^k) \hookrightarrow \dot{B}^{-\gamma,1}_{\Dim/(\Dim-\gamma),1}(\Lambda^k).
\end{equation}
\end{theorem}

\begin{proof}
  We begin with a Lorentz space operator norm bound for $e^{t\Delta_k}$, similar to \Cref{OperatorBoundsLorentz} but now between two different Lorentz spaces.
  Recalling \eqref{KLinftyBound}, 
  \begin{equation}\label{PropagatorL1LinftyOperatorBoundSmallt}
    \norm{e^{t\Delta_k}}_{L^1(\Lambda^k)\to L^\infty(\Lambda^k)} \lesssim \sup_{p,q} K_t(p,q) 
    \lesssim t^{-\Dim/2}
    \quad\text{uniformly over $0<t\leq 1$.}
  \end{equation}
  Because of compactness, this bound fails for $t>1$, and indeed $e^{t\Delta_k}$ behaves as the projection $P_{\mathcal{H}(\Lambda^k)}$ for large $t$.
  However, since $P_{\mathcal{H}^\perp(\Lambda^k)}$ is bounded as an operator on $L^1$ and $L^\infty$, the composition $e^{t\Delta_k}P_{\mathcal{H}^\perp(\Lambda^k)}$ obeys the same bound as \eqref{PropagatorL1LinftyOperatorBoundSmallt} while also decaying exponentially for $t>1$ by \Cref{exp_decay}.
  We conclude that
  \begin{equation}\label{ProjectedPropagatorL1Linfty}
    \norm{ e^{t\Delta_k}P_{\mathcal{H}^\perp(\Lambda^k)} }_{L^1(\Lambda^k)\to L^\infty(\Lambda^k)} \lesssim t^{-\Dim/2}
  \end{equation}
  uniformly over all $t>0$.
  Likewise, the bounds \eqref{PropagatorOperatorBounds} apply to $e^{t\Delta_k}P_{\mathcal{H}^\perp(\Lambda^k)}$, so \Cref{GrafakosLorentzOperatorNorms} yields
  \begin{equation}\label{ProjectedPropagatorLorentzLorentz}
    \norm{ e^{t\Delta_k}P_{\mathcal{H}^\perp(\Lambda^k)} }_{L^{r,1}\to L^{p,1}} \lesssim t^{-\Dim(1/r-1/p)/2}
  \end{equation}
  for $1<r\leq p<\infty$ fixed.
  
  In the parametrization of \eqref{Besov-nested-inclusions}, the operator bound \eqref{ProjectedPropagatorLorentzLorentz} gives
  \begin{equation}\label{OperatorNormLorentzLorentz}
    \|e^{t\Delta_k}F\|_{L^{\Dim/(\Dim-\gamma),1}(\Lambda^k)} \lesssim t^{-(\gamma-\alpha)/2}  \|e^{(t/2)\Delta_k}F\|_{L^{\Dim/(\Dim-\alpha),1}(\Lambda^k)}
  \end{equation}
  for $0\leq\alpha\leq\gamma<\Dim$ fixed and uniformly over $F\in \dot{B}^{-\alpha,1}_{\Dim/(\Dim-\alpha),1}(\Lambda^k)$ and $t>0$.
  The second inclusion in \eqref{Besov-nested-inclusions} follows by multiplying by $t^{\gamma/2-1}$ and integrating.
  
  For the first inclusion, we will apply \eqref{OperatorNormLorentzLorentz} with $\alpha=0$.
  The norm for $\dot{B}^{0,1}_{1,1}(\Lambda^k)$ requires an extra $\Delta_k$, and \Cref{Besov_homeomorphisms} provides one:
  \begin{align*}
    \norm{F}_{\dot{B}^{-\gamma,1}_{\Dim/(\Dim-\gamma),1}(\Lambda^k)} 
    &\lesssim \norm{\Delta_k F}_{\dot{B}^{-\gamma-2,1}_{\Dim/(\Dim-\gamma),1}(\Lambda^k)}
    \\&
    = \int_0^\infty t^{\gamma/2+1-1} \norm{\Delta_k e^{t\Delta_k} F}_{L^{\Dim/(\Dim-\gamma),1}(\Lambda^k)} \,\dd t
    \\&
    \lesssim \int_0^\infty \norm{\Delta_k e^{t\Delta_k} F}_{L^{1,1}(\Lambda^k)} \,\dd t 
    &&\text{by \eqref{OperatorNormLorentzLorentz}}
    \\&
    = \norm{F}_{\dot{B}^{0,1}_{1,1}(\Lambda^k)}
    .
    &&\qedhere
  \end{align*}
\end{proof}


\end{document}